\documentclass{article}
\usepackage{amsmath,amsthm,amssymb}
\usepackage[usenames,dvipsnames]{xcolor}
\usepackage{enumerate}
\usepackage{graphicx}
\usepackage{cite}
\usepackage{comment}
\usepackage{oands}
\usepackage{tikz}
\usepackage{changepage}
\usepackage{bbm}
\usepackage{mathtools}
\usepackage[margin=1in]{geometry}
\usepackage[pagewise,mathlines]{lineno}
\usepackage{appendix}
\usepackage{microtype}
\usepackage{tabularx}
\usepackage[pdftitle={Metric gluing of Brownian and sqrt(8/3)-Liouville quantum gravity surfaces},
  pdfauthor={Ewain Gwynne and Jason Miller},
colorlinks=true,linkcolor=NavyBlue,urlcolor=RoyalBlue,citecolor=PineGreen,bookmarks=true,bookmarksopen=true,bookmarksopenlevel=2,unicode=true,linktocpage]{hyperref}

\theoremstyle{plain}
\newtheorem{thm}{Theorem}[section]

\newtheorem{lem}[thm]{Lemma}

\newtheorem{notation}[thm]{Notation}

\def\@rst #1 #2other{#1}
\newcommand\MR[1]{\relax\ifhmode\unskip\spacefactor3000 \space\fi
  \MRhref{\expandafter\@rst #1 other}{#1}}
\newcommand{\MRhref}[2]{\href{http://www.ams.org/mathscinet-getitem?mr=#1}{MR#2}}

\theoremstyle{definition}
\newtheorem{defn}[thm]{Definition}
\newtheorem{remark}[thm]{Remark}

\numberwithin{equation}{section}

\newcommand{\dsb}{\begin{adjustwidth}{2.5em}{0pt}
\begin{footnotesize}}
\newcommand{\dse}{\end{footnotesize}
\end{adjustwidth}}

\newcommand{\ssb}{\begin{adjustwidth}{2.5em}{0pt}}
\newcommand{\sse}{\end{adjustwidth}}

\newcommand{\aryb}{\begin{eqnarray*}}
\newcommand{\arye}{\end{eqnarray*}}
\def\alb#1\ale{\begin{align*}#1\end{align*}}
\newcommand{\eqb}{\begin{equation}}
\newcommand{\eqe}{\end{equation}}
\newcommand{\eqbn}{\begin{equation*}}
\newcommand{\eqen}{\end{equation*}}

\newcommand{\BB}{\mathbbm}
\newcommand{\ol}{\overline}
\newcommand{\ul}{\underline}
\newcommand{\op}{\operatorname}

\newcommand{\im}{\operatorname{Im}}

\newcommand{\frk}{\mathfrak}
\newcommand{\eqD}{\overset{d}{=}}
\newcommand{\ep}{\epsilon}
\newcommand{\rta}{\rightarrow}

\newcommand{\wt}{\widetilde}
\newcommand{\wh}{\widehat} 
\newcommand{\mcl}{\mathcal}

\newcommand{\bdy}{\partial}

\let\originalleft\left
\let\originalright\right
\renewcommand{\left}{\mathopen{}\mathclose\bgroup\originalleft}
\renewcommand{\right}{\aftergroup\egroup\originalright}

\title{Metric gluing of Brownian and\\$\sqrt{8/3}$-Liouville quantum gravity surfaces}
\date{  }
\author{
\begin{tabular}{c} Ewain Gwynne\\[-5pt]\small MIT \end{tabular}
\begin{tabular}{c} Jason Miller\\[-5pt]\small Cambridge \end{tabular}
}

\begin{document}

\maketitle

\begin{abstract}
In a recent series of works, Miller and Sheffield constructed a metric on $\sqrt{8/3}$-Liouville quantum gravity (LQG) under which $\sqrt{8/3}$-LQG surfaces (e.g., the LQG sphere, wedge, cone, and disk) are isometric to their Brownian surface counterparts (e.g., the Brownian map, half-plane, plane, and disk).  

We identify the metric gluings of certain collections of independent $\sqrt{8/3}$-LQG surfaces with boundaries identified together according to LQG length along their boundaries.  
Our results imply in particular that the metric gluing of two independent instances of the Brownian half-plane along their positive boundaries is isometric to a certain LQG wedge decorated by an independent chordal SLE$_{8/3}$ curve.  If one identifies the two sides of the boundary of a single Brownian half-plane, one obtains a certain  LQG cone decorated by an independent whole-plane SLE$_{8/3}$. If one identifies the entire boundaries of two Brownian half-planes, one obtains a different LQG cone and the interface between them is a two-sided variant of whole-plane SLE$_{8/3}$.

Combined with another work of the authors, the present work identifies the scaling limit of self-avoiding walk on random quadrangulations with SLE$_{8/3}$ on $\sqrt{8/3}$-LQG.
\end{abstract}

\tableofcontents

\section{Introduction}
\label{sec-intro}

\subsection{Overview}
\label{sec-overview}

For $\gamma \in (0,2)$, a \emph{Liouville quantum gravity (LQG) surface} is (formally) a random Riemann surface parameterized by a domain $D\subset \BB C$ whose Riemannian metric tensor is $e^{\gamma h(z)} \, dx\otimes dy$, where $h$ is some variant of the Gaussian free field (GFF) on $D$ and $dx\otimes dy$ denotes the Euclidean metric tensor. Liouville quantum gravity surfaces are conjectured to arise as the scaling limits of various random planar map models; the case $\gamma =\sqrt{8/3}$ corresponds to uniformly random planar maps, and other values of $\gamma$ are obtained by weighting by the partition function of an appropriate statistical mechanics model (see~\cite{shef-burger,kmsw-bipolar,gkmw-burger} for examples of such statistical mechanics models).  This has so far been shown to be the case for $\gamma=\sqrt{8/3}$ with respect to the Gromov-Hausdorff topology in the works \cite{legall-uniqueness,miermont-brownian-map} and \cite{tbm-characterization,lqg-tbm1,sphere-constructions,lqg-tbm2,lqg-tbm3} and for all $\gamma \in (0,2)$ in the so-called peanosphere sense in \cite{shef-burger,kmsw-bipolar,gkmw-burger} and \cite{wedges}.

Since $h$ is only a distribution (i.e., generalized function) and does not have well-defined pointwise values, this object does not make rigorous sense. However, it was shown by Duplantier and Sheffield~\cite{shef-kpz} that one can rigorously define the volume measure associated with an LQG surface. More specifically, there is a measure~$\mu_h$ on~$D$, called the \emph{$\gamma$-LQG measure}, which is the a.s.\ limit of regularized versions of $e^{\gamma h(z)} \, dz$, where $dz$ is the Euclidean volume form. One can similarly define the \emph{$\gamma$-LQG length measure} $\nu_h$ on certain curves in $D$, including $\bdy D$ and Schramm's \cite{schramm0} SLE$_\kappa$-type curves for $\kappa  = \gamma^2$~\cite{shef-zipper}. 

\begin{comment}
The LQG measures $\mu_h$ and $\nu_h$ are conformally covariant in the following sense.  If $D, \wt{D} \subset \BB C$, $\phi : \wt D \rta D$ is a conformal map, and we set
\eqb \label{eqn-lqg-coord0}
\wt h = h\circ \phi + Q\log |\phi'| ,\quad \op{where}\quad Q = \frac{2}{\gamma} + \frac{\gamma}{2} 
\eqe
then $\mu_h$ is the pushforward of $\mu_{\wt h}$ and $\nu_h$ is the pushforward of $\nu_{\wt h}$ under $\phi$~\cite[Proposition~2.1]{shef-kpz}. 
Hence it is natural to view pairs $(D , h)$ and $(\wt D , \wt h)$ which differ by a transformation as in~\eqref{eqn-lqg-coord0} as different parameterizations of the same LQG surface.  That is, a \emph{$\gamma$-LQG surface} is formally defined to be an equivalence class of pairs $(D,h)$ consisting of a domain $D$ and a distribution $h$ on $D$, with two such pairs $(D,h)$ and $(\wt D , \wt h)$ declared to be equivalent if they differ by a conformal map as in~\eqref{eqn-lqg-coord}.
\end{comment} 

In the recent works~\cite{tbm-characterization,lqg-tbm1,sphere-constructions,lqg-tbm2,lqg-tbm3}, Miller and Sheffield showed that in the special case when $\gamma = \sqrt{8/3}$, an LQG surface is equipped with a natural metric (distance function) $\frk d_h$, so can be interpreted as a random metric space. 
We will review the construction of $\frk d_h$ in Section~\ref{sec-lqg-metric}.

In this paper we will be interested in several particular types of $\sqrt{8/3}$-LQG surfaces which are defined in~\cite{wedges}. These include quantum spheres, which are finite-volume LQG surfaces (i.e., the total mass of the $\sqrt{8/3}$-LQG measure $\mu_h$ is finite) parameterized by the Riemann sphere $\BB C \cup\{\infty\}$; 
quantum disks, which are finite-volume LQG surfaces with boundary; 
quantum cones, which are infinite-volume LQG surfaces homeomorphic to $\BB C$; and
quantum wedges, which are infinite-volume LQG surfaces with infinite boundary. 
We will review the definitions of these particular types of quantum surfaces in Section~\ref{sec-lqg-surface} below. 

The \emph{Brownian map}~\cite{legall-uniqueness,miermont-brownian-map} is a continuum random metric space which arises as the scaling limit of uniform random planar maps, and which is constructed using a continuum analog of the Schaeffer bijection~\cite{schaeffer-bijection}. We refer to the surveys~\cite{legall-sphere-survey,miermont-survey} for more details about this object. 
One can also define Brownian surfaces with other topologies, such as the Brownian half-plane, which is the scaling limit of the uniform infinite half-planar quadrangulations \cite{gwynne-miller-uihpq,bmr-uihpq};  the Brownian plane~\cite{curien-legall-plane}, which is the scaling limit of the uniform infinite planar quadrangulation; and the Brownian disk~\cite{bet-mier-disk}, which is the scaling limit of finite quadrangulations with boundary. 
It is shown in~\cite[Corollary~1.5]{lqg-tbm2} (or~\cite[Proposition~1.10]{gwynne-miller-uihpq} in the half-plane case) that each of these Brownian surfaces is isometric to a certain special $\sqrt{8/3}$-LQG surface, i.e.\ one can couple the two random metric spaces in such a way that there a.s.\ exists an isometry between them. It is shown in \cite{lqg-tbm3} that in fact the $\sqrt{8/3}$-LQG structure is a measurable function of the Brownian surface structure and it follows from the construction in \cite{lqg-tbm1,lqg-tbm2} that the converse holds). In particular,
\begin{itemize}
\item The Brownian map is isometric to the quantum sphere;
\item The Brownian disk is isometric to the quantum disk;
\item The Brownian plane is isometric to the $\sqrt{8/3}$- (weight-$4/3$) quantum cone;
\item The Brownian half-plane is isometric to the $\sqrt{8/3}$- (weight-$2$) quantum wedge.
\end{itemize}
Furthermore, the isometries are such that the $\sqrt{8/3}$-LQG area measure is pushed forward to the natural volume measure on the corresponding Brownian surface and (in the case of the disk or half-plane) the $\sqrt{8/3}$-LQG boundary length measure is pushed forward to the natural boundary length measure on the Brownian disk or half-plane. That is, the spaces are equivalent as metric measure spaces.  Hence the construction of the $\sqrt{8/3}$-LQG metric in~\cite{tbm-characterization,lqg-tbm1,sphere-constructions,lqg-tbm2,lqg-tbm3} can be interpreted as:
\begin{itemize}
\item Endowing the Brownian map, half-plane, plane, and disk with a canonical conformal structure and
\item Constructing many additional random metric spaces (corresponding to other LQG surfaces) which locally look like Brownian surfaces.
\end{itemize}

The goal of this paper is to study metric space quotients (also known as metric gluings) of $\sqrt{8/3}$-LQG surfaces, equivalently Brownian surfaces, glued along their boundaries according to quantum length. The results described in~\cite[Section~1.2]{wedges} (see also~\cite{shef-zipper}) show that one can \emph{conformally} glue various types of quantum surfaces along their boundaries according to quantum length to obtain different quantum surfaces. In each case, the interface between the glued surfaces is an SLE$_\kappa$-type curve with $\kappa \in \{\gamma^2 , 16/\gamma^2\}$ (so $\kappa \in \{8/3,6\}$ when $\gamma = \sqrt{8/3}$). We will show that in the case when the interface is a simple curve (so $\kappa = 8/3$), this conformal gluing is the same as the metric gluing. In other words, the $\sqrt{8/3}$-LQG metric on the glued surface is the metric quotient of the $\sqrt{8/3}$-LQG metrics on the surfaces being glued. See Sections~\ref{sec-bhp-gluing} and~\ref{sec-results} for precise statements. We emphasize that we will be considering metric gluings along rough, fractal curves and in general such gluings are not well-behaved.  See Section~\ref{sec-gluing-remarks} for a discussion of the difficulties involved in proving metric gluing statements of the sort we consider here.

\begin{figure}[ht!!]
\begin{center}
\includegraphics[scale=0.7]{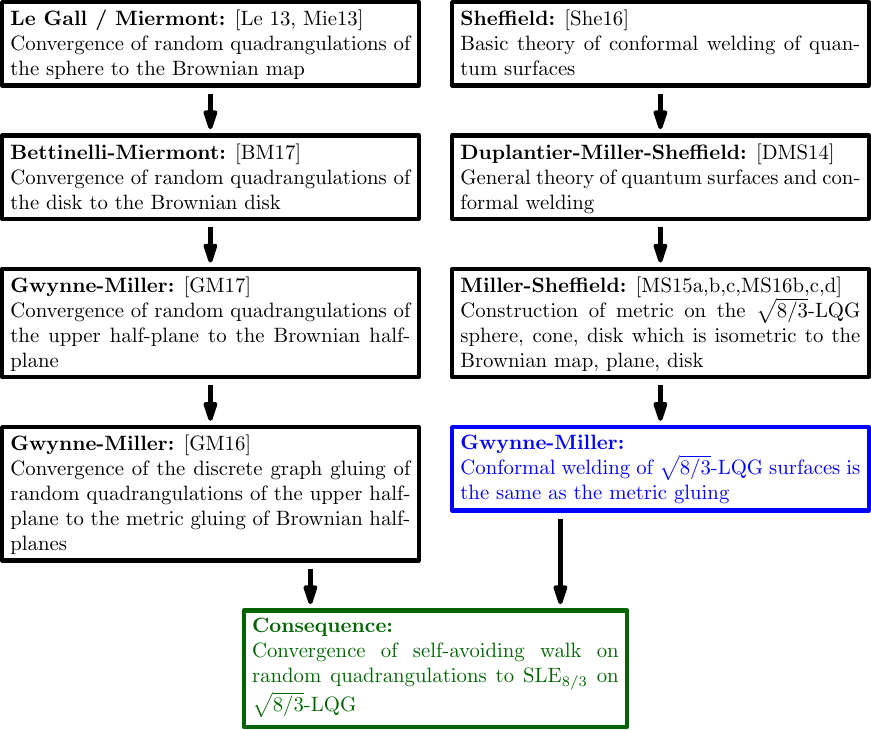}	
\end{center}
\vspace{-0.02\textheight}
\caption[Relationships between papers related to Chapter~\ref{chap-gluing}]{\label{fig-chart} A chart of the different components which serve as input into the proof that self-avoiding walk on random quadrangulations converges to SLE$_{8/3}$ on $\sqrt{8/3}$-LQG.  The present article corresponds to the blue box and implies that the embedding into $\mathbb C$ via so-called QLE$(8/3,0)$ of the metric gluing of a pair of Brownian half-planes to each other along their boundary or a single Brownian half-plane to itself along its boundary is described by $\sqrt{8/3}$-LQG where the interface is an SLE$_{8/3}$-type path.}
\end{figure}

In light of the aforementioned relationship between $\sqrt{8/3}$-LQG surfaces and Brownian surfaces, our results translate into results for Brownian surfaces. In particular, our results imply that each of the Brownian map and the Brownian plane are metric space quotients of countably many independent Brownian disks glued together along their boundaries (Theorems~\ref{thm-peanosphere-gluing} and~\ref{thm-peanosphere-gluing-finite}); and that when one metrically glues together two independent Brownian surfaces, the resulting surface locally looks like a Brownian surface (even at points along the gluing interface). To our knowledge, it is not known how to derive either of these facts directly from the Brownian surface literature. Hence this work can be viewed as an application of $\sqrt{8/3}$-LQG  to the theory of Brownian surfaces. However, our proofs will also make use of certain facts about Brownian surfaces which are not obvious directly from the $\sqrt{8/3}$-LQG perspective (in particular the estimates for the Brownian disk contained in Section~\ref{sec-disk-estimate}).

The results of this paper also have applications to scaling limits of random quadrangulations decorated by a self-avoiding walk.  Indeed, it is known that gluing together two uniformly random quadrangulations with simple boundary along their boundaries according to boundary length (or gluing two boundary arcs of a single uniformly random quadrangulation with simple boundary to each other according to boundary length) produces a uniformly random pair consisting of a quadrangulation decorated by a self-avoiding walk. See~\cite[Section~8.2]{bet-disk-tight} (which builds on~\cite{bbg-recursive-approach,bg-simple-quad}) for the case of finite quadrangulations with simple boundary and~\cite[Part III]{caraceni-thesis},~\cite{caraceni-curien-saw} for the case of the uniform infinite half-planar quadrangulation with simple boundary (UIHPQ$_{\op{S}}$).  The results of the present work combined with~\cite{gwynne-miller-saw,gwynne-miller-uihpq} imply that the scaling limit of a uniform infinite planar or half-planar quadrangulation decorated by a self-avoiding walk is an appropriate $\sqrt{8/3}$-LQG surface decorated by an SLE$_{8/3}$-type curve.  See Figure~\ref{fig-chart} for a schematic diagram of how the different works of the authors fit together with the existing literature to deduce this result.

\bigskip

\noindent{\bf Acknowledgements}
E.G.\ was supported by the U.S. Department of Defense via an NDSEG fellowship.  E.G.\ also thanks the hospitality of the Statistical Laboratory at the University of Cambridge, where this work was started.  J.M.\ thanks Institut Henri Poincar\'e for support as a holder of the Poincar\'e chair, during which this work was completed. We thank two anonymous referees for helpful comments on an earlier version of this article.

\subsection{Metric gluing of Brownian half-planes}
\label{sec-bhp-gluing}

Here we state several special cases of our main results which describe the curve-decorated metric spaces obtained by gluing together Brownian half-planes~\cite{gwynne-miller-uihpq,bmr-uihpq} along their boundaries according to the natural length measure as certain \emph{quantum wedges} or \emph{quantum cones} --- particular types of $\sqrt{8/3}$-LQG surfaces whose definition is reviewed in Section~\ref{sec-lqg-surface} --- equipped with their $\sqrt{8/3}$-LQG metric and decorated by SLE$_{8/3}$ curves (which correspond to the gluing interfaces). We note that a quantum wedge (resp.\ cone) can be parameterized by the half-plane (resp.\ whole plane). See Figure~\ref{fig-thick-gluing} for an illustration of the theorem statements in this section. 

The general versions of our main results, which describe the metric gluings of general quantum wedges, are stated in Section~\ref{sec-results} below. The results in this section follow from these general statements and the fact that the Brownian half-plane is the same as the $\sqrt{8/3}$- (weight-2) quantum wedge. 

We first consider two independent Brownian half-planes glued along ``half" of their boundaries, which corresponds to the left panel of Figure~\ref{fig-thick-gluing}. 

\begin{thm}
\label{thm-brownian-half-plane-wedge}
Suppose that we have two independent instances of the Brownian half-plane.  Then the metric quotient obtained by gluing the two surfaces together according to boundary length on their positive boundaries is isometric to the $\sqrt{8/3}$-LQG metric space associated with a weight-$4$ quantum wedge.  Moreover, the interface between the two Brownian half-plane instances is a chordal SLE$_{8/3}$ curve on this weight-$4$ quantum wedge, sampled independently from the wedge.
\end{thm}
 
We note that one can make sense of a chordal SLE$_{8/3}$ curve on a quantum wedge since the latter has a canonical conformal structure (i.e., a canonical embedding into $\BB H$, modulo scaling). One particular implication of Theorem~\ref{thm-brownian-half-plane-wedge}, which is not at all clear from the definition of the Brownian half-plane, is that the interface between the two Brownian half-planes (i.e., the image of the two glued boundary rays under the quotient map) is a simple curve. See Section~\ref{sec-gluing-remarks} for further discussion of this.

We next state an analog of Theorem~\ref{thm-brownian-half-plane-wedge} when we glue the two boundary rays of a single Brownian half-plane to itself, to get a quantum surface with the topology of the plane (Figure~\ref{fig-thick-gluing}, middle panel). 

\begin{thm}
\label{thm-brownian-half-plane-cone}
The metric space obtained by gluing the positive and negative boundaries of an instance of the Brownian half-plane together according to boundary length is isometric to the $\sqrt{8/3}$-LQG metric space associated with a weight-$2$ quantum cone and the interface is a whole-plane SLE$_{8/3}$ curve on this weight-$2$ quantum cone, sampled independently from the cone. 
\end{thm}

Finally, we consider two independent Brownian half-planes glued together along their \emph{whole} boundaries, which is illustrated in the right panel of Figure~\ref{fig-thick-gluing}. In this case the description of the gluing interface is slightly more complicated and involves whole-plane SLE$_{\kappa}(\rho)$ curves, which are defined in~\cite{ig4}.

\begin{thm}
\label{thm-brownian-half-plane-2side}
Suppose that we have two independent instances of the Brownian half-plane.  Then the metric quotient obtained by gluing the two surfaces together according to boundary length is isometric to the $\sqrt{8/3}$-LQG metric space associated with a weight-$4$ quantum cone.  Moreover, the interface between the two Brownian half-plane instances is an SLE$_{8/3}$-type curve independent from the weight-$4$ quantum cone, which can be sampled as follows. First sample a whole-plane SLE$_{8/3}(2)$ curve $\eta_1$ from $0$ to $\infty$; then, conditional on $\eta_1$, sample a chordal SLE$_{8/3}$ curve~$\eta_2$ from~$0$ to~$\infty$ in $\BB C\setminus \eta_1$. 
\end{thm}

We remark that the interface in Theorem~\ref{thm-brownian-half-plane-2side} can also be described by a pair of GFF flow lines \cite{ig1,ig4}. Theorem~\ref{thm-brownian-half-plane-2side} is deduced from Theorem~\ref{thm-brownian-half-plane-wedge} and Theorem~\ref{thm-cone-gluing}, in a manner described in Section~\ref{sec-results}.

\begin{figure}[ht!]
 \begin{center}
\includegraphics[scale=.9]{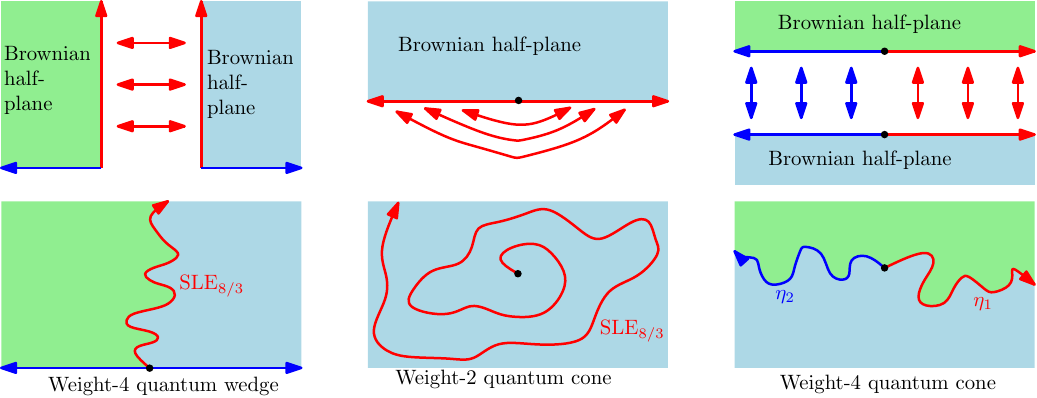} 
\caption[Gluing together Brownian half-planes]{\textbf{Left:} Illustration of Theorem~\ref{thm-brownian-half-plane-wedge} (the picture in Theorem~\ref{thm-wedge-gluing} in the case when $\frk w^- , \frk w^+ \geq 4/3$ is qualitatively the same). Gluing together two independent Brownian half-planes (green and blue) along their positive boundary rays produces a weight-$4$ quantum wedge decorated by a chordal SLE$_{8/3}$ curve. 
\textbf{Middle:} Illustration of Theorem~\ref{thm-brownian-half-plane-cone} (the picture in Theorem~\ref{thm-cone-gluing} is qualitatively the same). Gluing two complementary boundary rays of a Brownian half-plane produces a weight-$2$ quantum cone decorated by a whole-plane SLE$_{8/3}$. 
\textbf{Right:} Illustration of Theorem~\ref{thm-brownian-half-plane-2side}. Two weight-$2$ quantum wedges (Brownian half-planes) are glued together along their whole boundaries according to quantum length to obtain a weight-$4$ quantum cone. The gluing interface is a pair of non-intersecting SLE$_{8/3}$-type curves.
}\label{fig-thick-gluing}
\end{center}
\end{figure}

%gamma = Sqrt[8/3];
%FullSimplify[gamma/2 + (gamma/2 + 2/gamma) - 4 (1/gamma)]
%FullSimplify[(gamma/2 + 2/gamma) - 2 (1/(2 gamma))]
%FullSimplify[(gamma/2 + 2/gamma) - 4 (1/(2 gamma))]
%Solve[(gamma/2 + 2/gamma) - w (1/(2 gamma)) == gamma, w]

\subsection{Background on SLE and LQG}
\label{sec-prelim}

Before we state the general versions of our main results, we give a brief review of some definitions related to SLE and LQG which are needed for the statements. 

\subsubsection{Schramm-Loewner evolution}
\label{sec-sle}

Let $\kappa > 0$ (here we will only need the case $\kappa = 8/3$) and let $\ul\rho = (\rho_1,\dots , \rho_n)$ be a finite vector of weights. Also let $D\subset \BB C$ be a simply connected domain and let $x,y, z_1,\dots ,z_n \in D\cup \bdy D$. A \emph{chordal $\op{SLE}_\kappa(\ul\rho)$ from $x$ to $y$ in $D$} is a variant of chordal SLE$_\kappa$ from $x$ to $y$ in $D$ which has additional \emph{force points} at $z_1,\dots ,z_n$ of weights $\rho_1,\dots , \rho_n$, respectively.  It is a conformally invariant process which satisfies the so-called domain Markov process where one has to keep track of the extra marked points.  Chordal $\op{SLE}_\kappa(\ul\rho)$ processes were first introduced in~\cite[Section~8.3]{lsw-restriction}. See also~\cite{sw-coord} and~\cite[Section~2.2]{ig1}. In this paper we will primarily be interested in chordal SLE$_\kappa(\rho^L;\rho^R)$ with two force points of weight $\rho^L$ and $\rho^R$ located immediately to the left and right of the starting point, respectively. Such a process is well-defined provided $\rho^L , \rho^R > -2$~\cite[Section~2.2]{ig1}. We also recall the definition of whole-plane SLE$_\kappa(\rho)$ for $\rho  > -2$~\cite[Section~2.1]{ig4}.

\subsubsection{Liouville quantum gravity surfaces}
\label{sec-lqg-surface}

In this subsection we will give a brief review of Liouville quantum gravity (LQG) surfaces, as introduced in~\cite{shef-kpz,shef-zipper,wedges}. Such surfaces are defined for all $\gamma \in (0,2)$, but in this paper we will consider only the case when $\gamma = \sqrt{8/3}$, since this is the only case for which an LQG surface has a rigorously defined metric space structure (which we will review in Section~\ref{sec-lqg-metric}). 

A $\sqrt{8/3}$-LQG surface is an equivalence class of pairs $(D ,h)$, where $D\subset \BB C$ is an open set and $h$ is some variant of the Gaussian free field (GFF)~\cite{shef-gff,ss-contour,shef-zipper,ig1,ig4} on $D$. Two pairs $(D ,h)$ and $(\wt D , \wt h)$ are declared to be equivalent (meaning that they represent two different parameterizations of the same surface) if there is a conformal map $\phi : \wt D \rta D$ such that
\eqb \label{eqn-lqg-coord}
\wt h = h\circ \phi + Q\log |\phi'| ,\quad \op{where}\quad Q = \frac{2}{\gamma} + \frac{\gamma}{2} = \sqrt{\frac{3}{2}} + \sqrt{\frac{2}{3}}.
\eqe 
The particular choice of distribution $h$ is referred to as the \emph{embedding} of the quantum surface (into $D$). One can also define quantum surfaces with $k \in \BB N$ marked points in $D\cup\bdy D$ (with points in $\bdy D$ viewed as prime ends) by requiring the map $\phi$ in~\eqref{eqn-lqg-coord} to take the marked points of one surface to the corresponding marked points of the other. In this paper, we will often work with domains whose boundary is a simple curve, which means that a prime end is the same as a boundary point. But, we will sometimes work with domains for which this is not the case (e.g., a Jordan domain cut by a segment of a simple curve). 

It is shown in~\cite{shef-kpz} that a $\sqrt{8/3}$-LQG surface has a natural area measure $\mu_h$, which is a limit of regularized versions of $e^{\sqrt{8/3} h }\, dz$, where $dz$ denotes Lebesgue measure on $D$. Furthermore, there is a natural length measure $\nu_h$ which is defined on certain curves in $D$, including $\bdy D$ (viewed as a collection of prime ends) and SLE$_{8/3}$-type curves which are independent from $h$~\cite{shef-zipper}, and which is a limit of regularized versions of $e^{\sqrt{8/3} h/2}\, |dz|$, where $|dz|$ is the Euclidean length element. The measures $\mu_h$ and $\nu_h$ are invariant under transformations of the form~\eqref{eqn-lqg-coord} (see~\cite[Proposition~2.1]{shef-kpz} and its length measure analog). In the case of $\nu_h$, we recall here that a conformal map between simply connected domains induces a bijection between prime ends on their boundaries~\cite{pom-book}.

In this paper we will be interested in several specific types of $\sqrt{8/3}$-LQG surfaces which are defined and studied in~\cite{wedges}. Let $Q$ be as in~\eqref{eqn-lqg-coord}. For $\alpha \leq Q$, an \emph{$\alpha$-quantum wedge} is a doubly-marked quantum surface $(\BB H , h , 0, \infty)$ which can be obtained from a free-boundary GFF on $\BB H$ plus $-\alpha\log|\cdot|$ by ``zooming in near $0$" so as to fix the additive constant in a canonical way. See~\cite[Section~4.2]{wedges} for a precise definition. %The $\alpha$-quantum wedge locally looks like a free-boundary GFF plus $-\alpha\log |\cdot|$ (e.g., in the sense of local absolute continuity), but possesses a different scale invariance property~\cite[Proposition~4.7(i)]{wedges}. 
Quantum wedges in the case when $\alpha \leq Q$ are called \emph{thick wedges} because they describe surfaces homeomorphic to $\BB H$.  %Bounded neighborhoods of the marked point at $0$ a.s.\ have finite mass while neighborhoods of $\infty$ a.s.\ have infinite mass.

The quantum wedge for $\alpha=\sqrt{8/3}$ (which is isometric to the Brownian half-plane when equipped with its LQG metric) is special since a GFF-type distribution has a $-\sqrt{8/3}\log|\cdot|$ singularity at a point sampled uniformly from its $\sqrt{8/3}$-LQG boundary length measure~\cite[Section~6]{shef-kpz}. This means that a $\sqrt{8/3}$-quantum wedge describes the local behavior of a $\sqrt{8/3}$-LQG surface near a quantum typical point on its boundary. This is analogous to the statement that the Brownian half-plane describes the local behavior of a quantum disk at a typical point of its boundary (see, e.g.,~\cite[Proposition 4.2]{gwynne-miller-uihpq}). 

For $\alpha \in (Q , Q + \sqrt{2/3})$, an \emph{$\alpha$-quantum wedge} is an ordered Poissonian collection of doubly-marked quantum surfaces, each with the topology of the disk (the two marked points correspond to the points $\pm\infty$ in the infinite strip in~\cite[Section~4.5]{wedges}). The individual surfaces are called \emph{beads} of the quantum wedge. One can also consider a single bead of an $\alpha$-quantum wedge conditioned on its quantum area and/or its left and right quantum boundary lengths. See~\cite[Section~4.4]{wedges} for more details. Quantum wedges in the case when $\alpha \in (Q , Q + \sqrt{2/3})$ are called \emph{thin wedges} (because they are not homeomorphic to $\BB H$). 

It is sometimes convenient to parameterize the set of quantum wedges by a different parameter $\frk w$, called the \emph{weight}, which is given by
\eqb \label{eqn-wedge-weight}
\frk w = \sqrt{\frac{8}{3}}\left( \sqrt{\frac{2}{3}} + Q - \alpha \right) > 0.
\eqe
The case $\alpha =\sqrt{8/3}$ corresponds to $\frk w =2$.  
The reason why the weight parameter is convenient is that it is additive under the gluing and cutting operations for quantum wedges and quantum cones studied in~\cite{wedges}.
 
For $\alpha < Q$, an \emph{$\alpha$-quantum cone}, defined in~\cite[Definition~4.10]{wedges}, is a doubly-marked quantum surface $(\BB C , h , 0, \infty)$ obtained from a whole-plane GFF plus $-\alpha\log|\cdot|$ by ``zooming in near $0$" (so as to fix the additive constant in a canonical way) in a similar manner to how a thick wedge is obtained from a free-boundary GFF with a log singularity. %It locally looks like a whole-plane GFF plus $-\alpha\log |\cdot|$, but possesses a different scale invariance property~\cite[Proposition 4.13(i)]{wedges}. 
The \emph{weight} of an $\alpha$-quantum cone is given by
\eqb \label{eqn-cone-weight}
\frk w = 2 \sqrt{\frac{8}{3}} \left( Q - \alpha\right) .
\eqe
As in the quantum wedge case, the quantum cone for $\alpha = \sqrt{8/3}$ ($\frk w = 4/3$), which is isometric to the Brownian plane when equipped with its LQG metric, is special since it describes the local behavior of a $\sqrt{8/3}$-LQG surface at a typical point with respect to its LQG area measure.

A \emph{quantum sphere} is a finite-volume quantum surface $(\BB C , h)$ defined in~\cite[Definition~4.21]{wedges}, which is often taken to have fixed quantum area. One can also consider quantum spheres with one, two, or three marked points, which we take to be sampled uniformly and independently from the $\sqrt{8/3}$-LQG area measure~$\mu_h$. Note that the marked points in~\cite[Definition~4.21]{wedges} correspond to the points $\pm\infty$ in the cylinder, and are shown to be sampled uniformly and independently from $\mu_h$ in~\cite[Proposition~A.13]{wedges} when one conditions on the quantum surface structure of a quantum sphere (i.e., modulo conformal transformation). Equivalently, the law of a quantum sphere using the parameterization by the cylinder as described in \cite{wedges} is invariant under the operation of picking $x,y$ from $\mu_h$ independently and then applying a change of coordinates which takes $x,y$ to $\pm \infty$.

A \emph{quantum disk} is a finite-volume quantum surface with boundary $(\BB D , h)$ defined in~\cite[Definition~4.21]{wedges}, which can be taken to have fixed area or fixed area and fixed boundary length. A \emph{singly (resp.\ doubly) marked quantum disk} is a quantum disk together with one (resp.\ two) marked points in $\bdy \BB D$ sampled uniformly (and independently) from the $\sqrt{8/3}$-LQG boundary length measure $\nu_h$. Note that the marked points in~\cite[Definition~4.21]{wedges} correspond to the points $\pm\infty$ in the infinite strip. It is shown in~\cite[Proposition~A.8]{wedges} that the marked points are sampled uniformly from $\nu_h$, meaning that the law of the quantum disk is invariant under the operation of sampling two independent points from $\nu_h$ then applying a conformal map which sends these points to $\pm\infty$. It follows from the definitions in~\cite[Section~4.4 and 4.5]{wedges} that a doubly-marked quantum disk has the same law as a single bead of a (weight-$2/3$) $\sqrt 6$-quantum wedge if we condition on quantum area and left/right quantum boundary lengths (note that this is only true for $\gamma = \sqrt{8/3}$). 
%Solve[3-4/gamma^2==3+(4-2alpha gamma)/gamma^2,gamma]

Suppose $\frk w^-,\frk w^+ > 0$ and $\frk w = \frk w^- + \frk w^+$. It is shown in~\cite[Theorem~1.2]{wedges} that if one cuts a weight-$\frk w$ quantum wedge by an independent chordal SLE$_{8/3}(\frk w^--2 ; \frk w^+-2)$ curve (or a concatenation of such curves in the thin wedge case) then one obtains a weight-$\frk w^-$ quantum wedge and an independent-$\frk w^+$ quantum wedge. Since the $\sqrt{8/3}$-LQG length measures as viewed from either side of the curve match up~\cite{shef-zipper}, these two quantum wedges can be glued together according to quantum boundary length to recover the original weight-$\frk w$ quantum wedge. Similarly, by~\cite[Theorem~1.5]{wedges}, if one cuts a weight-$\frk w$ quantum cone by an independent whole-plane SLE$_{8/3}(\frk w-2)$ curve, then one obtains a weight-$\frk w$ quantum wedge whose left and right boundaries can be glued together according to quantum length to recover the original weight-$\frk w$ quantum cone.

\subsection{Metric gluing of general quantum wedges}
\label{sec-results}
 
In this section we will state the main results of this paper in full generality. Our first two theorems state that whenever we have a conformal welding of two or more $\sqrt{8/3}$-LQG surfaces along a simple SLE$_{8/3}$-type curve as in~\cite{shef-zipper,wedges}, the metric on the welding of the surfaces is equal to the metric space quotient of the surfaces being welded together according to boundary length. We start by addressing the case of two quantum wedges glued together to obtain another quantum wedge. 
For the statement (and at several other points in the paper) we will use the following definition.

\begin{defn} \label{def-cont}
Let $X$ be a topological space, let $Y\subset X$, and let $d$ be a metric on $Y$ which is continuous with respect to the topology on $Y$ inherited from $X$. If $A\subset \ol Y\setminus Y$, we say that $d$ \emph{extends by continuity} to $A$ if there is a metric $d'$ on $Y\cup A$ which agrees with $d$ on $Y$ and is continuous with respect to the topology on $Y\cup A$ which it inherits from $X$. 
\end{defn}

\begin{thm} \label{thm-wedge-gluing}
Let $\frk w^- , \frk w^+ > 0$ and let $\frk w := \frk w^- + \frk w^+$. If $\frk w \geq 4/3$, let $(\BB H ,h , 0 ,\infty)$ be a weight-$\frk w$ quantum wedge (recall~\eqref{eqn-wedge-weight}). If $\frk w \in(0,4/3)$, let $(\BB H , h , 0, \infty)$ be a single bead of a weight-$\frk w$ quantum wedge with area $\frk a > 0$ and left/right boundary lengths $\frk l^-, \frk l^+ > 0$. Let $\eta$ be an independent chordal SLE$_{8/3}(\frk w^- - 2;\frk w^+-2)$ from $0$ to $\infty$ in $\BB H$.  Let $W^-$ (resp.\ $W^+$) be the region lying to the left (resp.\ right) of $\eta$ and let $\mcl W^-$ (resp.\ $\mcl W^+$) be the quantum surface obtained be restricting $h$ to $W^-$ (resp.\ $W^+$). Let $\mcl U^-$ (resp.\ $\mcl U^+$) be the ordered sequence of connected components of the interior of $W^-$ (resp.\ $W^+$).  Let $\frk d_h$ be the $\sqrt{8/3}$-LQG metric induced by $h$. For $U\in \mcl U^\pm$ let $\frk d_{h|_U}$ be the $\sqrt{8/3}$-LQG metric induced by $h|_U$. Then a.s.\ each $\frk d_{h|_U}$ extends by continuity (with respect to the Euclidean topology) to $\bdy U$ and the metric space $(\BB H , \frk d_h)$ is the quotient of the disjoint union of the metric spaces $(\ol U , \frk d_{h|_U})$ for $U\in \mcl U^- \cup \mcl U^+$ under the natural identification (i.e., according to quantum length). 
\end{thm}

See Figure~\ref{fig-thick-gluing} (resp.\ Figure~\ref{fig-thin-gluing}) for an illustration of the statement of Theorem~\ref{thm-wedge-gluing} in the case when $\frk w^- , \frk w^+ \geq 4/3$ (resp.\ $\frk w^- , \frk w^+ < 4/3$). The reason why the metrics $\frk d_{h|_U}$ for $U\in \mcl U^-\cup \mcl U^+$ extend continuously to $\bdy U$ is explained in Lemma~\ref{prop-metric-bdy-wedge} below. 

In the setting of Theorem~\ref{thm-wedge-gluing},~\cite[Theorem~1.2]{wedges} implies that the quantum surfaces $\mcl W^-$ and $\mcl W^+$ are independent, the former is a weight-$\frk w^-$ quantum wedge (or a collection of beads of such a wedge if $\frk w  < 4/3$), and the latter is a weight-$\frk w^+$ quantum wedge (or  a collection of beads of such a wedge if $\frk w < 4/3$). 
Note in particular that $\mcl U^- = W^-$ if $\frk w^- \geq 4/3$ and $\mcl U^-$ is countably infinite if $\frk w^- < 4/3$, and similarly for $\mcl U^+$. 
  
The wedges (or subsets of wedges) $\mcl W^-$ and $\mcl W^+$ are glued together according to $\sqrt{8/3}$-LQG boundary length~\cite[Theorem~1.3]{shef-zipper}. 
Hence Theorem~\ref{thm-wedge-gluing} says that one can metrically glue a (subset of a) weight-$\frk w^-$ quantum wedge and a (subset of a) weight$-\frk w^+$ quantum wedge together according to quantum boundary length to get a (bead of a) weight-$\frk w$ quantum wedge. 

We note that Theorem~\ref{thm-brownian-half-plane-wedge} is a special case of Theorem~\ref{thm-wedge-gluing} since the Brownian half-plane is isometric to the weight-$2$ quantum wedge~\cite[Proposition 1.10]{gwynne-miller-uihpq}. 
More generally, the quotient metric space in Theorem~\ref{thm-wedge-gluing} is an LQG surface so locally looks the same as a Brownian surface, even near the gluing interface. In fact, the quotient metric $\frk d_h$ is independent from the gluing interface $\eta$. Hence if one takes a metric quotient of two quantum wedges, it is not possible to recover the two original wedges from the quotient metric space (one would need to see the gluing interface as well to do this). 

By the local absolute continuity of different $\sqrt{8/3}$-LQG surfaces near points of their boundaries, Theorem~\ref{thm-wedge-gluing} implies that a metric space quotient of any pair of independent $\sqrt{8/3}$ (equivalently Brownian) surfaces with boundary glued together according to boundary length also locally looks like a Brownian surface. This applies in particular if one metrically glues together two Brownian disks according to boundary length. We emphasize that it is not at all clear from the definition of a Brownian disk that gluing two such objects together produces something which locally looks like a Brownian disk near the gluing interface --- one needs to use LQG theory to obtain this fact.

Our next theorem concerns a quantum wedge glued to itself to obtain a quantum cone and is illustrated in the middle of Figure~\ref{fig-thick-gluing}.

\begin{thm}
\label{thm-cone-gluing}
Let $\frk w \geq 4/3$, let $(\BB C , h , 0, \infty)$ be a weight-$\frk w$ quantum cone (recall~\eqref{eqn-cone-weight}), and let $\frk d_h$ be the $\sqrt{8/3}$-LQG  metric induced by $h$. Let $\eta$ be a whole-plane SLE$_{8/3}(\frk w-2)$ process from $0$ to $\infty$. Let $U = \BB C\setminus \eta$ and let $\frk d_{h|_U}$ be the $\sqrt{8/3}$-LQG metric induced by $h|_U$. Then a.s.\ $\frk d_{h|_U}$ extends by continuity (with respect to the Euclidean topology) to $\bdy U$, viewed as a collection of prime ends, and a.s.\ $(\BB C , \frk d_h)$ is the metric quotient of $(U , \frk d_{h|_U})$ under the natural identification of the two sides of $\eta$. 
\end{thm}

\begin{figure}[ht!]
 \begin{center}
\includegraphics[scale=.8]{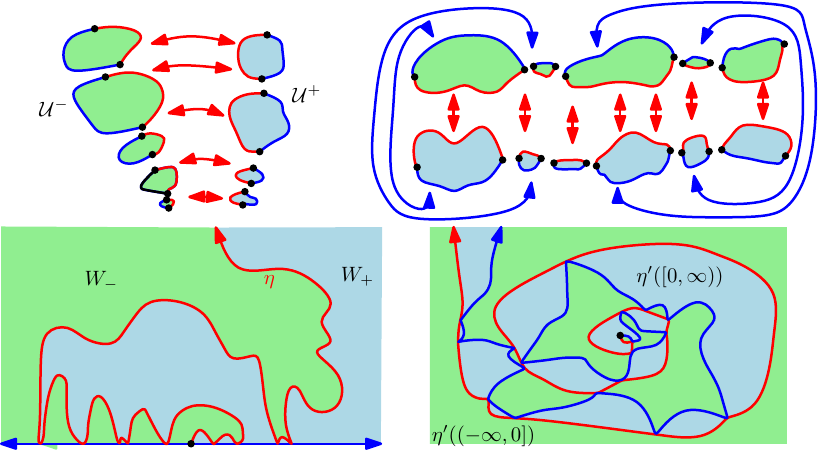} 
\caption[Gluing together thin quantum wedges]{
\textbf{Left:} Illustration of the statement of Theorem~\ref{thm-wedge-gluing} in the case when $\frk w \geq 4/3$ and $\frk w^- , \frk w^+ < 4/3$. The metric space $(\BB H , \frk d_h)$, which is a weight-$\frk w$ quantum wedge, is the metric quotient of the disjoint union of the metric spaces $(U , \frk d_{h|_U})$ for $U \in \mcl U^-\cup\mcl U^+$, which are the beads of a pair of independent quantum wedges of weights $\frk w^-$ and $\frk w^+$, glued together according to quantum boundary length. The gluing interface is a chordal SLE$_{8/3}(\frk w_1 - 2;\frk w_2-2)$ curve $\eta$ from $0$ to $\infty$ in $\BB H$, which a.s.\ hits both $(-\infty, 0]$ and $[0,\infty)$. 
\textbf{Right:} Illustration of the statement of Theorem~\ref{thm-peanosphere-gluing}. The boundary of a space-filling SLE$_6$ curve stopped when it hits~$0$ divides $\BB C$ into two independent weight-$2/3$ quantum wedges (green and blue), whose beads are independent quantum disks (Brownian disks) when we condition on boundary length. The weight $4/3$ quantum cone metric space $(\BB C , \frk d_h)$, which is a Brownian plane, is the metric quotient of these Brownian disks glued together according to boundary length.  
}\label{fig-thin-gluing}
\end{center}
\end{figure}

In the setting of Theorem~\ref{thm-cone-gluing},~\cite[Theorem~1.5]{wedges} implies that the surface $(U , h|_{U} , 0 , \infty)$ has the law of a weight-$\frk w$ quantum wedge. Furthermore, the boundary arcs of this quantum wedge lying to the left and right of $0$ are glued together according to $\sqrt{8/3}$-LQG boundary length~\cite[Theorem~1.3]{shef-zipper}. Hence Theorem~\ref{thm-cone-gluing} implies that one can metrically glue the left and right boundaries of a weight-$\frk w$ quantum wedge together according to quantum length to get a weight-$\frk w$ quantum cone. Similar absolute continuity remarks to the ones made after the statement of Theorem~\ref{thm-wedge-gluing} apply in the setting of Theorem~\ref{thm-cone-gluing}. 
 
Requiring that $\frk w \geq 4/3$ in Theorem~\ref{thm-cone-gluing} is equivalent to requiring that this wedge is thick, equivalently the curve $\eta$ is simple and the set $U$ is connected. We expect that one can also metrically glue the beads of a weight-$\frk w$ quantum wedge for $\frk w \in (0,4/3)$ together according to quantum length to get a weight-$\frk w$ quantum cone, but will not treat this case in the present work.
 
Note that Theorem~\ref{thm-brownian-half-plane-cone} is the special case of Theorem~\ref{thm-cone-gluing} when $\frk w = 2$. 
 
Iteratively applying Theorems~\ref{thm-wedge-gluing} and~\ref{thm-cone-gluing} can give us metric gluing statements for when we cut a quantum surface by multiple SLE$_{8/3}$-type curves. One basic application of this fact is Theorem~\ref{thm-brownian-half-plane-2side}, in which one identifies the entire boundary of the two Brownian half-plane instances together. This theorem is obtained by applying Theorem~\ref{thm-brownian-half-plane-wedge} to glue the positive boundaries of the two Brownian half-planes and then applying Theorem~\ref{thm-cone-gluing} to glue the two sides of the resulting weight-$4$ quantum wedge.

Another result which can be obtained from multiple applications of Theorems~\ref{thm-wedge-gluing} and~\ref{thm-cone-gluing} is a metric gluing statement in the setting of the peanosphere construction~\cite[Theorem~1.9]{wedges}, which allows us to express the Brownian plane or Brownian map as the metric space quotient of a countable collection of Brownian disks glued together along their boundaries. Before stating this result, we need to briefly recall the definition of space-filling SLE$_6$. 

\emph{Whole-plane space-filling SLE$_6$ from $\infty$ to $\infty$} is a variant of SLE$_6$ which travels from $\infty$ to $\infty$ in $\BB C$ and iteratively fills in bubbles as it disconnects them from $\infty$ (so in particular is not a Loewner evolution). As explained in~\cite[Footnote 4]{wedges}, whole-plane space-filling SLE$_6$ can be sampled as follows.
\begin{enumerate}
\item Let $\eta^L$ and $\eta^R$ be the flow lines of a whole-plane GFF started from $0$ with angles $\pi/2$ and $-\pi/2$, respectively, in the sense of~\cite{ig4}. Note that by~\cite[Theorem~1.1]{ig4}, $\eta^L$ has the law of a whole-plane SLE$_{8/3}(-2/3)$ from $0$ to $\infty$ and by~\cite[Theorem~1.11]{ig4}, the conditional law of $\eta^R$ given $\eta^L$ is that of a concatenation of independent chordal SLE$_{8/3}(-4/3; -4/3)$ processes in the connected components of $\BB C\setminus \eta^L$. 
\item Conditional on $\eta^L$ and $\eta^R$, concatenate a collection of independent chordal space-filling SLE$_{6}$ processes, one in each connected component of $\BB C\setminus (\eta^L \cup \eta^R)$. Such processes are defined in~\cite[Sections~1.1.3 and~4.3]{ig4} and can be obtained from ordinary chordal SLE$_6$ by, roughly speaking, iteratively filling in the bubbles it disconnects from $\infty$.
\end{enumerate}

\begin{thm} \label{thm-peanosphere-gluing} 
Let $(\BB C , h , 0, \infty)$ be a $\sqrt{8/3}$-quantum cone (weight-$4/3$) and let $\eta'$ be an independent whole-plane space-filling SLE$_6$ parameterized by quantum mass with respect to $h$ and satisfying $\eta'(0) =0$. Let $\mcl U^-$ (resp.\ $\mcl U^+$) be the set of connected components of the interior of $\eta'((-\infty ,0])$ (resp.\ $\eta'([0,\infty))$). Let $\frk d_h$ be the $\sqrt{8/3}$-LQG metric induced by $h$ and for $U\in \mcl U^-\cup \mcl U^+$ let $\frk d_{h|_U}$ be the $\sqrt{8/3}$-LQG metric induced by $h|_U$. Then a.s.\ each $\frk d_{h|_U}$ extends by continuity (with respect to the Euclidean topology) to $\bdy U$ and $(\BB C , \frk d_h)$ is the metric quotient of the disjoint union of the metric spaces $(\ol U , \frk d_{h|_U})$ for $U\in\mcl U^-\cup \mcl U^+$ under the natural identification. 
\end{thm} 
 
By~\cite[Corollary 1.5]{lqg-tbm2}, the metric space $(\BB C , \frk d_h)$ in Theorem~\ref{thm-peanosphere-gluing} is isometric to the Brownian plane. Furthermore, by~\cite[Theorems 1.2 and 1.5]{wedges}, each of the surfaces $(U , h|_U)$ has the law of a single bead of a weight-$2/3$ quantum wedge, which has the same law as a quantum disk. Hence each of the metric spaces $(U , \frk d_{h|_U})$ is isometric to a Brownian disk. Therefore Theorem~\ref{thm-peanosphere-gluing} expresses the Brownian plane as a metric space quotient of a countable collection of Brownian disks glued together according to boundary length.  
  
The following is an analog of Theorem~\ref{thm-peanosphere-gluing} in the setting of the finite-volume peanosphere construction~\cite[Theorem~1.1]{sphere-constructions}. 
  
\begin{thm} \label{thm-peanosphere-gluing-finite} 
Let $(\BB C , h ,\infty)$ be a singly marked unit area $\sqrt{8/3}$-quantum sphere and let $\eta'$ be an independent whole-plane space-filling SLE$_6$ parameterized by quantum mass with respect to $h$. Let $\BB t\in [0,1]$ be sampled uniformly from Lebesgue measure, independent from everything else. Let $\mcl U^-$ (resp.\ $\mcl U^+$) be the set of connected components of the interior of $\eta'([0,\BB t])$ (resp.\ $\eta'([\BB t  ,1]$). Let $\frk d_h$ be the $\sqrt{8/3}$-LQG metric induced by $h$ and for $U\in \mcl U^-\cup \mcl U^+$ let $\frk d_{h|_U}$ be the $\sqrt{8/3}$-LQG metric induced by $h|_U$. Then a.s.\ each $h|_U$ extends by continuity (with respect to the Euclidean topology) to $\bdy U$ and $(\BB C , \frk d_h)$ is the metric quotient of the disjoint union of the metric spaces $(\ol U , \frk d_{h|_U})$ for $U\in\mcl U^-\cup \mcl U^+$ under the natural identification. 
\end{thm} 
 
By~\cite[Corollary 1.5]{lqg-tbm2}, the metric space $(\BB C , \frk d_h)$ in Theorem~\ref{thm-peanosphere-gluing-finite} is isometric to the Brownian map. Furthermore, by~\cite[Theorem~7.1]{sphere-constructions} each of the surfaces $(U , h|_U)$ are conditionally independent given their quantum boundary lengths and areas, and each has the law of a single bead of a weight-$2/3$ quantum wedge (which has the same law as a quantum disk with given boundary length and area) under this conditioning. Hence each of the metric spaces $(U , \frk d_{h|_U})$ is isometric to a Brownian disk. Therefore Theorem~\ref{thm-peanosphere-gluing} expresses the Brownian map as a metric space quotient of a countable collection of Brownian disks glued together according to boundary length.

\subsection{Basic notation}
\label{sec-basic}

Here we record some basic notation which we will use throughout this paper.

\begin{notation} \label{def-integers}
We write $\BB N = \{1,2,3,\dots\}$. 
\end{notation}

\begin{notation} \label{def-discrete-intervals}
For $a < b \in \BB R$ and $c > 0$, we define the discrete intervals $[a,b]_{c \BB Z} := [a, b]\cap (c \BB Z)$ and $(a,b)_{c\BB Z} := (a,b)\cap (c\BB Z)$. 
\end{notation}
 
\begin{notation}\label{def-asymp}
If $a$ and $b$ are two quantities, we write $a\preceq b$ (resp.\ $a \succeq b$) if there is a constant $C$ (independent of the parameters of interest) such that $a \leq C b$ (resp.\ $a \geq C b$). We write $a \asymp b$ if $a\preceq b$ and $a \succeq b$. 
\end{notation}

\begin{notation} \label{def-o-notation}
If $a$ and $b$ are two quantities which depend on a parameter $x$, we write $a = o_x(b)$ (resp.\ $a = O_x(b)$) if $a/b \rta 0$ (resp.\ $a/b$ remains bounded) as $x \rta 0$ (or as $x\rta\infty$, depending on context).
\end{notation}

Unless otherwise stated, all implicit constants in $\asymp, \preceq$, and $\succeq$ and $O_x(\cdot)$ and $o_x(\cdot)$ errors involved in the proof of a result are required to depend only on the auxiliary parameters that the implicit constants in the statement of the result are allowed to depend on.

\subsection{Outline}
\label{sec-outline}

The remainder of this article is structured as follows. In Section~\ref{sec-prelim-metric}, we review the definitions of internal metrics and quotient metrics (Section~\ref{sec-metric-basic}) and discuss the difficulties involved in gluing together metric spaces along a curve in a general setting (Section~\ref{sec-gluing-remarks}). We also review the construction of the $\sqrt{8/3}$-LQG metric from~\cite{lqg-tbm1,lqg-tbm2,lqg-tbm3} (Section~\ref{sec-lqg-metric}). 

In Section~\ref{sec-bd}, we recall the definition of the Brownian disk from~\cite{bet-mier-disk} and prove some estimates for this metric space which will be needed for the proofs of our main results. Roughly speaking, these estimates tell us that in various situations one has the relations
\eqbn
\op{Area} \approx \op{Length}^{1/2} \approx \op{Distance}^{1/4} 
\eqen
with high probability, where area, length, and distance, respectively, refer to the natural area measure, boundary length measure, and metric on the Brownian disk. 

In Section~\ref{sec-metric-gluing} we prove our main results using the estimates from Section~\ref{sec-bd} and some facts about SLE and LQG. We first argue in Section~\ref{sec-geodesic-hit} that geodesics in $\sqrt{8/3}$-LQG surfaces do not hit the boundary, using~\cite[Lemma 18]{bet-mier-disk} and the equivalence of Brownian and $\sqrt{8/3}$-LQG surfaces. 

Section~\ref{sec-wedge-gluing} contains the proof of Theorem~\ref{thm-wedge-gluing} (the proof of Theorem~\ref{thm-cone-gluing} is identical). 
We will apply the estimates of Section~\ref{sec-bd} together with the equivalence of the Brownian disk and the quantum disk and the local absolute continuity of various $\sqrt{8/3}$-LQG surfaces to show that a certain regularity event governing the relationships between the $\mu_h$-mass, $\nu_h$-length, and $\frk d_h$-diameters of certain sets holds with high probability (Lemma~\ref{lem-metric-reg}). We then argue that on this regularity event, $\sqrt{8/3}$-LQG geodesics can be ``re-routed" (i.e., replaced by slightly different paths) in such a way that they cross the SLE curve only finitely many times, and their length increases by only a small amount. Since distances with respect to the metric space quotient are defined as the infimum of the lengths of paths which cross the gluing interface only finitely many times (Section~\ref{sec-metric-basic}), this implies the theorem statement. A more detailed outline of this argument appears at the beginning of Section~\ref{sec-wedge-gluing}. Section~\ref{sec-general-gluing} deduces Theorems~\ref{thm-peanosphere-gluing} and~\ref{thm-peanosphere-gluing-finite} from Theorems~\ref{thm-wedge-gluing} and~\ref{thm-cone-gluing}.

\section{Metric space preliminaries}
\label{sec-prelim-metric}

\subsection{Basic definitions for metrics}
\label{sec-metric-basic}

In this paper we will consider a variety of metric spaces, including subsets of $\BB C$ equipped with the Euclidean metric, $\sqrt{8/3}$-LQG surfaces equipped with the metric induced by QLE$(8/3,0)$, and various Brownian surfaces equipped with their intrinsic metric. Here we introduce some notation to distinguish these metric spaces and recall some basic constructions for metric spaces. 

\begin{defn} \label{def-metric-ball}
Let $(X,d_X)$ be a metric space. For $A\subset X$ we write $\op{diam} (A ; d_X)$ for the supremum of the $d_X$-distance between pairs of points in $A$. For $r>0$, we write $B_r(A;d_X)$ for the set of $x\in X$ with $d_X(x,A) < r$. If $A = \{y\}$ is a singleton, we write $B_r(\{y\};d_X) = B_r(y;d_X)$. For $A\subset \BB C$, we write $B_r(A)$ for the set of points lying at Euclidean distance (strictly) less than $r$ from $A$. 
\end{defn}

Recall that a \emph{pseudometric} on a set $X$ is a function $d_X : X\times X \rta [0,\infty)$ which satisfies all of the conditions in the definition of a metric except that possibly $d_X(x,y) = 0$ for $x\not=y$. 

Let $(X,d_X)$ be a topological space equipped with a continuous pseudometric $d_X$, let $\sim$ be an equivalence relation on $X$, and let $\ol X = X/\sim$ be the corresponding topological quotient space. For equivalence classes $\ol x , \ol y\in \ol X$, let $\mcl Q(\ol x , \ol y)$ be the set of finite sequences $(x_1 , y_1 ,    \dots , x_n , y_n)$ of elements of $X$ such that $x_1 \in \ol x$, $y_n \in \ol y$, and $y_i \sim x_{i+1}$ for each $i \in [1,n-1]_{\BB Z}$. Let
\eqb \label{eqn-quotient-def}
\ol d_X(\ol x , \ol y) := \inf_{(x_1 , y_1 ,    \dots , x_n , y_n) \in \mcl Q(\ol x , \ol y)} \sum_{i=1}^n d_X (x_i ,y_i ) .
\eqe  
Then $\ol d_X$ is a pseudometric on $\ol X$, which we call the \emph{quotient pseudometric}. It is easily seen from the definition that the quotient pseudometric possesses the following universal property. Suppose $f : (X,d_X) \rta (Y , d_Y)$ is a $1$-Lipschitz map such that $f(x) = f(y)$ whenever $x,y\in X$ with $x\sim y$. Then $f$ factors through the metric quotient to give a 1-Lipschitz map $\ol f : \ol X \rta Y$ such that $\ol f \circ p = f$, where $p : X\rta \ol X$ is the quotient map.

For a curve $\gamma : [a,b] \rta X$, the \emph{$d_X$-length} of $\gamma$ is defined by 
\eqbn
\op{len}\left( \gamma ; d_X \right) := \sup_P \sum_{i=1}^{\# P} d_X(\gamma(t_i) , \gamma(t_{i-1})) 
\eqen
where the supremum is over all partitions $P : a= t_0 < \dots < t_{\# P} = b$ of $[a,b]$. Note that the $d_X$-length of a curve may be infinite.

Suppose $Y\subset X$. The \emph{internal metric of $d_X$ on $Y$} is defined by
\eqb \label{eqn-internal-def}
d_Y (x,y)  := \inf_{\gamma \subset Y} \op{len}\left(\gamma ; d_X\right) ,\quad \forall x,y\in Y 
\eqe 
where the infimum is over all curves in $Y$ from $x$ to $y$. 
The function $d_Y$ satisfies all of the properties of a pseudometric on $Y$ (or a metric, if $d_X$ is a metric) except that it may take infinite values.

\subsection{Remarks on metric gluing}
\label{sec-gluing-remarks}

There are a number of pathologies that can arise in the context of metric gluing.  In what follows, we will describe two such examples.  The first is concerned with what types of problems can arise when one tries to recover a metric space as the metric quotient of the two metric spaces which arise by considering the internal metric when one cuts along a simple curve.  The second example will show that when one considers the metric quotient of two copies of $[0,1]^2$ glued along $[0,1]$ the gluing interface can in fact collapse to a single point.

\medskip

\noindent{\bf Gluing along a simple curve.}  We know a priori (see Lemma~\ref{prop-internal-equal} below) that in the setting of either Theorem~\ref{thm-wedge-gluing} or Theorem~\ref{thm-cone-gluing}, it holds for each $U\in \mcl U^- \cup \mcl U^+$ that the restriction of the metric $\frk d_{h|_U}$ to the open set $U$ coincides with the internal metric of $\frk d_h$ on $U$, as defined in Section~\ref{sec-metric-basic}. However, this fact together with the fact that the gluing interface $\eta$ is a continuous simple curve is far from implying the statements of Theorems~\ref{thm-wedge-gluing} and~\ref{thm-cone-gluing} since neither of these properties rules out the possibility that paths which hit $\eta$ infinitely many times are much shorter than paths which cross only finitely many times (recall that the quotient metric is defined in terms of the infimum of the lengths of paths which cross the interface only finitely many times).

The problem of proving that a metric space cut by a simple curve is the quotient of the internal metrics on the two sides of the curve is similar in spirit to the problem of proving that a curve in $\BB C$ is \emph{conformally removable}~\cite{jones-smirnov-removability}, which means that any homeomorphism of $\BB C$ which is conformal off of the curve is in fact conformal on the whole plane. Indeed, proving each involves estimating how much the length of a path (the image of a straight line in the case of removability or a geodesic in the case of metric gluing) is affected by its crossings of the curve.  Moreover, in the setting of LQG and SLE, both the question of removability and the metric gluing problem addressed in this paper are ways to show that the surfaces formed by cutting along an SLE curve together determine the overall surface (see~\cite{shef-zipper,wedges} for further discussion of this point in the case of removability), although we are not aware of a direct relationship between the two concepts.

SLE$_\kappa$-type curves for $\kappa < 4$ are conformally removable since they are boundaries of H\"older domains~\cite{schramm-sle,jones-smirnov-removability}.  However, there is no such simple criterion for metric gluing.  We know that the SLE$_{8/3}$ gluing interfaces in our setting are H\"older continuous for any exponent less than $1/2$ with respect to $\frk d_h$ (see, e.g., Lemma~\ref{prop-disk-bdy-holder} below). However, even Lipschitz continuity of the gluing interface does not imply the sorts of metric gluing statements we are interested in here, as the following example demonstrates.

Let $X = [0,1]\times [-1,1]$ equipped with the Euclidean metric $d$ and let $(\wt X , \wt d)$ be the metric quotient of $X \sqcup [0,1/2]$ under the equivalence relation which identifies $t\in [0,1/2]$ with $(2t,0)$ in $X$. In other words, $\wt d$ is obtained from $d$ by shortening the lengths of paths which trace along $[0,1] \times \{0\}$ by a factor of $1/2$.
The space $(X,d)$ is the metric quotient of the disjoint union of $[0,1]\times [-1,0]$ and $[0,1]\times [0,1]$, each equipped with their $d$-internal metrics (which both coincide with the Euclidean metric) under the natural identification. Furthermore, $(X,d)$ and $(\wt X , \wt d)$ are homeomorphic (in fact bi-Lipschitz) via the obvious identification and the internal metrics of $d$ and $\wt d$ on each of the two sides $[0,1] \times [-1,0)$ and $[0,1]\times (0,1]$ of $[0,1]\times \{0\}$ coincide. However, $d\not=\wt d$ since points near the interface $[0,1]\times \{0 \}$ are almost twice as far apart with respect to $d$ as with respect to $\wt d$.

In the example above, $\wt d$-geodesics between points near $[0,1]\times \{0\}$ spend most of their time in the gluing interface $[0,1] \times \{0\}$. In fact, paths which trace along the gluing interface are substantially shorter than those which do not, so $\wt d$-distances cannot be approximated by the lengths of paths which cross this interface only finitely many times.
The proofs of Theorems~\ref{thm-wedge-gluing} and~\ref{thm-cone-gluing} amount to ruling out this sort of behavior for $\frk d_h$-geodesics. 
In particular, we will use estimates for  how often a $\frk d_h$ geodesic hits the SLE$_{8/3}$ curve $\eta$ and how distances behave near $\eta$ to show that one can slightly perturb such a geodesic in such a way that it crosses the gluing interface only finitely many times and its length is increased by only a small amount.

\medskip
 
\noindent{\bf Gluing interface collapses to a single point.}  It is possible to have much more pathological behavior when we consider metric gluings where the function which identifies points along the boundaries of the two spaces being glued is not Lipschitz. For example, as Lemma~\ref{lem-bdy-collapse} below demonstrates, it is possible for the boundary to collapse to a single point.  If one considers metric gluings of Brownian surfaces along their boundary as in Section~\ref{sec-bhp-gluing} (without reference to SLE/LQG theory), then this is the type of pathology one would be led to worry about as it is not immediate from the Brownian surface theory that this does not happen.  The following lemma shows that such pathological behavior does in fact arise in many settings.

\begin{lem} \label{lem-bdy-collapse}
Let $(X_1,d_1)$ and $(X_2,d_2)$ be two copies of $[0,1]\times[0,1]$, equipped with the Euclidean distance. Let $\nu$ be a non-atomic Borel measure on $[0,1]$ with $\nu([0,1]) = 1$ which is mutually singular with respect to Lebesgue measure (e.g., $\nu$ could be a $\gamma$-LQG boundary measure for $\gamma \in (0,2)$) and let $f(s) := \nu([0,s])$ for $s\in[0,1]$. 
Let $(X,d)$ be the metric space quotient of the disjoint union of $X_1$ and $X_2$ under the equivalence relation which identifies $(s , 0) \in X_1$ with $(f(s) ,0) \in X_2$. Then the $d$-distance between any two points of the gluing interface (i.e., the image of the two copies of $[0,1]\times \{0\}$ under the quotient map) is 0. 
\end{lem}

Before we give the proof of Lemma~\ref{lem-bdy-collapse}, let us mention that SLE/LQG theory allows us to immediately rule out the possibility that the gluing interface degenerates to a point in each of the theorems of Section~\ref{sec-bhp-gluing} (in fact, the gluing interface must be a simple curve). The reason for this is as follows. In each theorem, the claimed quotient metric space (namely, a certain type of quantum wedge or cone) can be obtained by identifying one or more Brownian half-planes (weight-2 wedges) together along their boundaries due to the conformal welding results of~\cite{wedges}. By the universal property of the quotient metric (Section~\ref{sec-metric-basic}) the quotient metric is the \emph{largest} metric compatible with the equivalence relation, so there must be a 1-Lipschitz map from the actual quotient metric space to the claimed quotient metric space which preserves the gluing interface. Since the gluing interface in the claimed quotient metric space is an SLE$_{8/3}$-type curve, the gluing interface in the actual metric space quotient is also a simple curve.

\begin{proof}[Proof of Lemma~\ref{lem-bdy-collapse}]
Let $q : X_1\sqcup X_2 \rta X$ be the quotient map, let $0 < s_1 < s_2 < 1$, and view $(s_1,0)$ and $(s_2,0)$ as points of $X_1$. We will show that $d(q(s_1,0) , q(s_2,0)) = 0$. 
To this end, fix $\ep > 0$. Since $\nu$ is mutually singular with respect to Lebesgue measure, we can find $s_1 = y_0 < x_1 < y_1 < \dots < x_n < y_n < x_{n+1} = s_2$ such that 
\eqbn
\sum_{j=0}^{n+1} \nu([y_{j-1} , x_j] ) \leq \ep \quad \text{and} \quad \sum_{j=1}^n (y_j-x_j) \leq \ep .
\eqen 
By definition of the quotient metric, we therefore have
\alb
d(q(s_1,0) , q(s_2,0)) 
&\leq \sum_{j=1}^n d_1((x_j,0) , (y_j , 0) ) + \sum_{j=0}^{n+1} d_2\left((f(y_{j-1} ) , 0) , (f(x_j) , 0) \right) \\
&\leq \sum_{j=1}^n (y_j - x_j) + \sum_{j=0}^{n+1} (f(x_j) - f(y_{j-1}) ) \leq 2\ep ,
\ale
which concludes the proof since $\ep > 0$ is arbitrary. 
\end{proof}

\subsection{The $\sqrt{8/3}$-LQG metric}
\label{sec-lqg-metric}

Suppose $(D , h)$ is a $\sqrt{8/3}$-LQG surface. 
In~\cite{lqg-tbm1,lqg-tbm2,lqg-tbm3}, it is shown that if $h$ is some variant of the GFF on $D$, then $h$ induces a metric $\frk d_h$ on $D$ (which in many cases extends to a metric on $D\cup\bdy D$). The construction of this metric builds on the results of~\cite{wedges,tbm-characterization,sphere-constructions,qle}. 

In the special case when $(D,h)$ is a quantum sphere, the metric space $(D , \frk d_h)$ is isometric to the Brownian map~\cite[Theorem~1.4]{lqg-tbm2}. It is shown in~\cite[Corollary 1.5]{lqg-tbm2} that the metric space $(D , \frk d_h)$ is isometric to a Brownian surface in two additional cases: when $(D , h)$ is a quantum disk we obtain a Brownian disk~\cite{bet-mier-disk} and when $(D,h)$ is a $\sqrt{8/3}$-quantum cone we obtain a Brownian plane~\cite{curien-legall-plane}.  It is shown in~\cite{gwynne-miller-uihpq} that the Brownian half-plane is isometric to the $\sqrt{8/3}$-quantum wedge.
Hence $\sqrt{8/3}$-LQG surfaces can be viewed as Brownian surfaces equipped with a conformal structure.

For the convenience of the reader, we give in this section a review of the construction of the metric $\frk d_h$ and note some basic properties which it satisfies.

\subsubsection{The metric on a quantum sphere}
\label{sec-lqg-metric-sphere}
 
The $\sqrt{8/3}$-LQG metric is first constructed in the case of a $\sqrt{8/3}$-LQG sphere $(\BB C , h)$.  Conditional on $h$, let $\mcl C$ be a countable collection of i.i.d.\  points sampled uniformly from the $\sqrt{8/3}$-LQG area measure $\mu_h$. One first defines for $z , w \in \mcl C$ a $\op{QLE}(8/3,0)$ growth process $\{\Gamma_t^{z,w}\}_{t\geq 0}$ started from $z$ and targeted at $w$, which is a continuum analog of first passage percolation on a random planar map~\cite[Section~2.2]{qle}. 
\footnote{It is expected that the process $\Gamma_t^{z,w}$ is a re-parameterization of a whole-plane version of the QLE$(8/3,0)$ processes considered in~\cite{qle}, which are parameterized by capacity instead of by quantum natural time. However, this has not yet been proven to be the case. }
This is accomplished as follows. Let $\delta >0$ and let $\eta_0^\delta$ be a whole-plane SLE$_6$ from $z$ to $w$ sampled independently from $h$ and then run for $\delta$ units of quantum natural time as determined by $h$~\cite{wedges}.  For $t\in [0,\delta]$, let $\Gamma_t^{z,w,\delta} := \eta_0^\delta([0, \tau_0^\delta \wedge \delta])$, where $\tau_0^\delta$ is the first time $\eta_0^\delta$ hits $w$.  

Inductively, suppose $k\in\BB N$ and $\Gamma_t^{z,w,\delta}$ has been defined for $t\in [0, k\delta]$. If $w\in \Gamma_{k\delta}^{z,w,\delta}$, let $\Gamma_t^{z,w,\delta} = \Gamma_{k\delta}^{z,w,\delta}$ for each $t\in [k\delta,(k+1)\delta]$. Otherwise, let $x_k^\delta$ be sampled uniformly from the $\sqrt{8/3}$-LQG length measure $\nu_h$ restricted to the boundary of the connected component of $\BB C\setminus \Gamma_t^{z,w,\delta}$ containing $w$. Let $\eta_k^\delta$ be a radial SLE$_6$ from $x_k^\delta$ to $w$ sampled conditionally independently of $h$ given $x_k^\delta$ and $\{ \Gamma_s^{z,w,\delta} \}_{s\leq t}$ and parameterized by quantum natural time as determined by $h$. For $t\in [k\delta,(k+1)\delta]$, let $\Gamma_t^{z,w,\delta} := \eta_k^\delta([0,\tau\wedge (t-k\delta)]) \cup \Gamma_{k\delta}^{z,w,\delta}$, where $\tau_k^\delta$ is the first time that $\eta_k^\delta$ hits $w$.

The above procedure defines for each $\delta > 0$ a growing family of sets $\{\Gamma^{z,w,\delta}_t\}_{t\geq 0}$ started from~$z$ and stopped when it hits~$w$. It is shown in~\cite{lqg-tbm1} that (along an appropriately chosen subsequence), one can take an a.s.\ limit (in an appropriate topology) as $\delta \rta 0$ to obtain a growing family of sets $\{\Gamma_t^{z,w}\}_{t\geq0 }$ from~$z$ to~$w$, which we call $\op{QLE}(8/3,0)$.  It is shown in \cite{lqg-tbm2} that the limiting process $\{\Gamma_t^{z,w}\}_{t\geq0 }$ is a.s.\ determined by~$h$, even though the approximations are not and that the limit does not depend on the choice of subsequence.

For $t \geq 0$, let $X_t^{z,w}$ be the $\nu_h$-length of the boundary of the connected component of $\BB C\setminus \Gamma_t^{z,w}$ containing~$w$. Let~$\sigma^{z,w}_{r}$ for $r\geq 0$ be defined by
\begin{equation}
\label{eqn-qle-time-change}
r = \int_0^{\sigma^{z,w}_r} \frac{1}{X_t^{z,w}} \, dt .
\end{equation}
Set $\wt\Gamma^{z,w}_r := \Gamma^{z,w}_{\sigma^{z,w}_r}$. The $\sqrt{8/3}$-LQG distance between $z$ and $w$ is defined by
\eqbn
\frk d_h(z,w) := \inf\left\{r \geq 0 \,:\, w \in \wt\Gamma_r^{z,w} \right\} .
\eqen
The time-change~\eqref{eqn-qle-time-change} is natural from the perspective of first passage percolation.  Indeed, quantum natural time is the continuum analog of parameterizing a percolation growth by the number of edges traversed, hence the time change~\eqref{eqn-qle-time-change} is the continuum analog of adding edges to the cluster at a rate proportional to boundary length.
It is shown in~\cite{lqg-tbm1} that this function defines a metric on the set $\mcl C$ (which is a.s.\ dense in $\BB C$). It is shown in~\cite{lqg-tbm2} that $\frk d_h$ in fact extends continuously to a metric on all of $\BB C$, which is mutually H\"older continuous with respect to the metric on $\BB C$ induced by the stereographic projection of the standard metric on the Euclidean sphere $\BB S^2$ and isometric to the Brownian map. In particular, $\frk d_h$ is a geodesic metric.

The re-parameterized $\op{QLE}(8/3,0)$ processes $\{\wt\Gamma^{z,w}_r\}_{r\geq 0}$ for $z,w\in\mcl C$ are related to metric balls for $\frk d_h$ as follows. For $z,w\in \mcl C$ and $r \geq 0$, the connected component of $\BB C\setminus \wt\Gamma_r^{z,w}$ containing $w$ is the same as the connected component of $\BB C\setminus \ol B_r(z;\frk d_h)$ containing $w$.  

Each of the re-parameterized QLE$(8/3,0)$ hulls $\wt\Gamma^{z,w}_r$ is a local set for $h$ in the sense of~\cite[Section~3.9]{ss-contour}, and furthermore is determined locally by $h$.\footnote{One can check that the same is true if $r$ is a stopping time for the filtration generated by $(\wt\Gamma_r^{z,w}  , h|_{\wt\Gamma_r^{z,w}})$ by the usual argument (approximate by stopping times which take only dyadic values and use that the local set property behaves well under taking limits using, e.g., the first characterization of local sets from~\cite[Lemma 3.9]{ss-contour}.)} Hence the definition of the metric $\frk d_h$ implies that if $U \subset \BB C$ is a deterministic connected open set, then the quantities $\{\frk d_h(z,w) \wedge \frk d_h(z,\bdy U) \,:\, z,w\in U\}$ are a.s.\ determined by $h|_U$. In particular, the internal metric of $\frk d_h$ on $U$ (Section~\ref{sec-metric-basic}) is a.s.\ determined by $h|_U$. 

The above metric construction also works with $h + R$ in place of $h$ for any $R\in\BB R$, in which case~\cite[Lemma~2.2]{lqg-tbm2} yields a scaling property for the metric $\frk d_h$:
\eqbn
\frk d_{h+R}(z,w) = e^{\sqrt{8/3}R/4} \frk d_h(z,w) .
\eqen  
 
It is shown in \cite{lqg-tbm3} that the $\sqrt{8/3}$-LQG surface associated with a given Brownian surface is almost surely determined by the metric measure space structure associated with the Brownian surface.  This in particular implies that if one is given an instance of the Brownian map, disk, half-plane, or plane, respectively, then there is a measurable way to embed the surface to obtain an instance of a $\sqrt{8/3}$-LQG sphere, disk, wedge, or cone, respectively.  As mentioned above, the construction of the $\sqrt{8/3}$-LQG metric also implies that the Brownian map, disk, half-plane, or plane structure is a measurable function of the corresponding $\sqrt{8/3}$-LQG structure.  In this way, Brownian and $\sqrt{8/3}$-LQG surfaces are one and the same.

\subsubsection{Metrics on general $\sqrt{8/3}$-LQG surfaces}
\label{sec-lqg-metric-general}

In this subsection we let $D\subset \BB C$ be a connected open set and we let $h$ be a random distribution on $D$ with the following property. For each bounded deterministic open set $U\subset D$ at positive Euclidean distance from $\bdy D$, the law of $h|_U$ is absolutely continuous with respect to the corresponding restriction of some embedding into~$\BB C$ of a quantum sphere (with possibly random area). For example, $h$ could be an embedding of a quantum disk, a thick ($\alpha \leq Q$) quantum wedge, a single bead of a thin ($\alpha \in (Q , Q +\sqrt{2/3} )$) quantum wedge, or a quantum cone (in this last case we take $D$ to be the complement of the two marked points for the cone). We will show how to obtain a $\sqrt{8/3}$-LQG metric $\frk d_h$ on $D$ from $h$. 

The discussion at the end of Section~\ref{sec-lqg-metric-sphere} together with local absolute continuity allows us to define a metric $\frk d_{h|_U}$ for any bounded open connected set $U\subset D$ at positive Euclidean distance from $\bdy D$. If we let $\{U_n\}_{n\in\BB N}$ be an increasing sequence of such open sets with $\bigcup_{n=1}^\infty U_n = D$, then the metrics $\frk d_{h_{U_n}}$ (extended to be identically equal to $\infty$ for points outside of $U_n$) are decreasing as $n\rta\infty$, so the limit
\eqbn
\frk d_h(z,w) := \lim_{n\rta\infty} \frk d_{h|_{U_n}}(z,w)
\eqen
exists for each $z,w\in D$ and defines a metric on $D$. It is easy to see that this metric does not depend on the choice of $\{U_n\}_{n\in\BB N}$. 

We now record some elementary properties of the metric $\frk d_h$. The first property is immediate from the above definition. 
 
\begin{lem} \label{prop-internal-equal}
Suppose we are in the setting described just above, so that~$\frk d_h$ is a well-defined metric on~$D$. For any deterministic open connected set~$U\subset \BB C$, the internal metric (Section~\ref{sec-metric-basic}) of~$\frk d_h$ on~$U$ is a.s.\ equal to~$\frk d_{h|_U}$. 
\end{lem}

\begin{remark} \label{remark-internal-compare}
It follows from Lemma~\ref{prop-internal-equal} that if $A\subset \bdy U$ and  $\frk d_{h|_U}$ extends by continuity (with respect to the Euclidean topology) in the sense of Definition~\ref{def-cont} to a metric on $U\cup A$ (e.g., using the criterion of Lemma~\ref{prop-metric-bdy} just below) then for $x,y\in U\cup A$, we have $\frk d_{h|_U}(x,y) \geq \frk d_h(x,y)$. 
Indeed, the statement of the lemma immediately implies that this is the case whenever $x,y\in U$. For points in $A$, we take limits and use that both $\frk d_{h|_{U\cup A}}$ and $\frk d_h|_{U\cup A}$ (and the Euclidean metric) induce the same topology on $U$. We do \emph{not} prove in this paper that $\frk d_{h|_U}$ is the same as the internal metric of $\frk d_h$ on $U\cup A$. 
\end{remark}

Next we note that the LQG metric is coordinate invariant.

\begin{lem} \label{prop-metric-coord}
Suppose we are in the setting above. Let $\wt D$ be another domain and let $\phi : \wt D \rta  D$ be a conformal map. If we let $\wt h:= h\circ \phi + Q\log |\phi'|$ (with $Q$ as in~\eqref{eqn-lqg-coord}) then a.s.\ $\frk d_h(\phi(z) , \phi(w)) = \frk d_{\wt h}(z,w)$ for each $z,w\in \wt D$.
\end{lem}
\begin{proof}
This follows since all of the quantities involved in the definition of the $\op{QLE}(8/3,0)$ processes $\wt\Gamma^{z,w}$ used to define the metric are preserved under coordinate changes as in the statement of the lemma.
\end{proof}
 
Finally, we give conditions under which $\frk d_h$ extends by continuity to a subset of $\bdy D$ (Definition~\ref{def-cont}). 
We note that the above definition a priori only defines $\frk d_h$ on the interior of $D$.

\begin{lem}
\label{prop-metric-bdy}
Suppose we are in the setting above and that $D$ is simply connected, with simple boundary. Let $A\subset \bdy D$ be a connected set and suppose that there is an open set $U\subset D$ such that $A\subset \bdy U$, $A$ lies at positive distance from $\bdy D\setminus \bdy U$, and the law of $h|_U$ is absolutely continuous with respect to the law of the corresponding restriction of an embedding into $D$ of a quantum disk with fixed or random boundary length and area. Then $\frk d_h$ extends by continuity (with respect to the Euclidean topology on $D\cup A$) to a metric on $D\cup A$ which induces the Euclidean topology on $D\cup A$.
\end{lem}
\begin{proof}
In the case when $h$ is in fact an embedding into $D$ of a quantum disk,~\cite[Corollary 1.5]{lqg-tbm2} implies that $(D , \frk d_h)$ is isometric to the Brownian disk. Since the Brownian disk has the topology of a \emph{closed} disk~\cite{bet-disk-tight,bet-mier-disk} the Brownian disk metric extends by continuity to the boundary (with respect to the disk topology), hence $\frk d_h$ extends by continuity to $\bdy D$. In general, the first assertion of the lemma shows that $\frk d_{h|_U}$ extends by continuity to a metric on $U \cup A$ which induces the Euclidean topology on $U\cup A$. By this and Lemma~\ref{prop-internal-equal}, the same is true of $\frk d_h$. Note that the condition that $A$ lies at positive distance from $\bdy D\setminus \bdy U$ is used to avoid worrying about the behavior of $\frk d_{h|_U}$ near $\bdy U$. 
\end{proof}

The following lemma shows that one can also extend the metric to a boundary point where the field has a log singularity, such as the first marked point of a quantum wedge.

\begin{lem}
\label{prop-metric-bdy-wedge}
Suppose $Q$ is as in~\eqref{eqn-lqg-coord}, $D\subset \BB C$ is simply connected with simple boundary, $x,y\in \bdy D$, and $h$ is either an embedding into $(D ,x,y)$ of an $\alpha$-quantum wedge for $\alpha \leq Q$ or an embedding into $(D,x,y)$ of a single bead of an $\alpha$-quantum wedge with area $\frk a >0$ and left/right boundary lengths $\frk l^- , \frk l^+ >0$ for $\alpha \in (Q , Q + \sqrt{2/3})$. 
Then a.s.\ $\frk d_h$ extends by continuity to $\bdy D \setminus \{y\}$ (in the first case) or $\bdy D$ (in the second case), where in each case we use the Euclidean topology on $\ol D$ in Definition~\ref{def-cont}.
\end{lem}

We note that in the case when the size of the log singularity is smaller than 2 (e.g., if $\alpha < 2$), one can give a short proof of Lemma~\ref{prop-metric-bdy-wedge},  which does not use Appendix~\ref{sec-quantum-diam}, using the fact that a GFF-type distribution has an $\alpha$-log singularity at a typical point sampled from its $\alpha$-LQG boundary length measure (see, e.g.,~\cite[Lemma A.7]{wedges}). We give a longer argument which works for all log singularities of size $\leq Q$. 

\begin{proof}[Proof of Lemma~\ref{prop-metric-bdy-wedge}]
Since the laws of the particular choices of $h$ discussed in the second statement are locally absolutely continuous with respect to the law of an embedding into $D$ of a quantum disk away from $x$ and $y$, it follows that the metric $\frk d_h$ in each of these cases extends by continuity to $\bdy D\setminus \{x,y\}$. In the case of a single bead of a thin wedge for $\alpha \in (Q,Q+\sqrt{2/3})$, the law of our given distribution $h$ is locally absolutely continuous near each of its marked points with respect to the law of an embedding of a $\beta$-quantum wedge into $(D,x,y)$ for $\beta < Q$. By this and Lemma~\ref{prop-internal-equal}, it therefore suffices to show that if $(D, h , x , y)$ is a $\beta$-quantum wedge for $\beta \leq Q$, then $\frk d_h$ (which we know is defined in $\ol D\setminus \{x,y\}$) extends by continuity to $\ol D\setminus \{y \}$. 

By making a conformal change of coordinates and applying Lemma~\ref{prop-metric-coord}, we can assume without loss of generality that $(D,x,y) = (\BB S , +\infty,-\infty)$ where $\BB S = \BB R \times (0,\pi)$ is the infinite horizontal strip.  We seek to extend $\frk d_h$ from $\ol{\BB S}$ to $\ol{\BB S} \cup \{+\infty\}$ in such a way that the topology on $\BB S\cup \{+\infty\}$ induced by the extension is the same as the topology obtained by conformally mapping $\BB S$ to $\BB H$ in such a way that $-\infty \mapsto \infty$.
To do this it suffices to show that $\op{diam}\left( [k,+\infty) \times [0,\pi] \right) \rta 0$ as $k \rta\infty$. Indeed, this together with the triangle inequality and Cauchy's convergence criterion shows that $\lim_{w\rta+\infty} \frk d_h(z,w)$ exists for each $z \in \ol{\BB S}$, and we can then define $\frk d_h(z,+\infty)$ to be this limit. 

The reason for parameterizing by $\BB S$ is that~\cite[Remark~4.6]{wedges} implies that (after possibly applying a translation) the distribution $h$ can be described as follows. If $\beta <Q$, let $X : \BB R\rta \BB R$ be the process such that for $s \geq 0$, $X_s = B_{2s} - (Q-\beta) s$, where $B$ is a standard linear Brownian motion with $B_0 = 0$; and for $s < 0$, $X_s = \wh B_{-2s} + (Q-\beta) s$, where $\wh B$ is an independent standard linear Brownian motion with $\wh B_0 = 0$ conditioned so that $\wh B_{2t} + (Q-\beta)t \geq 0$ for all $t\geq 0$. If $\beta = Q$, instead let $X : \BB R\rta \BB R$ be the process such that $\{X_s\}_{s\geq 0}$ is $-1$ times a 3-dimensional Bessel process started from $0$ and $\{X_{-s}\}_{s \leq 0}$ is an independent standard linear Brownian motion started from 0. 
Then $h = h^0 + h^\dagger$, where $h^0$ is the function on $\BB S$ such that $h^0(z) = X_s$ for $z\in \{s\} \times (0,\pi)$ for each $s\in\BB R$; and $h^\dagger$ is a random distribution independent from $h^0$ whose law is that of the projection of a free boundary GFF on $\BB S$ onto the space of functions on $\BB S$ whose average over every segment $\{s\} \times (0,\pi)$ is 0.  

For $k \in \BB Z$, let $S_k := [k , k+1]\times [0,\pi]$ and $S_k' := [k-1,k+2]\times [0,\pi]$.  
By the scaling property of the $\sqrt{8/3}$-LQG metric~\cite[Lemma 2.2]{lqg-tbm2} and since $h^\dagger =h - h^0$, if we set $S = [-1,2] \times [0,\pi]$ then 
\eqb \label{eqn-h^dagger-dist}
\op{diam} \left( S_0  ; \frk d_{h^\dagger|_{S_0'}} \right) \leq   \exp\left( \frac{1}{\sqrt 6} \sup_{t\in [-1,2]} |X_t| \right) \op{diam} \left( S_1  ; \frk d_{h |_{S_1'}} \right)  .
\eqe  
The law of $h^\dagger$ does not depend on $\alpha$ and by the above description of the law of $X$, the first factor on the right in~\eqref{eqn-h^dagger-dist} has finite moments of all positive orders. By Lemma~\ref{lem-circle-avg-ball-diam}, the second factor on the right in~\eqref{eqn-h^dagger-dist} has a finite moment of some positive order in the special case when $\gamma=\sqrt{8/3}$. By H\"older's inequality, we find that there is a universal constant $p \in (0,1]$ such that 
\eqb \label{eqn-h^dagger-mean}
c :=  \BB E\left[ \op{diam} \left( S_0  ; \frk d_{h^\dagger|_{S_0'}} \right)^p  \right]  < \infty .
\eqe  
By~\eqref{eqn-h^dagger-mean}, the translation invariance of the law of $h^\dagger$, the independence of $h^\dagger$ and $X$, and the scaling property of the $\sqrt{8/3}$-LQG metric~\cite[Lemma 2.2]{lqg-tbm2}, we infer that for each $k\in\BB N$, 
\eqb \label{eqn-h-rect-mean}
\BB E\left[ \op{diam} \left( S_k  ; \frk d_{h |_{S_k'}} \right)^p  \right]  \leq c \BB E\left[  \exp\left( \frac{p}{\sqrt 6} \sup_{s \in [k-1,k+2]} X_s \right) \right] .
\eqe 
By summing over all $j \geq k$ and using that $x\mapsto x^p$ is concave, hence subadditive, we get that 
\eqb \label{eqn-h-rect-sum}
\BB E\left[ \op{diam} \left( [k ,+\infty) \times [0,\pi] ; \frk d_h    \right)^p  \right]  \leq c \sum_{j=k}^\infty \BB E\left[  \exp\left( \frac{p}{\sqrt 6} \sup_{s \in [j-1,j+2]} X_s \right) \right]  .
\eqe 
Recalling that $X|_{[0,\infty)}$ is a Brownian motion with negative drift if $\beta < Q$ or $-1$ times a 3-dimensional Bessel process if $\beta = Q$, we find that the right side of~\eqref{eqn-h-rect-sum} a.s.\ tends to 0 as $k\rta+\infty$, as required. In the case of a 3-dimensional Bessel process, the finiteness of the sum can be seen by using the Gaussian tail bound to compare $\sup_{s \in [j-1,j+2]} X_s$ to $X_j$ and then using the explicit formula for the transition density of such a process given on~\cite[Page 446]{revuz-yor}.
\end{proof}

\section{The Brownian disk}
\label{sec-bd}

For the proofs of our main results, we will require several facts about the Brownian disk, which was originally introduced in~\cite{bet-mier-disk}. We collect these facts in this section.

\subsection{Brownian disk definition}
\label{sec-brownian-disk}

Fix $a , \ell > 0$. Here we will give the definition of the Brownian disk with area $a$ and perimeter $\ell$, following~\cite{bet-mier-disk}.  Let $X : [0,a] \rta [0,\infty)$ be a standard Brownian motion started from $\ell$ and conditioned to hit $0$ for the first time at time $a$ (such a Brownian motion is defined rigorously in, e.g.,~\cite[Section~2.1]{bet-mier-disk}). 
For $s,t\in [0,a]$, set
\eqb \label{eqn-d_X}
d_X(s,t) := X_s + X_t - 2\inf_{u \in [s\wedge t , s\vee t]} X_u .
\eqe  
As explained in~\cite[Section~2.1]{bet-mier-disk}, $d_X$ defines a pseudometric on $[0,a]$ and the quotient metric space $[0,a] / \{d_X = 0\}$ is a forest of continuum random trees, indexed by the excursions of $X$ away from its running infimum.

Conditioned on $X$, let $Z^0$ be the centered Gaussian process with
\eqb \label{eqn-Z^0-def}
\op{Cov}(Z_s^0 ,Z_t^0 ) = \inf_{u\in [s\wedge t , s\vee t]} \left( X_u - \inf_{v \in [0,u]} X_v \right) , \quad s,t\in [0,a] .
\eqe 
One can readily check using the Kolmogorov continuity criterion that $Z^0$ a.s.\ admits a continuous modification which is $\alpha$-H\"older continuous for each $\alpha <1/4$. For this modification we have $Z_s^0 = Z_t^0$ whenever $d_X(s,t) = 0$, so $Z^0$ defines a function on the continuum random forest $[0,a] / \{d_X = 0\}$. 
 
Let $\frk b$ be $\sqrt 3$ times a Brownian bridge from $0$ to $0$ independent from $(X,Z)$ with time duration $\ell$. 
For $r \in [0,\ell]$, let $T_r := \inf\left\{t \geq 0 \,:\, X_t = \ell-r \right\}$ and for $t \in [0,a]$, let $T^{-1}(t) := \sup\left\{r \in [0,\ell] \,:\, T_r \leq t \right\}$. Set
\eqb \label{eqn-bd-label-def}
Z_t := Z_t^0 +  \frk b_{T^{-1}(t)} .
\eqe 
We view $[0,a]$ as a circle by identifying~$0$ with~$a$ and for $s,t\in [0,a]$ we define $\ul Z_{s,t}$ to be the minimal value of~$Z$ on the counterclockwise arc of $[0,a]$ from $s$ to $t$. 
For $s,t \in [0,a]$, define
\eqb \label{eqn-d_Z}
d_Z \left(s,t \right) = Z_s + Z_t - 2\left( \ul Z_{s,t} \vee \ul Z_{t,s} \right) .
\eqe 
Then $d_Z$ is not a pseudometric on $[0,a]$, so we define
\eqb \label{eqn-dist0-def}
d_{a,\ell}^0(s,t) = \inf \sum_{i=1}^k d_Z(s_i , t_i)
\eqe 
where the infimum is over all $k\in\BB N$ and all $2k+2$-tuples $( t_0, s_1 , t_1 , \dots , s_{k } , t_{k } , s_{k+1}) \in [0,a]^{2k+2}$ with $t_0 = s$, $s_{k+1} = t$, and $d_X(t_{i-1} , s_i) = 0$ for each $i\in [1,k+1]_{\BB Z}$. In other words, $d_{a,\ell}^0$ is the largest pseudometric on $[0,a]$ which is at most $d_Z$ and is zero whenever $d_X$ is  $0$. 

The \emph{Brownian disk} with area $a$ and perimeter $\ell$ is the quotient space $ \op{BD}_{a,\ell}  = [0,a]/\{d_{a,\ell}^0 = 0\}$ equipped with the quotient metric, which we call $d_{a,\ell}$. It is shown in~\cite{bet-mier-disk} that $(\op{BD}_{a,\ell} , d_{a,\ell}) $ is a.s.\ homeomorphic to the closed disk. 

We write $\BB p : [0,a] \rta \op{BD}_{a,\ell}$ for the quotient map. The pushforward $\mu_{a,\ell}$ of Lebesgue measure on $[0,a]$ under $\BB p$ is a measure on $\op{BD}_{a,\ell}$ with total mass $a$, which we call the \emph{area measure} of $\op{BD}_{a,\ell}$. 
The \emph{boundary} of $\op{BD}_{a,\ell}$ is the set $\bdy\op{BD}_{a,\ell} = \BB p \left(\{T_r \,:\, r \in [0,\ell] \} \right)$. We note that $\op{BD}_{a,\ell}$ has a natural orientation, obtained by declaring that the path $t\mapsto \BB p(t)$ traces $\bdy \op{BD}_{a,\ell}$ in the counterclockwise direction. 
The pushforward $\nu_{a,\ell}$ of Lebesgue measure on $[0,\ell]$ under the composition $r\mapsto \BB p(T_r)$ is called the \emph{boundary measure} of $\op{BD}_{a,\ell}$.  

By~\cite[Corollary 1.5]{lqg-tbm2}, the law of the metric measure space $(\op{BD}_{a,\ell} , d_{a,\ell} , \mu_{a,\ell}, \nu_{a,\ell})$ is the same as that of the $\sqrt{8/3}$-LQG disk with area $a$ and boundary length $\ell$, equipped with its $\sqrt{8/3}$-LQG area measure and boundary length measure. 
  
\begin{defn} \label{def-length-disk}
Let $\ell > 0$ and let $T_\ell$ be a random variable with the law of the first time a standard linear Brownian motion hits $-\ell$. We write $(\op{BD}_{*,\ell} , d_{*,\ell}) := (\op{BD}_{T_\ell ,\ell} , d_{T_\ell ,\ell})$, so that $\op{BD}_{*,\ell}$ is a Brownian disk with random area. Note that in this case the corresponding function $X$ defined above has the law of a standard linear Brownian motion started from $\ell$ and stopped at the first time it hits $0$.
\end{defn}

It is often convenient to work with a random-area Brownian disk rather than a fixed-area Brownian disk since the encoding function $X$ has a simpler description in this case. 

\subsection{Area, length, and distance estimates for the Brownian disk}
\label{sec-disk-estimate}

In this subsection we will prove some basic estimates relating distances, areas, and boundary lengths for the Brownian disk which are needed for the proofs of our main results. These estimates serve to quantify the intuition that for Brownian surfaces we have
\eqbn
\op{Area} \approx \op{Length}^{1/2} \approx \op{Distance}^{1/4} .
\eqen
In other words, a subset of the Brownian disk with area $\delta$ typically has boundary length approximately $\delta^{1/2}$ and diameter approximately $\delta^{1/4}$.  

Throughout this section, for $x,y\in \bdy \op{BD}_{a,\ell}$ we write $[x,y]_{\bdy\op{BD}_{a,\ell}}$ for the counterclockwise arc of $\bdy \op{BD}_{a,\ell}$ from $x$ to $y$ (i.e., in the notation of Section~\ref{sec-brownian-disk}, the arc traced by the boundary path $r\mapsto \BB p(T_r)$ between the times when it hits $x$ and $y$).  Our first estimate tells us that distances along the boundary are almost $1/2$-H\"older continuous with respect to boundary length (in the same way that a standard Brownian motion on a compact time interval is almost $1/2 $-H\"older continuous).

\begin{lem} \label{prop-disk-bdy-holder}
Let $a > 0$ and $\ell > 0$ and let $(\op{BD}_{a,\ell} , d_{a,\ell})$ be a Brownian disk with area $a$ and perimeter $\ell$. For each $\zeta>0$, there a.s.\ exists $C>0$ such that for each $x,y\in \bdy \op{BD}_{a,\ell}$, we have 
\eqb \label{eqn-disk-bdy-holder}
d_{a,\ell} \left(x,y\right) \leq C \nu_{a,\ell} \left( [x,y]_{\bdy \op{BD}_{a,\ell}} \right)^{1/2}  \left( \left| \log \nu_{a,\ell} \left( [x,y]_{\bdy \op{BD}_{a,\ell}} \right) \right| + 1 \right)^{\frac74 + \zeta} .
\eqe 
The same holds for a random-area disk with fixed boundary length as in Definition~\ref{def-length-disk}. In this latter case, if we let $C$ be the smallest constant for which~\eqref{eqn-disk-bdy-holder} is satisfied, then for $A >1$, $\BB P[C > A]$ decays faster than any negative power of $A$.
\end{lem}

The idea of the proof of Lemma~\ref{prop-disk-bdy-holder} is as follows. Recall that $\bdy \op{BD}_{a,\ell}$ is the image under the quotient map $[0,a] \rta \op{BD}_{a,\ell}$ of the set of times when the encoding function $X$ attains a running infimum relative to time $0$ (i.e.\ the image of $r\mapsto T_r$). The pair $(X , Z^0)$ restricted to each such excursion evolves as a Brownian excursion together with the head of the Brownian snake driven by this excursion, so we can use tail bounds for the Brownian snake~\cite[Proposition~14]{serlet-snake} to bound the maximum of the restriction of $Z^0$ to each such excursion in terms of its time length. On the other hand, the excursions with unusually long time length are distributed according to a Poisson point process, so there cannot be too many such excursions in $[T_{r_1} , T_{r_2}]$ for any fixed $0 < r_1<r_2<\ell$. We use~\eqref{eqn-dist0-def} to construct a path in $\op{BD}_{a,\ell}$ between the boundary points corresponding to $T_{r_1}$ and $T_{r_2}$ which skips all of the long excursions. See Figure~\ref{fig-bdy-holder} for an illustration. A similar argument is used in~\cite[Section 7.4]{bet-disk-tight} to prove that the Hausdorff dimension of $\bdy \op{BD}_{a,\ell}$ is at most 2, but we need a somewhat more precise estimate so we will give a self-contained proof.

We will use the elementary estimate
\eqb \label{eqn-poisson-tail}
\BB P \left[X \geq x \right] \leq \frac{e^{-\lambda} (e \lambda)^x}{x^x} ,\quad \forall x > \lambda \quad \text{for $X \sim \op{Poisson}(\lambda)$}. 
\eqe

\begin{figure}[ht!]
 \begin{center}
\includegraphics[scale=.8]{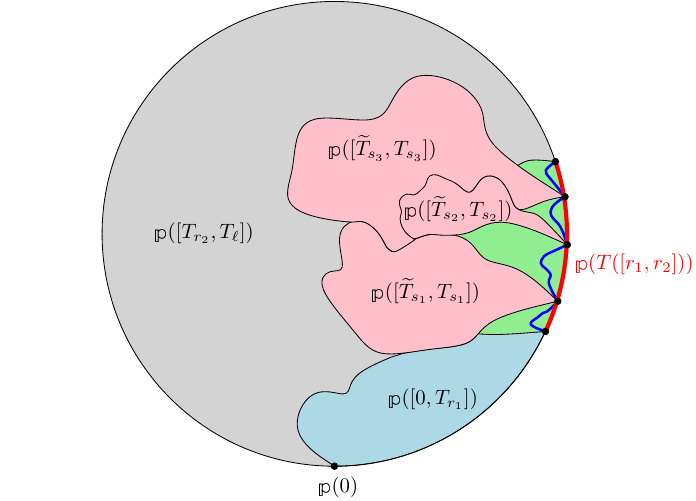}
\vspace{-0.01\textheight}
\caption[Illustration of the proof of Lemma~\ref{prop-disk-bdy-holder}]{Illustration of the proof of Lemma~\ref{prop-disk-bdy-holder}. Shown is the Brownian disk $\op{BD}_{*,\ell}$ and an arc $\BB p(T([r_1,r_2])) \subset \bdy \op{BD}_{*,\ell}$ (red) whose diameter we are trying to bound. To construct a path from $\BB p(T_{r_1})$ to $\BB p(T_{r_2})$, we consider a path (shown in blue) consisting of a concatenation of segments which avoid the large excursions of $t\mapsto \BB p(t)$ away from $\BB p(T([r_1,r_2]))$  (pink). The number of such excursions is bounded using~\eqref{eqn-few-big-jump} and the lengths of each of the segments between excursions is bounded using~\eqref{eqn-d_Z} and~\eqref{eqn-snake-sup}.}\label{fig-bdy-holder}
\end{center}
\end{figure}

\begin{proof}[Proof of Lemma~\ref{prop-disk-bdy-holder}]
First fix $\ell\in\BB N$ and let~$\op{BD}_{*,\ell}$ be a random-area Brownian disk with boundary length~$\ell$.  We use the notation introduced at the beginning of this subsection with $a = T_\ell$ as in Definition~\ref{def-length-disk}. Recall in particular that $T_r$ for $r  \in [0,\ell]$ denotes the first time $X$ hits $-r$, and that $T([0,\ell])$ is the pre-image of $\bdy \op{BD}_{a,\ell}$ under the quotient map. 
\medskip
 
\noindent\textit{Step 1: bounds for the number of big excursions.} 
We will first bound the number of large time intervals which do not contain a point mapped to $\bdy \op{BD}_{a,\ell}$ by the quotient map, equivalently, one of the times $T_r$. For $r>0$, let $\wt T_r := \sup_{s < r} T_s$. Note that $d_X(\wt T_r , T_r) = 0$, with $d_X$ as in~\eqref{eqn-d_X}. Then the intervals $[\wt T_r , T_r]$ for $r \in [0,\ell]$ with $T_r > \wt T_r$ are precisely the excursion intervals for $X$ away from its running infimum. 

The time lengths of the excursions of $X$ away from its running infimum, parameterized by minus the running infimum of $X$, have the law of a Poisson point process on $\BB R$ with It\^o measure $\pi^{-1/2} t^{-3/2} \, dt$ (see, e.g.~\cite[Theorem~2.4 and Proposition~2.8, Section~XII]{revuz-yor}). Hence for $0\leq r_1 \leq r_2 \leq \ell$ and $A>0$, the law of $\# \left\{r \in [r_1,r_2] \,:\, T_r - \wt T_r >  A \right\}$ is Poisson with mean $2 \pi^{-1/2} (r_2-r_1) A^{-1/2} $. %Integrate[ Pi^(-1/2) *t^(-3/2), {t, A, Infinity}] 
 By~\eqref{eqn-poisson-tail} and the union bound, for each $k\in \BB N$ it holds except on an event of probability decaying faster than any exponential function of $k$ that 
\eqb \label{eqn-few-big-jump}
 \#\left\{ r\in [(j-1) 2^{-k } ,  j 2^{-k}] \,:\, T_r - \wt T_r > \ell^2 2^{-2k} \right\} \leq k ,\quad \forall j \in [1,2^k]_{\BB Z}. 
\eqe  

For $n\in\BB N$, let $\mcl R_n$ be the set of $r\in [0, \ell]$ for which $T_r - \wt T_r \in [(n+1)^{-1} \ell^2  , n^{-1} \ell^2  ]$. Then $\# \mcl R_n$ is Poisson with mean $2\pi^{-1/2}   ( (n+1)^{1/2} - n^{1/2}) \asymp n^{-1/2}$. Hence, by~\eqref{eqn-poisson-tail}, except on an event of probability decaying faster than any power of $n$, we have
\eqb \label{eqn-total-jump}
\# \mcl R_m \leq  \log(m+1)  ,\quad \forall m \geq n .
\eqe
\medskip
 
\noindent\textit{Step 2: variation of $Z$ over the complement of the large excursion intervals.}
We now consider how much the label process $Z$ can vary if we ignore the big excursions of $X$ away from its running infimum. In particular, we will show that 
except on an event of probability decaying faster than any power of $n$,  
\eqb \label{eqn-snake-sup}
\sup\left\{ |Z_s - Z_t|  \,:\, s,t\in    [T_{r_1}  ,T_{r_2}] \setminus  \bigcup_{m=1}^{n-1} \bigcup_{r\in \mcl R_m} [\wt T_r , T_r] \right\} \leq 4 \ell^{1/2}  n^{-1/4} (\log n)^{\frac34 + \zeta}  
\eqe 
simultaneously for each $0 \leq r_1 < r_2 \leq \ell$ with $r_2 - r_1 \leq \ell n^{-1/2}$. 
Recall from~\eqref{eqn-bd-label-def} that $  Z  = Z^0 + \frk b_{T^{-1}(\cdot)}$, where $Z^0$ is as in~\eqref{eqn-Z^0-def} and~$\frk b$ is equal to~$\sqrt 3$ times a Brownian bridge from~$0$ to~$0$ in time~$\ell$. We will bound $Z^0$ and $\frk b$ separately. 
  
If we condition on $T|_{[0,\ell]}$ then the conditional law of the processes\\
$\left\{(X - X_{\wt T_r} , Z^0)|_{[\wt T_r , T_r]} \,:\, r \in [0,\ell] ,\: T_r > \wt T_r \right\}$ is described as follows. These processes for different choices of $r$ are conditionally independent given $T|_{[0,\ell]}$; the conditional law of each $(X - X_{\wt T_r})|_{[\wt T_r , T_r]}$ is that of a Brownian excursion with time length $T_r - \wt T_r$; and each $Z^0|_{[\wt T_r , T_r]}$ is the head of the Brownian snake driven by $(X - X_{\wt T_r})|_{[\wt T_r , T_r]}$. By the large deviation estimate for the head of a Brownian snake driven by a standard Brownian excursion~\cite[Proposition~14]{serlet-snake} and two applications of Brownian scaling, we find that for each $r\in [0,\ell]$ and each $A>1$,  
\eqb \label{eqn-snake-tail}
\BB P \left[ \sup_{t\in [\wt T_r , T_r]} |Z^0_s - Z^0_t| > A (T_r - \wt T_r)^{1/4} \,|\,  T|_{[0,\ell]}  \right] \leq \exp\left( - \frac{3}{2} (1+o_A(1))  A^{4/3}  \right)
\eqe
at a universal rate as $A \to \infty$. 

If the event in~\eqref{eqn-total-jump} occurs for some $n\in\BB N$ (which we emphasize is determined by $T|_{[0,\ell]}$), then~\eqref{eqn-snake-tail} implies that with
\begin{align}  \label{eqn-snake-event}
E_n :=  \left\{ \exists r\in  \bigcup_{m=n}^\infty \mcl R_m \:\text{with}\: \sup_{t\in [\wt T_r , T_r]} |  Z^0_t| >    \ell^{1/2} n^{-1/4} (\log n)^{\frac{3}{4} +\zeta}   \right\}
\end{align}
we have
\begin{align} \label{eqn-snake-tail'}
\BB P\left[ E_n \,|\, T|_{[0,\ell]}  \right] 
&\leq \sum_{m=n}^\infty  \log(m+1) \exp\left( - \frac{3}{2}  (1+o_n(1))(\log n)^{1 + \frac{4}{3}\zeta}  (m/n)^{1/3}      \right)  \notag\\
&= o_n(n^{-p}) ,\quad \forall p > 0 .
\end{align}

Since $\frk b$ is equal to $\sqrt 3$ times a Brownian bridge, by a straightforward Gaussian estimate it holds except on an event of probability decaying faster than any power of~$n$ that whenever $0 \leq r_1\leq r_2 \leq \ell$ with $r_2-r_1 \leq \ell n^{-1/2}$, 
\eqb \label{eqn-bridge-tail}
\sup_{\rho_1,\rho_2 \in [r_1,r_2]} |\frk b_{\rho_1} - \frk b_{\rho_2}| \leq  \ell^{1/2} n^{-1/4} (\log n)^{3/4}.
\eqe
(In fact, we could replace the power $3/4$ of $\log n$ above with any power strictly larger than $1/2$.)  If the event $E_n$ in~\eqref{eqn-snake-event} does not occur, then for any such $r_1 , r_2 \in [0,\ell]$ and any $t \in [T_{r_1}, T_{r_2}]$ for which $t \notin \bigcup_{m=1}^{n-1} \bigcup_{r \in \mcl R_m} [\wt T_r , T_r]$, we have either $Z_t^0 = 0$ or $t\in [\wt T_r , T_r]$ for some $r\in \bigcup_{m=n}^\infty \mcl R_m$. 
In either case $|Z_t^0| \leq \ell^{1/2} n^{-1/4} (\log n)^{3/4 + \zeta}$.
Hence~\eqref{eqn-snake-tail'}, \eqref{eqn-bridge-tail}, and the triangle inequality together imply~\eqref{eqn-snake-sup}. 
\medskip

\noindent\textit{Step 3: conclusion.} 
Suppose now that $k\in\BB N$,~\eqref{eqn-few-big-jump} occurs and~\eqref{eqn-snake-sup} occurs with $n =2^{2k}$.  Let $\BB p: [0,T_\ell] \rta \op{BD}_{*,\ell}$ be the quotient map.  Note that $\BB p(T([(j-1)2^{-k} \ell , j2^{-k} \ell ]))$ corresponds to the counterclockwise segment of the boundary of the disk of length $2^{-k} \ell$ which connects $\BB p(T_{(j-1)2^{-k} \ell})$ to $\BB p(T_{j 2^{-k} \ell})$.  We will use~\eqref{eqn-dist0-def} to bound the diameter of $\BB p(T([(j-1)2^{-k} \ell , j2^{-k} \ell ]))$. 

Fix $j\in[1,2^k]_{\BB Z}$ and $r_1 , r_2 \in [ (j-1) 2^{-k} \ell  ,  j 2^{-k} \ell ]$ with $r_1 < r_2$. 
By~\eqref{eqn-few-big-jump}, there are at most $k$ values of $r$ for which $[\wt T_r , T_r] \subset [T_{r_1}, T_{r_2}]$ and $T_r - \wt T_r \geq \ell^2 2^{-2k}$. 
Let $t_0  = s_1  = T_{r_1}$, $t_k = s_{k+1}  =T_{r_2}$, and for $i \in [2,k ]_{\BB Z}$ let $[t_{i-1} , s_i]$ be the $i$th interval $[\wt T_r , T_r] \subset [T_{r_1} , T_{r_2}]$ with $T_r - \wt T_r \geq \ell^2 2^{-2k}$, counted from left to right; or $t_{i-1} = s_i = T_{r_2}$ if there fewer than $i$ such intervals. Then $d_X(t_{i-1} , s_i) = 0$ for each $i\in [1,k+1]_{\BB Z}$ and (by~\eqref{eqn-few-big-jump}) each $r\in [r_1,r_2]\cap \bigcup_{m = 1 }^{2^{2k}} \mcl R_m $ satisfies $[\wt T_r , T_r] = [t_{i-1}, s_i]$ for some $i\in [2,k]_{\BB Z}$. Hence each of the intervals $(s_i , t_i)$ for $i\in [1,k]_{\BB Z}$ is disjoint from $\bigcup_{m=1 }^{2^{2k}} \bigcup_{r\in \mcl R_m} [\wt T_r , T_r] $.
   
By~\eqref{eqn-dist0-def} and~\eqref{eqn-snake-sup},  
\eqbn 
 d_{*,\ell}^0 \left(  T_{r_1} , T_{r_2} \right) \leq 2 \sum_{i=1}^k \sup_{s,t \in [s_i , t_i]} |Z_s - Z_t| \preceq \ell^{1/2} 2^{-k/2}  k^{\frac74 + \zeta}    , 
\eqen
with universal implicit constant, where in the first inequality we recall the definition~\eqref{eqn-d_Z} of $d_Z$. Since our choice of $r_1 , r_2 \in  [ (j-1) 2^{-k}  ,  j 2^{-k}]$ was arbitrary,
\eqb \label{eqn-disk-interval-dist}
\op{diam} \left( \BB p \left(T\left(\left[(j-1)2^{-k} \ell , j2^{-k} \ell\right]\right)\right) ; d_{*,\ell} \right) \preceq \ell^{1/2} 2^{-k/2}  k^{\frac74 + \zeta}     
\eqe 
where here we recall the construction of $d_{*,\ell}$ from $d_{*,\ell}^0$. 

By the Borel-Cantelli lemma, there a.s.\ exists $k_0 \in \BB N$ such that for $k\geq k_0$,~\eqref{eqn-few-big-jump} occurs and~\eqref{eqn-snake-sup} occurs with $n =2^{2k}$. In fact, if we take $k_0$ to be the smallest integer for which this is the case, then $\BB P[2^{k_0} > A]$ decays faster than any negative power of $A$. Suppose given $k\geq k_0$ and $x,y \in \bdy \op{BD}_{*,\ell}$ with $2^{-k-1} \leq \nu_{*,\ell}([x,y]_{\bdy \op{BD}_{*,\ell}}) \leq 2^{-k }$. By~\eqref{eqn-disk-interval-dist} and the triangle inequality,   
\eqbn
 d_{*,\ell} \left(x,y\right) \preceq \nu_{*,\ell} \left( [x,y]_{\bdy \op{BD}_{*,\ell}} \right)^{1/2}  \left( \left| \log \nu_{*,\ell} \left( [x,y]_{\bdy \op{BD}_{*,\ell}} \right)  \right| + 1 \right)^{\frac74 + \zeta}
 \eqen
with the implicit constant depending only on $\ell$. Taking $C$ to be equal to $2^{k_0}$, say, times this implicit constant (to deal with the case when $\nu_{*,\ell}([x,y]_{\bdy \op{BD}_{*,\ell}}) >2^{-k_0}$), we get the desired continuity estimate in the case of a random-area disk.

The fixed area case follows since for any $a>0$, the law of $(\op{BD}_{a,\ell} , d_{a,\ell})$ is locally absolutely continuous with respect to the law of  $(\op{BD}_{*, 2\ell} , d_{*, 2\ell})$ conditioned on the positive probability event that $T_\ell > 2a$ (c.f.~\cite[Section~2.1]{bet-mier-disk}). 
\end{proof}

Next we need an estimate for the areas of metric balls, which is a straightforward consequence of the H\"older continuity of $Z$ and the upper bound for areas of metric balls in the Brownian map~\cite[Corollary~6.2]{legall-geodesics}. 

\begin{lem} \label{prop-ball-size}
Let $a ,\ell > 0$ and let $(\op{BD}_{a,\ell} , d_{a,\ell})$ be a Brownian disk with area $a$ and perimeter $\ell$. For each $u \in (0,1)$, there a.s.\ exists $C> 1$ such that the following is true. For each $\delta \in (0,1)$ and each $z\in \op{BD}_{a,\ell}$, we have
\eqb \label{eqn-ball-size-lower}
  \mu_{a,\ell} \left( B_\delta(z ; d_{a,\ell} ) \right) \geq C^{-1} \delta^{4+u}
\eqe 
and for each $z\in \op{BD}_{a,\ell}$ with $B_\delta(z ; d_{a,\ell} ) \cap \bdy \op{BD}_{a,\ell} = \emptyset$, we have
\eqb \label{eqn-ball-size-upper}
  \mu_{a,\ell} \left( B_\delta(z ; d_{a,\ell} ) \right) \leq C \delta^{4-u} .
\eqe 
\end{lem}
\begin{proof}
It is easy to see from the Kolmogorov continuity criterion that the process $Z$ used in the definition of the Brownian disk is a.s.\ H\"older continuous with exponent $(4+u)^{-1}$. By this and~\eqref{eqn-d_Z}, there a.s.\ exists a random $C > 0$ such that 
\eqb \label{eqn-disk-holder}
d_{a,\ell} \left(\BB p(s) , \BB p(t) \right) \leq C  |t-s|^{(4+u)^{-1}}  ,\quad \forall s,t\in [0,a].
\eqe 
Since the quotient map $\BB p$ pushes forward Lebesgue measure on $[0,a]$ to $\mu_{a,\ell}$, we infer that for each $t\in [0,\ell]$ and each sufficiently small $\delta \in (0,1)$ (how small depends on $C$, $u$, and $a$),
\eqbn
\mu_{a,\ell} \left( B_\delta(\BB p(t) ; d_{a,\ell} ) \right) \geq C^{-1} \delta^{4+u} .
\eqen
Upon shrinking $C$, this implies~\eqref{eqn-ball-size-lower}.

We now prove~\eqref{eqn-ball-size-upper}.
By local absolute continuity of the process $(Z , X,\frk b)$ for different choices of $a$ and $\ell$, it suffices to prove that for Lebesgue a.e.\ pair $(a,\ell) \in (0,\infty)^2$, there a.s.\ exists $C>0$ so that~\eqref{eqn-ball-size-upper} holds. Suppose to the contrary that this is not the case. Then there is a positive Lebesgue measure set $\mcl A \subset  (0,\infty)^2$ such that for each $(a,\ell) \in \mcl A$, it holds with positive probability that  
\eqb \label{eqn-bad-disk-balls}
\sup_{\delta >0 } \sup_{ z\in \op{BD}_{a,\ell} } \frac{\mu_{a,\ell}(B_\delta(z;d_{a,\ell}) )}{\delta^{4-u}} = \infty .
\eqe 

Choose $A > 0$ such that the projection of $\mcl A$ onto its first coordinate intersects $[0,A]$ in a set of positive Lebesgue measure and let $(\op{BM}_A,d_A)$ be a Brownian map with area $A$ and let $\mu_A$ be its area measure. By~\cite[Corollary~6.2]{legall-geodesics}, a.s.\ 
\eqb \label{eqn-tbm-balls}
\sup_{\delta >0 } \sup_{z\in \op{BM}_A} \frac{\mu_A(B_\delta(z;d_A) )}{\delta^{4-u}} < \infty .
\eqe 
Now fix $r>0$ and let $z_0 ,z_1 \in \op{BM}_A$ be sampled uniformly from $\mu_A$. By~\cite[Proposition~4.4]{tbm-characterization} (c.f.~\cite[Theorem 3]{legall-disk-snake}), the complementary connected component $D$ of $B_r(z_0 ; d_A)$ containing $z_1$, equipped with the internal metric $d_{A,D}$ induced by $d_A$, has the law of a Brownian disk if we condition on its area and boundary length. With positive probability, the area and boundary length of $D$ belong to $\mcl A$. Since each $d_{A,D}$-metric ball is contained in a $d_{A }$-metric ball with the same radius, we see that~\eqref{eqn-bad-disk-balls} contradicts~\eqref{eqn-tbm-balls}.
\end{proof}

We will next prove a lower bound for the amount of area near a boundary arc of given length. 

\begin{lem} \label{lem-interval-nghbd}
Let $a ,\ell > 0$ and let $(\op{BD}_{a,\ell} , d_{a,\ell})$ be a Brownian disk with area $a$ and perimeter $\ell$. For each $u > 0$, there a.s.\ exists $c > 0$ such that for each $\delta \in (0,1)$ and each $x,y\in \bdy \op{BD}_{a,\ell}$, we have
\eqb \label{eqn-interval-nghbd}
\mu_{a,\ell} \left( B_\delta \left( [x,y]_{\bdy\op{BD}_{a,\ell}} ; d_{a,\ell} \right) \right) \geq c \delta^{2 + u}  \nu_{a,\ell} \left(  [x,y]_{\bdy\op{BD}_{a,\ell}} \right) .
\eqe 
\end{lem}

The idea of the proof of Lemma~\ref{lem-interval-nghbd} is to prove a lower bound for the number of time intervals of the form $[(k-1) 2^{-2n} , k 2^{-2n}]$ for $k\in\BB Z$ whose images under the quotient map $\BB p$ intersect $ [x,y]_{\bdy\op{BD}_{a,\ell}}$, where $n$ is chosen so that $2^{-n} = \delta^{2 + o_\delta(1) }$. The images of these intervals are disjoint, and each such interval has $\mu_{a,\ell}$-mass $2^{-2n}$ and is contained in $B_\delta \left( [x,y]_{\bdy\op{BD}_{a,\ell}} ; d_{a,\ell} \right) $ by~\eqref{eqn-disk-holder}. The right side of~\eqref{eqn-interval-nghbd} will turn out to be $2^{-2n}$ times the number of such intervals.
Our lower bound for the number of time intervals will follow from the following elementary estimate for Brownian motion. 
 
\begin{lem} \label{lem-inf-count-lower}
Let $B$ be a standard linear Brownian motion started from $0$ and for $r>0$, let $T_r := \inf\left\{t > 0 \,:\, B_t = -r\right\}$. For $\delta>0$, let $N_\delta$ be the number of intervals of the form $[(k-1)\delta , k\delta]$ for $k\in\BB N$ which intersect $\{T_r \,:\, r\in [0,1]\}$. There is a universal constant $c_0 > 0$ such that for each $\delta  \in (0,1)$ and each $\zeta \in (0, (2\pi)^{-1/2})$ we have
\eqbn
\BB P \left[ N_\delta < \zeta \delta^{-1/2} \right] \leq  \exp\left( - \frac{c_0 }{ \zeta \delta^{1/2}} \right) .
\eqen
\end{lem}
\begin{proof}
Let $\rho_0 = 0$ and for $j \in\BB N$ inductively let $\rho_j$ be the smallest $r > \rho_{j-1}$ for which $T_r - T_{\rho_{j-1}} \geq \delta$.
Let $J$ be the largest $j \in\BB N$ for which $\rho_j \leq 1$. 
Each of the times $T_{\rho_j}$ for $j\in [1,J]_{\BB Z}$ lies in a distinct interval of the form $[(k-1)\delta , k\delta]$ for $k\in \BB N$. 
Hence $N_\delta \geq J$. 

By the strong Markov property for standard Brownian motion, the random variables $\rho_j -  \rho_{j-1} $ for $j\in\BB N$ are i.i.d.  Observe that $\rho_1 - \rho_0 = - \inf_{t\in [0,\delta]} B_t$ because if $r < -\inf_{t \in [0,\delta]} B_t$ then $T_r < \delta$ and if $r > -\inf_{t \in [0,\delta]} B_t$ then $T_r > \delta$.  Consequently, it follows that each of these random variables has the law of the absolute value of a Gaussian random variable with mean $0$ and variance $\delta $ (which has mean $(2/\pi)^{1/2} \delta^{1/2}$). 
By Hoeffding's inequality for sums of independent random variables with sub-Gaussian tails (see, e.g.,~\cite[Proposition~5.10]{vershynin-intro}), for $m \in \BB N$ and $R>0$,
\eqbn
\BB P \left[ \rho_m > R +  (2/\pi)^{1/2} \delta^{1/2} m  \right] 
\leq   \exp\left( - \frac{ c_1   R^2 }{ \delta  m  }  \right) 
\eqen 
with $c_1 > 0$ a universal constant. Therefore, for $\zeta \in (0, (2\pi)^{-1/2})$ we have
\eqbn
\BB P \left[ N_\delta <  \zeta \delta^{-1/2} \right] 
\leq \BB P \left[ J <    \zeta \delta^{-1/2} \right] 
= \BB P \left[ \rho_{\lfloor \zeta \delta^{-1/2}\rfloor}  > 1 \right] 
\leq  \exp\left( - \frac{c_0 }{ \zeta \delta^{1/2}} \right)  
\eqen
with $c_0 > 0$ as in the statement of the lemma.
\end{proof}

\begin{proof}[Proof of Lemma~\ref{lem-interval-nghbd}]
In light of Lemma~\ref{prop-ball-size} (applied with  $u/2$ in place of $u$), we can restrict attention to arcs $[x,y]_{\bdy \op{BD}_{a,\ell}}$ satisfying
\eqbn
 \nu_{a,\ell} \left(  [x,y]_{\bdy\op{BD}_{a,\ell}} \right) \geq \delta^{ 2-u/2 }.
 \eqen
 
We start by working with a random-area Brownian disk $(\op{BD}_{*,\ell} , d_{*,\ell})$ as in Definition~\ref{def-length-disk}. 
Fix a small parameter $v \in (0,u)$ to be chosen later, depending only on $u$. 
For $n,m\in\BB N$ with $n \geq  m$ and $j\in [1,2^m]_{\BB Z}$, let $N_{m,j}^n$ be the number of intervals of the form $[(k-1)2^{-2n} , k 2^{-2n}]$ for $k\in\BB N$ which intersect $\{T_r \,:\, r\in [ (j-1) 2^{-m} \ell , j 2^{-m} \ell ]\}$.   
Let $E_m $ be the event that $N_{m,j}^n \geq \lfloor \ell 2^{ (1-v) n  -m} \rfloor $ for each $n  \geq  m$ and all $j\in [1,2^m]_{\BB Z}$.  
By Lemma~\ref{lem-inf-count-lower} (applied with $\delta = 2^{2m-2n} \ell^{-2}$ and $\zeta = 2^{-v n}$) and scale and translation invariance,
\eqbn
\BB P \left[ N_{m,j}^n <  \lfloor \ell  2^{(1-v) n  -m} \rfloor   \right] \leq  \exp\left( - c_0 \ell 2^{(1+v) n-m}   \right) .
\eqen
By the union bound,
\eqbn
\BB P \left[ E_m^c \right] \leq  2^m \sum_{n= m}^\infty \exp\left( - c_0 \ell  2^{(1+v)n-m}    \right) \preceq 2^m \exp\left( -  c_0 \ell 2^{v m}  \right)
\eqen
with implicit constant depending only on $\ell$. By the Borel-Cantelli Lemma, a.s.\ there exists $m_0 \in \BB N$ such that $E_m$ occurs for each $m\geq m_0$. 

Let $C  >0$ be a random constant chosen so that~\eqref{eqn-disk-holder} holds for $\op{BD}_{*,\ell}$, with $v/100$ in place of $u$.
Suppose we are given $\delta > 0$ with $\delta^2 \in (0, 2^{- (1+2v)m_0  }]$ and $x,y\in \bdy \op{BD}_{a,\ell}$ with 
\eqb \label{eqn-interval-length-bdy}
\delta^{ \frac{2}{1+2v}} \leq \nu_{a,\ell} \left(  [x,y]_{\bdy \op{BD}_{a,\ell}} \right) \leq 2^{-m_0-1} .
\eqe 
Then we can choose
$n \geq m \geq   m_0  $ with 
\begin{align} \label{eqn-interval-nm}
&2^{-m +1  } \leq \nu_{*,\ell} \left(  [x,y]_{\bdy\op{BD}_{*,\ell}} \right) \leq 2^{-m +2} \notag\\
&\quad \op{and} \quad  2^{- (1-v) n} \leq C^{-100} \delta^2 \leq 2^{- (1-v)(n-1)} .
\end{align}

For some $j\in [1,2^m]_{\BB Z}$, the boundary arc $[x,y]_{\bdy \op{BD}_{*,\ell}}$ contains the image under the quotient map $\BB p$ of the set $\{T_r \,:\, r\in [ (j-1) 2^{-m} \ell , j 2^{-m} \ell ]\}$.  By definition of $E_m $, there are at least $\lfloor \ell  2^{ (1-v) n  -m} \rfloor$ intervals of the form  $[(k-1)2^{-2n} , k 2^{-2n}]$ for $k\in [1, \ell 2^{2n}]_{\BB Z}$ whose images under $\BB p$ intersect $[x,y]_{\bdy \op{BD}_{*,\ell}}$. By~\eqref{eqn-disk-holder} and our choice of $n$, the image of each of these intervals under $\BB p$ has $d_{*,\ell}$-diameter at most 
\eqbn
C 2^{-2(4+v/100)^{-1} n} \leq C 2^{-(1-v) n/2} \leq \delta ,
\eqen
 where the latter inequality is by~\eqref{eqn-interval-nm}. Hence each such image lies in $B_\delta\left( [x,y]_{\bdy\op{BD}_{*,\ell}} ; d_{*,\ell} \right)$. By combining this with~\eqref{eqn-interval-nm}, we find that for an appropriate random choice of $  c>0$,
\begin{align}
\mu_{*,\ell} \left( B_\delta \left( [x,y]_{\bdy\op{BD}_{*,\ell}} ; d_{*,\ell} \right) \right)
& \geq  \lfloor \ell  2^{ (1-v) n  -m} \rfloor \times 2^{-2n}   \notag\\
&\geq (\ell/2) 2^{-(1+v)n-m} 
\geq  c \delta^{\tfrac{2(1+v)}{1-v}}   \nu_{*,\ell} \left(  [x,y]_{\bdy\op{BD}_{*,\ell}} \right)  .
\end{align} 
This holds simultaneously for each boundary arc $[x,y]_{\bdy \op{BD}_{*,\ell}}$ satisfying~\eqref{eqn-interval-length-bdy}. 
By choosing $v$ sufficiently small that $2(1+v)/(1-v) \leq 2+u$ and shrinking $c$ in a manner depending on $m_0$, we obtain the statement of the lemma for a random-area Brownian disk. The statement for a fixed area Brownian disk follows by local absolute continuity.
\end{proof}

\section{Metric gluing}
\label{sec-metric-gluing}

In this section we will complete the proofs of the theorems stated in Section~\ref{sec-results}.  In Section~\ref{sec-geodesic-hit}, we will show using~\cite[Lemma 18]{bet-mier-disk} that $\sqrt{8/3}$-LQG geodesics between quantum typical points a.s.\ do not hit the boundary of the domain, which is one of the main inputs in the proofs of our main results.  Next, in Section~\ref{sec-wedge-gluing}, we will prove Theorem~\ref{thm-wedge-gluing}, noting that the proof of Theorem~\ref{thm-cone-gluing} is essentially identical.  Finally, in Section~\ref{sec-general-gluing}, we will deduce Theorems~\ref{thm-peanosphere-gluing} and~\ref{thm-peanosphere-gluing-finite}. 
 
Throughout this section, we define for $\rho > 1$
\eqb  \label{eqn-box-def}
\BB V_\rho := \left\{ z\in B_\rho(0) \,:\,  \im z > \rho^{-1} \right\}  \enskip \op{and}  \enskip \BB V_\rho' := \left\{ z\in B_\rho(0) \,:\, |z| > \rho^{-1} \right\}  .
\eqe

\subsection{LQG geodesics cannot hit the boundary}
\label{sec-geodesic-hit}
 
To prove our main results, we want to apply the estimates for the Brownian disk obtained in Section~\ref{sec-disk-estimate}. 
In order to apply these estimates, we need to ensure that we can restrict attention to finitely many quantum surfaces that locally look like quantum disks (equivalently, Brownian disks). 

If we are in the setting of Theorem~\ref{thm-wedge-gluing} with either $\frk w^- < 4/3$ or $\frk w^+ < 4/3$, the SLE$_{8/3}(\frk w^- -2;\frk w^+ -2)$ curve $\eta$ will intersect $\BB R$ in a fractal set (see \cite{miller-wu-dim} for a computation of the dimension of this set), so there will be infinitely many elements of $\mcl U^- \cup \mcl U^+$ (i.e., connected components of $\BB H\setminus \eta$) contained in small neighborhoods of certain points of $\BB R$. However, since $\eta$ does not intersect itself and is a.s.\ continuous and transient, there are only finitely many connected components of $\BB H\setminus \eta$ which intersect each of the sets $\BB V_\rho$ of~\eqref{eqn-box-def}. 
Hence one way to avoid dealing with infinitely many elements of $\mcl U^- \cup \mcl U^+$ is to work in a bounded set at positive distance from $\BB R$. The following lemma will allow us to do so. For the statement, we recall the parameter $Q$ from~\eqref{eqn-lqg-coord}.

\begin{lem} \label{prop-geodesic-bdy}
Suppose that $h$ is either a free boundary GFF on $\BB H$ plus $-\alpha\log|\cdot|$ for $\alpha \leq Q$, an embedding into $(\BB H , 0, \infty)$ of an $\alpha$-quantum wedge for $\alpha \leq Q$, or an embedding into $(\BB H , 0, \infty)$ of a single bead of an $\alpha$-quantum wedge for $\alpha \in (Q , Q + \sqrt{2/3})$ with area $\frk a >0$ and left/right boundary lengths $\frk l^- , \frk l^+ >0$. 
Let $R>1$ and let $z_1,z_2$ be sampled uniformly from $\mu_h|_{\BB V_R}$, normalized to be a probability measure, where $\BB V_R$ is as in~\eqref{eqn-box-def}. Almost surely, there is a unique $\frk d_h$-geodesic $\gamma_{z_1,z_2}$ from $z_1$ to $z_2$, and a.s.\ $\gamma_{z_1, z_2}$ does not intersect $\BB R \cup \{\infty\}$. 
\end{lem}

We will eventually deduce Lemma~\ref{prop-geodesic-bdy} from~\cite[Lemma~18]{bet-mier-disk} and local absolute continuity of $h$ with respect to an embedding of a quantum disk, which is isometric to a Brownian disk. (It is also possible to give a slightly longer proof using only SLE/LQG theory.) However, due to the presence of the log singularity at $0$ this absolute continuity only holds away from $0$ and $\infty$ so we first need to rule out the possibility that $\frk d_h$-geodesics hit $0$ or $\infty$ with positive probability. 
By an absolute continuity argument, it suffices to prove this in the case when $h$ is a free boundary GFF with a log singularity at  $0$. 
This is the purpose of the next two lemmas. 

The proof of the following lemma illustrates a general technique which can be used to show that various events associated with the $\sqrt{8/3}$-LQG metric induced by some variant of the GFF occur with positive probability. 

\begin{lem} \label{lem-diam-dist-pos}
Let $h$ be a free-boundary GFF on a simply connected domain $D\subset \BB C$. Let $A_1,A_2\subset \ol{D}$ be deterministic disjoint compact sets and let $c   > 0$. Then with positive probability, the $\sqrt{8/3}$-LQG metric $\frk d_{h}$ satisfies  
\eqb \label{eqn-diam-dist-pos}
\op{diam} \left( A_1 ; \frk d_{h}\right) \leq c \frk d_h \left( A_1, A_2 \right).
\eqe 
\end{lem} 
\begin{proof}
Let $U_1 , U_2  \subset \ol D$ be bounded connected relatively open sets such that $A_1 \subset U_1  $, $A_2\subset U_2  $, and $\ol U_1  \cap \ol U_2  = \emptyset$.  
We note that $h|_{U_1}$ has the law of a GFF on $U_1$ plus a random harmonic function. Hence the $\sqrt{8/3}$-LQG metric $\frk d_{h|_{U_1}}$ is well-defined, finite on compact subsets of $U_1$, and determines the same topology on $U_1$ as the Euclidean metric. We furthermore have $\frk d_{h|_{U_1}} \geq \frk d_h$ on $U_1$. 
Since $\frk d_h$ determines the same topology as the Euclidean metric, there exists $C >0$ such that  
\eqb  \label{eqn-good-metric-event}
\BB P \left[ \op{diam}   \left( A_1  ; \frk d_{h|_{U_1} } \right) \leq C \frk d_h \left( A_2, \bdy U_2 \right) \right] \geq \frac12   .
\eqe 
Let $R >0$ be a constant to be chosen later and let $\phi$ be a smooth function on $\ol D$ which is identically equal to $-R$ on $U_1$, identically equal to $R$ on $U_2$, and identically equal to $0$ outside a compact subset of $\ol D$. Let $\wh h := h + \phi$. By the scaling property of the $\sqrt{8/3}$-LQG metric~\cite[Lemma~2.2]{lqg-tbm2},
\eqbn
\op{diam}  \left( A_1  ; \frk d_{\wh h|_{U_1} } \right) = e^{- \frac{\sqrt{8/3}}{4} R  }  \op{diam}  \left( A_1  ; \frk d_{ h|_{U_1} } \right) \enskip \op{and}  \enskip \frk d_{\wh h} \left( A_2, \bdy U_2 \right) = e^{\frac{\sqrt{8/3}}{4} R } \frk d_{  h} \left( A_2, \bdy U_2 \right)  .
\eqen
Hence if we choose $R >0$ such that $ C e^{ -   \sqrt{8/3} R /4}   \leq c $, then~\eqref{eqn-good-metric-event} implies that with probability at least $1/2$,  
\eqbn
  \op{diam}   \left( A_1  ; \frk d_{ \wh h|_{U_1} } \right) \leq c  \frk d_{\wh h} \left( A_2, \bdy U_2 \right) ,
\eqen
in which case
\eqb \label{eqn-diam-dist-compare}
\op{diam}   \left( A_1 ; \frk d_{\wh h} \right) \leq c    \frk d_{\wh h}(A_1,A_2 ). 
\eqe 
On the other hand, since $\wh h - h$ is a smooth compactly supported function, the law of $\wh h$ is absolutely continuous with respect to the law of $h$, so with positive probability~\eqref{eqn-diam-dist-compare} holds with $h$ in place of $\wh h$. By scaling, this implies that~\eqref{eqn-diam-dist-compare} holds.
\end{proof}

Now we can rule out the possibility that geodesics for the $\sqrt{8/3}$-LQG metric induced by a GFF with a log singularity at $0$ hit $0$.

\begin{lem} \label{prop-free-gff-hit0}
Let $\wt h$ be a free boundary GFF on $\BB H$ and let $\alpha\in\BB R$. Let $h := \wt h -\alpha\log |\cdot|$ and let $\frk d_h$ be the $\sqrt{8/3}$-LQG metric induced by $h$. Almost surely, no $\frk d_h$-geodesic between points in $\BB H$ hits $0$.
\end{lem}  
\begin{proof}
For each $j \in \BB N$, let $r_j = e^{-j}$ and let $\mcl F_j$ be the $\sigma$-algebra generated by $h|_{\BB H\setminus B_{r_j}(0)}$. Let $\frk h_j$ be the conditional mean of $h|_{B_{r_j}(0)\cap \BB H}$ given $h|_{\BB H\setminus B_{r_j}(0)}$.
Let $E_j$ be the event that
\eqbn
\op{diam} \left( \bdy B_{r_{j-1}}(0) \cap \BB H ; \frk d_h  \right) <  \frk d_h \left( \bdy B_{ r_{j-1}}(0) \cap \BB H , \left( \bdy B_{r_j}(0) \cup \bdy B_{r_{j-2}}(0) \right) \cap \BB H\right).
\eqen
Then if $E_j$ occurs for some $j \in \BB N$, no geodesic between points in $\BB H\setminus B_{r_j}(0)$ can hit $0$. Hence we just need to show that for each $j_0 \in \BB N$,  
\eqbn
\BB P \left[ \bigcup_{j=j_0}^\infty E_j \right] = 1. 
\eqen

The event $E_j$ is the same as the event that the following is true: if we grow the $\frk d_h$-metric balls centered at any point of $\bdy B_{r_{j-1}}(0)$, then we cover $\bdy B_{r_{j-1}}(0)$ before reaching $\bdy B_{ r_j}(0)$ or $\bdy B_{r_{j-2}}(0)$. Since metric balls are locally determined by $h$ (this is immediate from the construction of $\frk d_h$ using QLE$(8/3,0)$ in~\cite{lqg-tbm1,lqg-tbm2,lqg-tbm3}), it follows that $E_j$ is determined by the restriction of $h$ to $B_{r_{j-2}}(0)\setminus B_{r_j}(0)$, and in particular $E_j \in \mcl F_j$. By the scaling property of $\frk d_h$~\cite[Lemma 2.2]{lqg-tbm2}, the events $E_j$ do not depend on the choice of additive constant for $h$. By the conformal invariance of the law of $h$, modulo additive constant, and Lemma~\ref{lem-diam-dist-pos}, we can find $p > 0$ such that $\BB P[E_j] \geq p$ for each $j \in \BB N$. 
  
Now fix $j_0  \in \BB N$ and $\ep >0$. Inductively, suppose $k\in\BB N$ and we have defined a $j_{k-1} \geq j_0$ which is a stopping time for the filtration $\{\mcl F_j\}_{j \in \BB N}$. The function $\frk h_{j_{k-1}}$ is a.s.\ harmonic, hence smooth, on $ \ol{B_{r_{j_{k-1}+1}}(0) \cap \BB H } $ so we can find $j_k \geq j_{k-1} + 2 $ such that 
\eqbn
\int_{B_{r_{j_k-2} }(0) \cap \BB H} | \nabla \frk h_{j_{k-1}}( w) |^2 \, dw   \leq \ep .
\eqen
Since $\frk h_{j_{k-1}}$ is $\mcl F_{j_{k-1}}$-measurable, it follows by induction that $j_k$ is a stopping time for $\{\mcl F_j\}_{j \in \BB N}$. By~\cite[Lemma~5.4]{lqg-tbm2}, if we choose $\ep >0$ sufficiently small, depending on $p$, then it holds with conditional probability at least $1-p/2$ given $\mcl F_{j_{k-1}}$ that the Radon-Nikodym derivative of the conditional law of $h|_{B_{r_{j_k-2}}}(0) \cap \BB H$ with respect to the unconditional law of a free-boundary GFF on $\BB H$ plus $-\alpha\log|\cdot|$ restricted to $B_{r_{j_k-2}}(0) \cap \BB H$, viewed as distributions defined modulo additive constant, is between $1/2$ and $2$. Hence
\eqbn
\BB P \left[ E_{j_k} \,|\, \mcl F_{j_{k-1}} \right] \geq \frac{p}{4} .
\eqen
By the Borel-Cantelli lemma, it follows that a.s.\ $E_{j_k}$ occurs for some $k\in\BB N$. 
\end{proof}

\begin{proof}[Proof of Lemma~\ref{prop-geodesic-bdy}]
For each sufficiently small $r>0$, the law of $h |_{B_r(0)\cap \BB H}$ for an appropriate choice of embedding $h$ is absolutely continuous with respect to the law of the corresponding restriction of a free-boundary GFF plus a $\log$ singularity at $0$ with appropriate choice of additive constant. By Lemma~\ref{prop-free-gff-hit0}, a.s.\ no geodesic of $\frk d_h$ hits $0$. In the case of a free-boundary GFF or a thick quantum wedge, it is clear that a.s.\ no geodesic of $\frk d_h$ hits $\infty$. In the case of a single bead of a thin wedge, this follows since $(\BB H , h , 0,\infty) \eqD (\BB H , h, \infty , 0)$. It follows that we can find $\rho = \rho(p) > 1$ such that with probability at least $1-p$, every geodesic from $z_1$ to $z_2$ is contained in $\BB V_\rho'$. Let $E$ be the event that this is the case.

By Lemma~\ref{prop-internal-equal}, on $E$ the set of $\frk d_h$-geodesics from $z_1$ to $z_2$ coincides with the set of $\frk d_{h|_{\BB V_\rho'}}$-geodesics from $z_1$ to $z_2$. 
The law of the field $h|_{\BB V_{\rho}'}$ is absolutely continuous with respect to the law of the corresponding restriction of an appropriate embedding into $(\BB H , 0 ,\infty)$ of a quantum disk, which by~\cite[Corollary~1.5]{lqg-tbm2} is isometric to a Brownian disk.  
By~\cite[Lemma~18]{bet-mier-disk}, we infer that a.s.\ there is only one $\frk d_{h|_{\BB V_\rho'}}$-geodesic from $z_1$ to $z_2$ and that this geodesic a.s.\ does not hit $\bdy \BB D$. Since $\BB P[E]$ can be made arbitrarily close to 1 by choosing $\rho$ sufficiently large, we obtain the statement of the lemma.  
\end{proof}

In the setting of Theorem~\ref{thm-cone-gluing}, there is no boundary to worry about but we still need to ensure that geodesics stay away from the origin, since, due to the presence of a log singularity, the restriction of the field~$h$ to a neighborhood of the origin is not absolutely continuous with respect to a quantum disk (or sphere).  The following lemma addresses this issue. We do not give a proof since the argument is essentially the same as the proof of Lemma~\ref{prop-geodesic-bdy}.

\begin{lem} \label{prop-geodesic-plane}
Suppose that $h$ is either a whole-plane GFF plus $-\alpha\log|\cdot|$ for $\alpha < Q$, an embedding into $(\BB  C , 0, \infty)$ of an $\alpha$-quantum cone for $\alpha <Q$, or an embedding into $(\BB C , 0 ,\infty)$ of a quantum sphere.
Let $R>1$ and let $z_1,z_2$ be sampled uniformly from $\mu_h|_{B_R(0) \setminus B_{1/R}(0) }$, normalized to be a probability measure. Almost surely, there is a unique $\frk d_h$-geodesic $\gamma_{z_1,z_2}$ from $z_1$ to $z_2$, and a.s.\ $\gamma_{z_1, z_2}$ does not hit $0$ or $\infty$. 
\end{lem}

\subsection{Metric gluing of two quantum wedges}
\label{sec-wedge-gluing}

In this subsection we will prove Theorem~\ref{thm-wedge-gluing}. 
The proof of Theorem~\ref{thm-cone-gluing} is similar to that of Theorem~\ref{thm-wedge-gluing}, but slightly simpler because we only need to consider a single complementary connected component of $\eta$ (note that the proof uses Lemma~\ref{prop-geodesic-plane} in place of Lemma~\ref{prop-geodesic-bdy}). So, we will only give a detailed proof of Theorem~\ref{thm-wedge-gluing}.

Throughout, we assume we are in the setting of Theorem~\ref{thm-wedge-gluing}.  Recall in particular that $(\BB H , h , 0, \infty)$ is a weight-$\frk w$ quantum wedge (so $\mu_h(\BB H) =\infty$) if $\frk w \geq 4/3$, or a single bead of a thin wedge with quantum area $\frk a$ and left/right quantum boundary lengths $\frk l^-$ and $\frk l^+$ (if $\frk w < 4/3$).  We assume that the embedding of $h$ is chosen so that the quantum mass of $\BB D\cap \BB H$ is $1$ (if $\frk w \geq 4/3$) or so that the quantum mass of $\BB D\cap \BB H$ is $\frk a/2$ (if $\frk w < 4/3$).  This is merely for convenience as the statement of the theorem is independent of the choice of embedding. We also assume that the SLE$_{8/3}(\frk w^--2;\frk w^+-2)$ curve $\eta$ is parameterized by quantum length with respect to $h$. 
 
For a connected component $U\in \mcl U^-\cup\mcl U^+$, let $x_U$ (resp.\ $y_U$) be the first (resp.\ last) point of $\bdy U$ hit by $\eta$. Then the quantum surfaces $(U , h|_U  , x_U , y_U)$ for $U\in\mcl U^\pm$ are the beads of $\mcl W^\pm$ if $\frk w^\pm < 4/3$, or all of $\mcl W^\pm$ if $\frk w^\pm \geq 4/3$. 

We now give an outline of the proof of Theorem~\ref{thm-wedge-gluing}. Recall from Section~\ref{sec-metric-basic} that distances with respect to the quotient metric are given by the infimum of the lengths of paths which cross $\eta$ only finitely many times, so to prove that $\frk d_h$ coincides with the quotient metric we need to show that $\frk d_h$-distances are well-approximated by the lengths of paths which cross $\eta$ only finitely many times.

For most of the proof, we will truncate on a global regularity event $G_C = G_C(u,\rho)$, which we define in Lemma~\ref{lem-metric-reg} just below. After establishing that $G_C$ holds with probability close to $1$ when the parameters are chosen appropriately (which is accomplished using the estimates of Section~\ref{sec-disk-estimate}), our main aim will be to prove the following. There is an $\alpha>0$ such that if $G_C$ occurs, then for small $\delta>0$ the $\frk d_h$-geodesic $\gamma_{z_1,z_2}$ between two points $z_1,z_2 \in B_\rho(0)$ sampled from $\mu_h|_{B_\rho(0)}$ hits at most $ \delta^{-1+\alpha + o_\delta(1)} $ SLE segments of the form $\eta([(k-1)\delta^2 , k\delta^2])$ for $k\in\BB N$.  In light of Lemma~\ref{prop-disk-bdy-holder}, this will allow us to construct a path from $z_1$ to $z_2$ of length approximately $\frk d_h(z_1,z_2)$ which consists of a concatenation of finitely many segments of $\gamma_{z_1,z_2}$ which are each contained in some $U\in \mcl U^-\cup \mcl U^+$; and $ \delta^{-1+\alpha + o_\delta(1)} $ paths of length at most $\delta^{1+o_\delta(1)}$ which do not cross $\eta$, each of which connects two points of $\eta([(k-1)\delta^2 , k\delta^2])$ for some $k\in\BB N$. Such a path crosses $\eta$ only finitely many times and has length at most $\op{len}(\gamma_{z_1,z_2} ; \frk d_h) + \delta^{\alpha +o_\delta(1)}$. By the definition of the quotient metric (Section~\ref{sec-metric-basic}), the existence of this path shows that $\frk d_h(z_1,z_2)$ differs from the quotient metric distance between $z_1$ and $z_2$ by at most $\delta^{\alpha +o_\delta(1)}$. Sending $\delta \rta 0$ then concludes the proof. 

The proof of our estimate for the number of $\delta^2$-length SLE segments hit by a $\frk d_h$-geodesic will be accomplished in two main steps. First we will show in Lemma~\ref{prop-sle-ball-count} that on $G_C$, a $\frk d_h$-metric ball of radius approximately $\delta$ cannot intersect too many $\delta^2$-length SLE segments. Then we will obtain the desired estimate in Lemma~\ref{prop-sle-intersect-mean} by bounding the probability that $\eta$ gets within $\frk d_h$-distance~$\delta$ of a given segment of the geodesic~$\gamma_{z_1,z_2}$. 

We are now ready to define our regularity event $G_C$ and show that it occurs with high probability. For the statement of the following lemma, we recall the sets $\BB V_\rho$ and $\BB V_\rho'$ from~\eqref{eqn-box-def}. 

\begin{lem} \label{lem-metric-reg}
There exists $\beta >0$ such that the following is true. 
For $u\in (0,1)$, $\rho > 2$, and $C> 1$, let $G_C = G_C(u,\rho)$ be the event that the following hold.
\begin{enumerate}
\item \label{item-area-upper} For each $z\in \BB V_\rho$ and each $\delta\in (0,1]$ with $B_\delta \left(z ; \frk d_h\right) \cap\BB R = \emptyset$, we have $\mu_h \left( B_\delta \left(z ; \frk d_h \right) \right) \leq C \delta^{4-u}$.
\item \label{item-area-lower} For each $U\in \mcl U^- \cup \mcl U^+$ with $ U \cap \BB V_\rho\not=\emptyset$, each $z \in \ol U \cap \BB V_\rho$, and each $\delta \in (0,1]$, we have $\mu_h \left( B_\delta \left(z ;  \frk d_{h|_U} \right)\right) \geq C^{-1} \delta^{4+u}$.  
\item \label{item-length-upper} For each $U\in \mcl U^- \cup \mcl U^+$ with $ U \cap \BB V_\rho\not=\emptyset$ and each $x,y\in \bdy U \cap \BB V_\rho$,  
\begin{equation}
\label{eqn-length-upper}
\frk d_{h|_U}(x,y) \leq C     \nu_h \left([x,y]_{\bdy U} \right)^{1/2}  \left(  \left| \log \nu_{a,\ell} \left( [x,y]_{\bdy \op{BD}_{a,\ell}} \right) \right| + 1 \right)^2. 
\end{equation}
\item \label{item-length-lower} For each $U\in \mcl U^- \cup \mcl U^+$ with $U \cap \BB V_\rho\not=\emptyset$, each $x,y\in \bdy U \cap \BB V_\rho$, and each $\delta \in (0,1)$,  
\eqbn
\mu_h  \left( B_\delta \left( [x,y]_{\bdy U} ; \frk d_{h|_{U}} \right) \right) \geq C^{-1} \delta^{2 + u}  \nu_h \left(  [x,y]_{\bdy U} \right).
\eqen
\item \label{item-euclidean-holder} For each $z\in \BB V_\rho'$ and each $\delta\in (0,1]$, we have $ B_\delta(z ; \frk d_h)   \subset B_{ C \delta^\beta}(z)$. 
\item \label{item-sle-cont} For each $t > s > 0$ such that $\eta(s) \in \BB V_{\rho/2}$ and $|t-s| \leq C^{-1}$, we have $\eta(t) \in \BB V_\rho$. 
\end{enumerate}
For each $u > 0$, $\rho > 2$, and $p\in (0,1)$ there exists $C> \rho$ such that $\BB P[G_C] \geq 1-p$.
\end{lem}  
We remark that in~\eqref{eqn-length-upper} it is possible to improve the exponent of the $\log$ term to $7/4+\zeta$ for any $\zeta > 0$ (see Lemma~\ref{prop-disk-bdy-holder}), but the particular exponent does not matter for our purposes so for simplicity we just take it to be 2.
\begin{proof}[Proof of Lemma~\ref{lem-metric-reg}]
Fix $u \in (0,1)$, $\rho > 2$, and $p\in (0,1)$. 
We will deduce the statement of the lemma from the results of Section~\ref{sec-disk-estimate} and local absolute continuity (recall that the Brownian disk is equivalent to the quantum disk). 
First we need to reduce ourselves to considering only finitely many quantum surfaces, rather than the surfaces parameterized by all of the complementary connected components of $\eta$ (which is an infinite set of one of $\frk w^\pm$ is smaller than $4/3$).

Let $\mcl F$ be the $\sigma$-algebra generated by the $\mu_h$-areas of the ordered sequences of beads of $\mcl W^-$ and $\mcl W^+$ (so that $\mcl F$ is trivial if $\frk w^- \wedge \frk w^+ \geq 4/3$). For $\ep > 0$, let $\mcl U_\ep^\pm$ be the set of $U\in \mcl U^\pm$ such that the following is true: we have $\mu_h(U) \geq \ep$, and the sum of the quantum areas of the beads of $\mcl W^\pm$ which come before~$U$ is at most~$\ep^{-1}$. Then each $\mcl U_\ep^\pm$ is a finite set.  Furthermore, $\mcl F$ determines which elements of $\mcl U^\pm$ belong to $\mcl U_\ep^\pm$ so the conditional law of the surfaces $(U , h|_U , x_U , y_U)$ for $U\in \mcl U_\ep^\pm$ given $\mcl F$ is that of a collection of independent beads of a weight-$\frk w^\pm$ quantum wedge with given areas (or a single weight-$\frk w^\pm$ quantum wedge if $\frk w^\pm \geq 4/3$). 

The boundary of each $U\in \mcl U^\pm$ intersects $\BB R$ (since $\eta$ is simple) and there are only finitely many $U\in\mcl U^\pm$ which intersect $B_\rho(0)$ and have diameter larger than $\rho^{-1}$ (since $\eta$ is continuous and transient \cite{ig1}). Consequently, we can find $\ep  = \ep(\rho) > 0$ such that with probability at least $1-p/5$, it holds that each $U\in \mcl U^- \cup \mcl U^+$ which intersects $\BB V_\rho$ belongs to $\mcl U_\ep^-\cup \mcl U_\ep^+$. Henceforth fix such an $\ep$. 
 
For $U\in \mcl U^\pm$, let $\phi_U : U \rta \BB H$ be the unique conformal map which takes $x_U$ to $0$ and $y_U$ to $\infty$, normalized so that the quantum measure induced by the field $h_U := h\circ \phi_U^{-1} + Q\log |(\phi_U^{-1})'|$ assigns mass $1$ to $\BB D\cap\BB H$ (if $\mu_h(U) = \infty$) or mass $\mu_h(U)/2$ to $\BB D\cap \BB H$ (if $\mu_h(U) < \infty$).  

Since each of the marked points $x_U , y_U$ for $U\in\mcl U^\pm$ lies in $\BB R \cup \{\infty\}$, we infer that each $\phi_U$ for $U\in \mcl U^\pm$ a.s.\ maps $U\cap \BB V_{2\rho}$ to a bounded subset of $\BB H$ lying at positive distance from  $0$. 
Since each $\mcl U_\ep^\pm$ is a.s.\ finite, we can find $\wt\rho > 1$ such that with probability at least $1-p/5$,  
 \eqbn
 \phi_U(U \cap \BB V_{2\rho}' ) \subset \BB V_{\wt\rho}' ,\quad \forall U \in\mcl U_\ep^- \cup \mcl U_\ep^+ .
\eqen
Again using the finiteness of $\mcl U_\ep^\pm$, we can find $a > 1\wedge\frk a$ and $\ell > 0$ such that with probability at least $1-p/5$,  
\alb
& \mu_h \left( \BB V_{2\rho}' \right) \leq a/2 , \quad  
 \nu_h \left( \bdy \BB V_{2\rho}' \cap \BB R\right) \leq \ell/2   ,\\
&\mu_{h_U}  \left( \BB V_{2\wt\rho}' \right) \leq a/2,\quad \op{and}\quad 
\nu_{h_U} \left(  \bdy \BB V_{2\wt\rho}' \cap\BB R\right) \leq  \ell/2  ,
  \quad \forall U\in \mcl U_\ep^- \cup \mcl U_\ep^+ .
\ale

Let $(\BB H , \wt h , 0 , \infty)$ be a doubly marked quantum disk with area $a$ and left/right boundary lengths each equal to $\ell$, with the embedding chosen so that $\mu_{\wt h}(\BB D\cap\BB H) = 1$ (if $\frk w \geq 4/3$) or $\mu_{\wt h}(\BB D\cap\BB H) = \frk a/2$ (if $\frk w < 4/3$).  
It is easy to see from the definitions given in~\cite[Section~4]{wedges} that on the event $\{\mu_h \left( \BB V_{2\rho}' \right) \leq a/2 , \,\nu_h \left( \bdy \BB V_{2\rho}' \cap \BB R\right) \leq \ell/2  \}$ the law of $h|_{\BB V_\rho'}$ is absolutely continuous with respect to the law of the corresponding restriction of $\wt h$.  
Furthermore, if we condition on $\mcl F$ then on the event 
\[ \left\{\mu_{h_U}  \left( \BB V_{2\wt\rho}'   \right) \leq a/2 ,\, \nu_{h_U} \left(  \bdy \BB V_{2\wt\rho}' \cap\BB R\right) \leq  \ell/2  \right\},\]
the conditional law of $h_U|_{\BB V_{\wt\rho}'}$ is absolutely continuous with respect to the law of the corresponding restriction of~$\wt h$.  

By~\cite[Corollary~1.5]{sphere-constructions}, the space $\BB H$ equipped with the $\sqrt{8/3}$-LQG metric induced by~$\wt h$ is isometric to the Brownian disk with area $a$ and boundary length~$2\ell$. Since $\bdy \BB V_r$ a.s.\ lies at positive quantum distance from~$\bdy \BB V_{2r}$ for each $r> 1$ and similarly for $\BB V_r'$ and $\BB V_{2r}'$, it follows from Lemmas~\ref{prop-disk-bdy-holder},~\ref{prop-ball-size}, and~\ref{lem-interval-nghbd} that we can find $C> 1$ such that conditions~\ref{item-area-upper} through~\ref{item-length-lower} in the statement of the lemma hold with probability at least $1 - p/2$. 

By~\cite[Theorem~1.2]{lqg-tbm2} and local absolute continuity, for an appropriate choice of universal $\beta>0$ and a large enough choice of $C$, condition~\ref{item-euclidean-holder} holds with probability at least $1-p/4$. By the continuity and transience of $\eta$~\cite[Theorem 1.3]{ig1}, for large enough $C > 1$ (depending on $\rho$) condition~\ref{item-sle-cont} holds with probability at least $1-p/4$.  
\end{proof}

Our next lemma shows that $\eta$ satisfies a reverse H\"older continuity condition on the event $G_C$ (here we recall that $\eta$ is parameterized by quantum length with respect to $h$). 

\begin{lem} \label{prop-sle-diam-lower}
For each $v \in (0,1)$, there exists $u_0 = u_0(v) \in (0,1)$ such that for $u\in (0,u_0]$ the following is true. Let $\rho>2$ and $C>1$ and let $G_C = G_C(u , \rho)$ be the event of Lemma~\ref{lem-metric-reg}. There exists $\ep_0 = \ep_0(C,u,v)$ such that the following is true a.s.\ on $G_C$. If $0 < a < b < \infty$ such that $b-a \leq \ep_0$ and $\eta([a,b]) \cap \BB V_{\rho/2} \not=\emptyset$, then
\eqbn
\op{diam}  \left( \eta([a,b]) ; \frk d_h \right) \geq 7 (b-a)^{ (1+v)/2} .
\eqen
\end{lem}
\begin{proof}
Throughout the proof we assume that $G_C$ occurs.
We first observe that condition~\ref{item-euclidean-holder} in Lemma~\ref{lem-metric-reg} implies that there is an $\ep_0 = \ep_0(C,\rho)$ such that for $\ep \in (0,\ep_0]$ and $z\in \BB V_{\rho/2}$,  
\eqb \label{eqn-sle-ball-contain}
B_{8 \ep^{(1+v)/2}}(z; \frk d_h) \subset  \BB V_{\rho } .
\eqe  
Suppose now that $0 < a < b < \infty$ such that $b-a \leq \ep_0$ and $\eta([a,b]) \cap \BB V_{\rho/2} \not=\emptyset$. Let $z\in \eta([a,b]) \cap \BB V_\rho$ and write $\delta := (b-a)^{(1+v)/2}$. It follows from~\eqref{eqn-sle-ball-contain} that $B_{8 \delta}(z; \frk d_h)$ does not intersect $\BB R$, and this combined with~\eqref{eqn-sle-ball-contain} implies there exists $U\in \mcl U^-$ such that $\eta([a,b]) \subset \bdy U \cap \BB V_\rho$.  
By condition~\ref{item-length-lower} in Lemma~\ref{lem-metric-reg} and since $\frk d_h(w_1,w_2) \leq \frk d_{h|_{U}}(w_1,w_2)$ for each $w_1,w_2 \in \ol U$  (by Lemma~\ref{prop-internal-equal} and Remark~\ref{remark-internal-compare}),  
\eqbn
\mu_h \left( B_{\delta } \left( \eta([a,b]) ; \frk d_h \right) \right) \geq \mu_h \left( B_{\delta } \left( \eta([a,b]) ; \frk d_{h|_U} \right) \right) \geq C^{-1} (b-a)^{ (2+u)(1+v)/2  +1} .
\eqen
By condition~\ref{item-area-upper}, 
\eqbn
\mu_h \left( B_{8 \delta }(z ; \frk d_h) \right) \leq 8^{4-u} C (b-a)^{(4-u)(1+v)/2} . 
\eqen
If $u$ is chosen sufficiently small relative to $v$, then after possibly shrinking $\ep_0$ we can arrange that $8^{4-u} C (b-a)^{(4-u)(1+v)/2}  <  C^{-1} (b-a)^{ (2+u)(1+v)/2  +1} $ whenever $b-a \leq \ep_0$, so $B_{\delta } \left( \eta([a,b]) ; \frk d_h \right)$ cannot be contained in $B_{8 \delta }(z ; \frk d_h)$. Therefore $\eta([a,b]) $ cannot be contained in $B_{7\delta }(z ; \frk d_h)$. 
\end{proof}
 
We next bound the number of $\delta^2$-length segments of $\eta$ which can intersect a $\frk d_h$-metric ball on the event $G_C$. Here we recall that $\eta$ is parameterized by $\nu_h$-length. 

\begin{lem} \label{prop-sle-ball-count}
For each $v \in (0,1)$, there exists $u_0 = u_0(v) \in (0,1)$ such that for $u\in (0,u_0]$, the following is true. Let $\rho>2$ and $C> 1$ and let $G_C = G_C(u , \rho)$ be the event of Lemma~\ref{lem-metric-reg}. There exists $\delta_0 = \delta_0(C,u,v) > 0$ such that the following is true almost surely on $G_C$. Let $z\in \BB V_{ \rho/2}$ and let $\delta \in (0,\delta_0]$. Then the number of $k\in \BB N$ for which $\eta([(k-1)\delta^2 , k\delta^2]) \cap B_{\delta^{1+v}}(z; \frk d_h) \not=\emptyset$ is at most $\delta^{-v}$.
\end{lem}
\begin{proof}
See Figure~\ref{fig-sle-ball} for an illustration of the argument.
Throughout the proof we assume that $G_C$ occurs. As in the proof of Lemma~\ref{prop-sle-diam-lower}, if $z$ is as in the statement of the lemma then
\eqb \label{eqn-sle-ball-contain'}
B_{3\delta^{1+v}}(z; \frk d_h) \subset  \BB V_{\rho }
\eqe
for small enough $\delta$ (depending only on $C$ and $\rho$).  By Lemma~\ref{prop-sle-diam-lower}, if $u\in(0,1)$ is chosen sufficiently small relative to $v$, then for small enough $\delta>0$ (depending only on $C$, $u$, and $v$) there cannot exist $k\in\BB N$ for which $\eta([(k-1)\delta^2 ,k\delta^2]) \subset B_{2\delta^{1+v}}(z ; \frk d_h)$.  Henceforth assume that $\delta$ and $u$ are chosen so that the above conditions are satisfied.

Let~$K$ be the set of $k\in \BB N$ for which $\eta([(k-1)\delta^2 , k\delta^2]) \cap B_{\delta^{1+v}}(z;\frk d_h) \not=\emptyset$. By our choices of~$\delta$ and~$u$, each curve segment $\eta([(k-1)\delta^2 , k\delta^2])$ for $k\in K$ also intersects $\BB H\setminus  B_{2 \delta^{1+v}}(z ; \frk d_h)$. 

To complete the proof, we will first use a topological argument to show that the number of connected components of $\BB H \setminus (\eta \cup B_{\delta^{1+v}}(z;\frk d_h))$ which intersect $\BB H\setminus B_{2\delta^{1+v}}(z;\frk d_h)$ is at least $(\# K-2)/2$. We will then use the definition of $G_C$ to argue that each such connected component $O$ contains a one-sided metric ball $B_O \subset O\cap B_{2\delta^{1+v}}(z;\frk d_h)$ which has $\mu_h$-mass at least a constant times $\delta^{(4+u)(1+v)}$. Hence the number of such components $O$ is at most a constant times $\delta^{-(4+u)(1+v)} \mu_h(B_{2\delta^{1+v}}(z;\frk d_h))$, which can be bounded above by $\delta^{-v}$ by the definition of $G_C$. 
\medskip

\noindent\textit{Step 1: counting complementary connected components.}
Let $\mcl V$ be the set of connected components of $\BB H\setminus B_{\delta^{1+v}}(z;\frk d_h)$. For $V\in \mcl V$, let $\mcl O_V$ be the set of connected components $O$ of $V\setminus \eta$ such that $O$ intersects $V\setminus B_{2\delta^{1+v}}(z;\frk d_h)$. Note that $\bigcup_{V\in\mcl V} \mcl O_V$ is the set of connected components of $\BB H \setminus (\eta \cup B_{\delta^{1+v}}(z;\frk d_h))$ which intersect $\BB H\setminus B_{2\delta^{1+v}}(z;\frk d_h)$. We first argue that
\eqb \label{eqn-component-excursion}
\# K \leq  2 \sum_{V\in \mcl V} \# \mcl O_V + 2.
\eqe 
To see this, fix $V\in \mcl V$. Let $\{ [\ul\tau_{V,j} , \ol\tau_{V,j} ] \}_{j \in [1,J_V]_{\BB Z} }$ be the ordered sequence of time intervals for $\eta$ with the property that $\eta((\ul\tau_{V,j} , \ol\tau_{V,j})) \subset V$, $\eta(\ul\tau_{V,j} )  , \eta(\ol\tau_{V,j}) \in \bdy V$, and $\eta((\ul\tau_{V,j} , \ol\tau_{V,j})) $ intersects $V\setminus B_{2\delta^{1+v}}(z ; \frk d_h)$. In other words, the segments $\eta([\ul\tau_{V,j} , \ol\tau_{V,j}])$ are the excursions of $\eta$ into $V$ which hit $\BB H\setminus B_{2\delta^{1+v}}(z ; \frk d_h)$. 
Note that the number $J_V$ of such time intervals is finite by continuity and transience of $\eta$ and since $V\setminus B_{2\delta^{1+v}}(z ; \frk d_h)$ lies at positive Euclidean distance from $\bdy V$. 

The set $V$ has the topology of the disk (or the complement of a disk in $\BB H$, in the case of the unbounded component). Since $\eta$ does not hit itself each curve segment $\eta((\ul\tau_{V,j} , \ol\tau_{V,j}))$ is contained in a single connected component of $V\setminus \eta([0,\ul\tau_{V,j}])$, which also has the topology of a disk (or the complement of a disk in $\BB H$ if $V$ is the unbounded component). Since $\eta(\ul\tau_{V,j} )  , \eta(\ol\tau_{V,j}) \in \bdy V$, the segment $\eta([\ul\tau_{V,j} , \ol\tau_{V,j}])$ divides this connected component into at least two further components. Since $\eta((\ul\tau_{V,j} , \ol\tau_{V,j})) $ intersects $V\setminus B_{2\delta^{1+v}}(z ; \frk d_h)$, each such component also intersects $V\setminus B_{2\delta^{1+v}}(z ; \frk d_h)$. It follows that $J_V \leq \#\mcl O_V $. 

Since each segment  $\eta([(k-1)\delta^2 , k\delta^2])$ for $k\in K$ intersects both $ B_{\delta^{1+v}}(z;\frk d_h)$ and $\BB H\setminus B_{2\delta^{1+v}}(z;\frk d_h)$, each interval $ [(k-1)\delta^2 , k\delta^2] $ for $k\in K$ intersects $[\ul\tau_{V,j} , \ol\tau_{V,j} ]$ for some $V\in \mcl V$ and some $j\in [1,J_V]_{\BB Z}$, except possibly the first and last such intervals. Furthermore, since $\eta([\ul\tau_{V,j} , \ol\tau_{V,j} ])$ intersects $ \bdy B_{\delta^{1+v}}(z;\frk d_h)$ only at its endpoints, each such interval $[\ul\tau_{V,j} , \ol\tau_{V,j} ]$  intersects at most two intervals of the form $[(k-1)\delta^2 , k\delta^2]$ for $k\in K$. Therefore, $\# K \leq 2 \sum_{V\in\mcl V} J_V+2 $, whence~\eqref{eqn-component-excursion} holds. 
\medskip

\noindent\textit{Step 2: metric balls in connected components.}
Now we will prove an upper bound for the right side of~\eqref{eqn-component-excursion}. 
Let $V\in \mcl V$ and $O\in \mcl O_V$. 
We can choose $w_O \in \bdy O \cap \eta$ such that $\frk d_h \left(w_O , \bdy B_{\delta^{1+v}}(z) \right) = \frk d_h \left(w_O , \bdy B_{2\delta^{1+v}}(z) \right) = \frac14 \delta^{1+v}$.  We have $O\subset  U_O$ for some connected component $U_O \in \mcl U^- \cup \mcl U^+$ of $\BB H\setminus \eta$. 
Let $B_O$ be the open $\frk d_{h|_{U_O}}$-ball of radius $\frac14 \delta^{1+v}$ centered at $w_O$. Since $\eta$ does not cross $B_O$ and $\frk d_{h|_{U_O}} \geq \frk d_h$, 
\eqbn 
B_O \subset B_{\frac14\delta^{1+v}}(w_O ; \frk d_h) \subset B_{2 \delta^{1+v}}(z ; \frk d_h) \setminus B_{ \delta^{1+v}}(z ; \frk d_h) .
\eqen
Therefore, $B_O \subset O$ so in particular $B_O\cap B_{O'}  = \emptyset$ for distinct $O,O' \in \bigcup_{V\in \mcl V} \mcl O_V$. 

By~\eqref{eqn-sle-ball-contain'} and condition~\ref{item-area-lower} in Lemma~\ref{lem-metric-reg}, for each $O \in \bigcup_{V\in \mcl V} \mcl O_V$ we have $\mu_h(B_O) \succeq C^{-1} \delta^{(4+u)(1+v)}$, with universal implicit constant. Therefore, condition~\ref{item-area-upper} in Lemma~\ref{lem-metric-reg} implies that 
\eqbn
 C^{-1} \delta^{(4+u)(1+v)}  \sum_{V\in \mcl V} \# \mcl O_V   \preceq  \mu_h \left(B_{2 \delta^{1+v}}(z ; \frk d_h) \right) \preceq C\delta^{(4-u)(1+v)}  
\eqen
with universal implicit constants. By combining this with~\eqref{eqn-component-excursion} we get
\eqbn
\# K \preceq C^2 \delta^{-2u(1+v)} .
\eqen
After possibly shrinking $u$ and $\delta$, we obtain the statement of the lemma.
\end{proof}
 
\begin{figure}[ht!]
 \begin{center}
\includegraphics{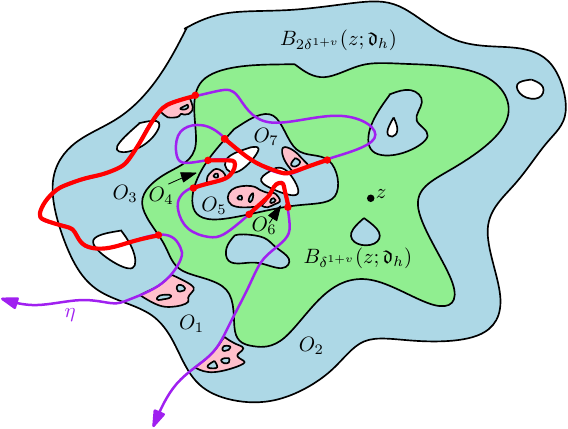} 
\caption[Illustration of the proof of Lemma~\ref{prop-sle-ball-count}]{Illustration of the proof of Lemma~\ref{prop-sle-ball-count}. The balls $B_{\delta^{1+v}}(z ; \frk d_h)$ and $B_{2\delta^{1+v}}(z ; \frk d_h)$ are shown in light green and light blue, respectively. We first use the upper bound for the area of $B_{3\delta^{1+v}}(z;\frk d_h)$ to argue that on $G_C$, no curve segment of the form $\eta([(k-1)\delta^2 , k\delta^2])$ can be completely contained in $B_{2\delta^{1+v}}(z;\frk d_h)$. This means that the number of such segments can be bounded above by twice the number of excursions of $\eta$  into a connected component $V$ of $\BB H\setminus B_{\delta^{1+v}}(z;\frk d_h)$ which intersects $\BB H\setminus B_{2\delta^{1+v}}(z;\frk d_h)$ (each such excursion is shown in red). 
Topological considerations imply that each such excursion $\eta([\ul\tau_{V,j} , \ol\tau_{V,j}])$ gives rise to at least one connected component $O$ of $\BB H\setminus (\eta \cup B_{\delta^{1+v}}(z ; \frk d_h))$ which intersects $\BB H\setminus B_{2\delta^{1+v}}(z;\frk d_h)$ (here there are 7 such components). 
To estimate the number of such components $O$, we consider for each $O$ an appropriate $\frk d_{h|_{W^\pm}}$-ball $B_O$ (shown in pink) contained in $O$. The balls $B_O$ are disjoint and each is contained in $B_{2\delta^{1+v}}(z ; \frk d_h) $. By definition of $G_C$, we have a lower bound for the quantum mass of each $B_O$ and an upper bound for the quantum mass of $B_{2\delta^{1+v}}(z ; \frk d_h)$. By comparing these bounds, we obtain an upper bound on the number of sets $B_O$ and hence the number of components $O$. }\label{fig-sle-ball}
\end{center}
\end{figure}

Next we bound the number of segments of $\eta$ with quantum length $\delta^2$ which are hit by a $\frk d_h$-geodesic between two typical points. 

\begin{lem}  \label{prop-sle-intersect-mean}
Let $v\in (0,1)$ and let $u_0 = u_0(v)$ be chosen as in Lemma~\ref{prop-sle-ball-count}. Let $u\in (0,u_0]$, $\rho > 2$, and $C> 1$ and let $G_C = G_C(u , \rho)$ be as in Lemma~\ref{lem-metric-reg}.  Let $z_1,z_2 \in \BB V_\rho$ and let $\gamma_{z_1,z_2}$ be a $\frk d_h$-geodesic from $z_1$ to $z_2$ which is contained in $\BB V_{\rho/2}$, all chosen in a manner which is independent from $\eta$ (viewed as a curve modulo parameterization). For $\delta \in (0,1)$, let $\mcl K_{z_1,z_2}^\delta$ be the set of $k\in\BB N$ for which $\gamma_{z_1,z_2} \cap \eta([(k-1)\delta^2 , k\delta^2]) \not=\emptyset$. There is an exponent $\alpha > 0$, depending only on $\frk w^-,\frk w^+$, and the exponent $\beta$ from Lemma~\ref{lem-metric-reg}, such that for $\delta \in (0,1)$, 
\eqb \label{eqn-sle-intersect-mean}
\BB E \left[\#\mcl K^\delta_{z_1,z_2} \BB 1_{G_C} \,|\, h, \gamma_{z_1,z_2} \right] \preceq   \delta^{-1-2v+ \alpha (1+v)}  \frk d_h(z_1,z_2)
\eqe
where the implicit constant in $\preceq$ is deterministic and depends only on $u, v,C,$ and $\rho$.
\end{lem}
\begin{proof}
By possibly increasing the implicit constant in~\eqref{eqn-sle-intersect-mean}, it suffices to prove the estimate for $\delta \in (0,\delta_0]$, where $\delta_0= \delta_0(C,u,v)$ is as in Lemma~\ref{prop-sle-ball-count}.  
Fix $z_1,z_2$ as in the statement of the lemma. Let $N := \lfloor  \delta^{-1-v}\frk d_h(z_1,z_2) \rfloor + 1$. For $j\in [1,N-1]_{\BB Z}$ let $t_j :=   j \delta^{1+v}$ and let $t_N := \frk d_h(z_1,z_2)$. Also let $V_j := B_{\delta^{1+v}}(\gamma_{z_1,z_2}(t_j) ; \frk d_h)$. Since $\gamma_{z_1,z_2}$ travels one unit of quantum distance in one unit of time, the union of the balls $V_j$ covers $\gamma_{z_1,z_2}$ and the intersection of any four such balls is empty.  Let $\mcl J_{z_1,z_2}^\delta$ be the number of $j\in [1,N]_{\BB Z}$ for which $\eta\cap V_j\not=\emptyset$.  
Lemma~\ref{prop-sle-ball-count} implies that for $\delta \in (0,\delta_0]$, if $G_C$ occurs then for each $j\in [1,N ]_{\BB Z}$ there are at most $ \delta^{-v}$ elements of $\mcl K_{z_1,z_2}^\delta$ for which $\eta([(k-1)\delta^2 , k\delta^2])$ intersects $V_j$.
Hence for any such $\delta$, 
\eqb \label{eqn-sle-intersect-compare}
\#\mcl K_{z_1,z_2}^\delta \leq  3 \delta^{- v} \#  \mcl J_{z_1,z_2}^\delta .
\eqe  
 
In light of~\eqref{eqn-sle-intersect-compare}, it remains only to show that
\eqb  \label{eqn-geodesic-length}
\BB E \left[ \#  \mcl J_{z_1,z_2}^\delta \BB 1_{G_C} \,|\, h, \gamma_{z_1,z_2} \right]   \preceq   \frk d_h(z_1,z_2) \delta^{ -1-v + \alpha (1+v)} 
\eqe 
for appropriate $\alpha>0$ as in the statement of the lemma.
To this end, define the balls $V_j$ as above. By condition~\ref{item-euclidean-holder} in Lemma~\ref{lem-metric-reg} (recall that we have assumed that $\gamma_{z_1,z_2} \subset \BB V_{\rho/2}$), for each $j\in [1,N]_{\BB Z}$, the metric ball $V_j$ is contained in the Euclidean ball $\wt V_j$ of radius $C \delta^{\beta(1+v) }$ and the same center as $V_j$, where $\beta>0$ is as in condition~\ref{item-euclidean-holder}. 

Since $\eta$ is an SLE$_{8/3}(\frk w^--2;\frk w^+-2)$, Lemma~\ref{lem-sle-hit} implies that there is an $\alpha_0 = \alpha_0(\frk w^-,\frk w^+)$ such that
\eqb \label{eqn-eta-hit}
\BB P \left[ \eta \cap B_\ep(w) \not=\emptyset   \right] \preceq \ep^{\alpha_0} , \quad \forall w \in \BB V_{\rho/2}
\eqe
with the implicit constant depending only on $\rho$, $C$, $\frk w^-$, and $\frk w^+$. 
By~\eqref{eqn-eta-hit} and the independence of $(h,\gamma_{z_1,z_2})$ and $\eta$, viewed modulo time parameterization, we find that for each $j\in [1,N]_{\BB Z}$,
\eqbn
\BB P \left[ \eta \cap \wt V_j \not=\emptyset ,\, G_C \,|\, h, \gamma_{z_1,z_2} \right] \preceq   \delta^{\alpha_0 \beta (1+v)} .
\eqen
By summing over all $j\in [1,N]_{\BB Z}$ we obtain~\eqref{eqn-geodesic-length} with $\alpha = \alpha_0 \beta$. 
\end{proof}

\begin{figure}[ht!]
 \begin{center}
\includegraphics{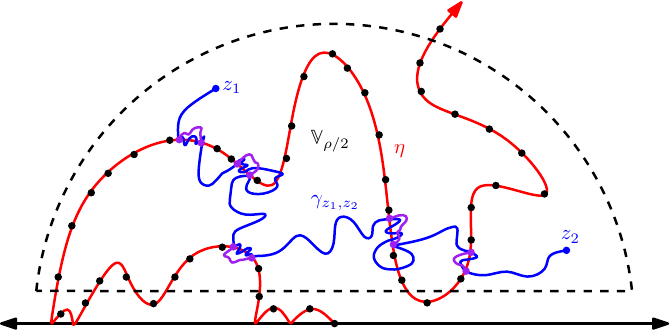} 
\caption[Re-routing procedure used in the proof of Lemma~\ref{prop-gluing-uniform}]{Illustration of the proof of Lemma~\ref{prop-gluing-uniform}. We sample $z_1,z_2$ uniformly from $\mu_h|_{\BB V_R}$ and consider the $\frk d_h$-geodesic $\gamma_{z_1,z_2}$ (blue) from $z_1$ to $z_2$, which is typically contained in $\BB V_{\rho/2}$ for some $\rho > R$. We divide the segment of $\eta$ (red) which is contained in $\BB V_{\rho/2}$ into increments of quantum length $\delta^2$. Lemma~\ref{prop-sle-intersect-mean} implies that on $G_C$, the number of such segments which are hit by $\gamma_{z_1,z_2}$ is at most $\delta^{-1 + \alpha'}$ for a small positive constant $\alpha'$. On the other hand, condition~\ref{item-length-upper} in the definition of $G_C$ (recall Lemma~\ref{lem-metric-reg}) implies that for each $\delta^2$-length segment of $\eta$ which is hit by $\gamma_{z_1,z_2}$, we can find a complementary connected component $U\in \mcl U^-$ and a $\frk d_{h|_U}$-geodesic between the first and last points of this segment which are hit by $\gamma_{z_1,z_2}$ whose length is at most $\delta^{1 + o_\delta(1)}$. These one-sided geodesics are shown in purple. We then concatenate these one-sided geodesics with the segments of $\gamma_{z_1,z_2}$ between the times when it hits $\eta$ to obtain a path from $z_1$ to $z_2$ which crosses $\eta$ only finitely many times and whose length is at most $\frk d_h(z_1,z_2) + o_\delta(1)$. }\label{fig-gluing-uniform}
\end{center}
\end{figure}

The following lemma shows that the metric $\frk d_h$ is equal to the quotient metric in Theorem~\ref{thm-wedge-gluing} at quantum typical points.

\begin{lem} \label{prop-gluing-uniform}
Fix $R> 1$. Let $z_1,z_2$ be sampled uniformly from $\mu_h|_{\BB V_R}$, normalized to be a probability measure. Let $\wt{\frk d}_h$ be the quotient metric on $\BB H$ obtained by gluing the metric spaces $(U , \frk d_{h|_U})$ according to the natural identification. Almost surely, we have 
\eqbn
\frk d_h(z_1,z_2) = \wt{\frk d}_h(z_1,z_2). 
\eqen
\end{lem}
\begin{proof}
See Figure~\ref{fig-gluing-uniform} for an illustration of the proof. 
Fix $p\in (0,1)$. Also let $\alpha$ be as in Lemma~\ref{lem-metric-reg}, let $v \in (0,\alpha/100)$, and let $u \in (0,u_0]$ where $u_0=u_0(v)$ is as in Lemma~\ref{prop-sle-ball-count}. Let $\gamma_{z_1,z_2}$ be the $\frk d_h$-geodesic from $z_1$ to $z_2$, which by Lemma~\ref{prop-geodesic-bdy} is a.s.\ unique and a.s.\ does not intersect $\BB R \cup \{\infty\}$. 
We can find $\rho > 2$ such that with probability at least $1-p/5$,  
\eqbn
\gamma_{z_1,z_2} \subset \BB V_{ \rho/2} \quad \op{and}\quad  \frk d_h(z_1,z_2) \leq \rho .
\eqen
Let $E_\rho$ be the event that this is the case. Also choose $C>0$ such that for this choice of $\rho$ and $u$ as above, the event $G_C = G_C(u,\rho)$ occurs with probability at least $1-p/5$. 

By taking the expectation of the estimate from Lemma~\ref{prop-sle-intersect-mean}, we obtain, in the notation of that lemma,
\eqbn
\BB E \left[\#\mcl K_{z_1,z_2} \BB 1_{G_C \cap E_\rho}  \right] \preceq \delta^{-1-2v+ \alpha (1+v)}  
\eqen
with the implicit constant depending only on $u, v,C,$ and $\rho$. By Markov's inequality, there exists $\delta_0 = \delta_0(u, v,C, \rho) > 0$ such that for $\delta  \in (0,\delta_0]$, it holds with probability at least $1-p/2$ that $E_\rho\cap G_C$ occurs and
\eqb \label{eqn-few-sle-segment}
\#\mcl K_{z_1,z_2}^\delta \leq \delta^{-1 -3v + \alpha (1+v)} \leq \delta^{-1 + \alpha/2} ,
\eqe 
where the second inequality follows from our choice of $v$. 

Now fix $\delta \in (0,\delta_0]$ and assume that $E_\rho \cap G_C$ and the event in~\eqref{eqn-few-sle-segment} occur. We show that $\wt{\frk d}_h(z_1,z_2) \leq \frk d_h(z_1,z_2)$ by constructing a path from $z_1$ to $z_2$ which is a concatenation of finitely many paths which are contained in the closure of a single set $U$ for $U\in \mcl U^- \cup \mcl U^+$ (the reverse inequality is trivial). 

By condition~\ref{item-sle-cont} in Lemma~\ref{lem-metric-reg}, if we take $\delta \leq C^{-1/2}$ then $\eta([(k-1)\delta^2 , k\delta]) \subset \BB V_\rho$ for each $k \in \mcl K_{z_1,z_2}^\delta$. In particular, $\eta([(k-1)\delta^2 , k\delta])$ does not intersect $\BB R$, so $\eta([(k-1)\delta^2 , k\delta]) \subset \bdy U_k$ for some $U_k \in \mcl U^-$ with $U_k \cap \BB V_\rho \not=\emptyset$. 

For $k \in \mcl K_{z_1,z_2}^\delta$, let $\tau_k$ (resp.\ $\sigma_k$) be the first (resp.\ last) time $\gamma_{z_1,z_2}$ hits $\eta([(k -1)\delta^2 , k \delta^2])$. Let $\wt\gamma_k$ be the $\frk d_{h|_{U_k}}$-geodesic from $\gamma_{z_1,z_2}(\tau_k)$ to $\gamma_{z_1,z_2}(\sigma_k)$. 
By condition~\ref{item-length-upper} in Lemma~\ref{lem-metric-reg}, on $G_C$ it is a.s.\ the case that for this choice of $U_k$, we have $\op{diam}   \left( \eta([(k-1)\delta^2 , k\delta^2])  ;  \frk d_{h|_{U_k}} \right) \leq 4 C\delta (|\log \delta| + 1 )^2$, so in particular 
\eqb \label{eqn-one-side-curve}
 \op{len} (\wt\gamma_k ; \frk d_{h|_{U_k}} )  \leq 4 C\delta (|\log \delta| + 1   )^2 , \quad \forall k\in \mcl K_{z_1,z_2}^\delta .
\eqe

Let $k_1$ be the element of $\mcl K_{z_1,z_2}^\delta$ with $\tau_{k_1}$ minimal (i.e.\ $\tau_{k_1}$ is the first time $\gamma_{z_1,z_2}$ hits $\eta$). Inductively, if $j \in [2,\#\mcl K_{z_1,z_2}^\delta]_{\BB Z}$ and $k_{j-1}$ has been defined, let $k \in \mcl K_{z_1,z_2}^\delta$ be chosen so that $\tau_{k_j}$ is the smallest time $\tau_k$ for $k\in \mcl K_{z_1,z_2}^\delta$ which satisfies $\tau_k \geq \sigma_{k_{j-1}}$, if such a $k$ exists; and otherwise let $k_j = \infty$. Let $J$ be the smallest $j\in \BB N$ for which $k_j = \infty$.
Let $\mathring\gamma_1 := \gamma_{z_1,z_2}|_{[0,\tau_{k_1}]}$, $\mathring \gamma_J := \gamma_{z_1,z_2}|_{[\sigma_J , \frk d_h(z_1,z_2)]}$, and for $j\in [2,J-1]_{\BB Z}$ let $\mathring \gamma_j := \gamma_{z_1,z_2}|_{[\sigma_{j-1} , \tau_j]}$. Then each curve $\mathring \gamma_j$ for $j\in [1,J]_{\BB Z}$ does not hit $\eta$ except at its endpoints so is contained in the closure of a single element $\mathring U_j \in \mcl U^- \cup \mcl U^+$, and its $\frk d_{h|_{\mathring U_j}}$-length is the same as its $\frk d_h$-length. 

Let $\wt\gamma$ be the curve obtained by concatenating the curves $\mathring\gamma_1 , \wt\gamma_{k_1} , \dots , \mathring \gamma_{k_{J-1}} , \wt\gamma_j , \mathring \gamma_J$. Then $\wt\gamma$ is a path from $z_1$ to $z_2$ and the quotient metric $\wt{\frk d}_h$ satisfies
\alb
\wt{\frk d}_h(z_1,z_2) 
&\leq \sum_{j=1}^{J-1} \op{len} (\wt\gamma_{k_j}  ; \frk d_{h|_{U_{k_j}}}  ) + \sum_{j=1}^{J } \op{len} (\mathring\gamma_j ; \frk d_{h|_{\mathring U_j}}) \\
&\leq \sum_{j=1}^{J-1} \op{len} (\wt\gamma_{k_j} ; \frk d_{h|_{U_{k_j} }} ) + \frk d_h(z_1,z_2) \\
&\leq  O_\delta\left( \delta^{ \alpha/2}  ( |\log \delta| + 1 )^2 \right) + \frk d_h(z_1,z_2) .
\ale
where the last inequality is by~\eqref{eqn-one-side-curve}. Sending $\delta \rta 0$ shows that $\wt{\frk d}_h(z_1,z_2) \leq \frk d_h(z_1,z_2)$. By Remark~\ref{remark-internal-compare} (c.f.\ the definition of the quotient metric in Section~\ref{sec-metric-basic}) we have $\wt{\frk d}_h(z_1,z_2) \geq \frk d_h(z_1,z_2)$, so in fact $\wt{\frk d}_h(z_1,z_2) = \frk d_h(z_1,z_2)$. Since $p \in (0,1)$ can be made arbitrarily small, we conclude.
\end{proof}

\begin{proof}[Proof of Theorem~\ref{thm-wedge-gluing}]
Lemma~\ref{prop-gluing-uniform} implies that a.s.\ $\frk d_h(z_1,z_2) = \wt{\frk d}_h(z_1,z_2)$ for each pair $(z_1,z_2)$ in a subset of $ \BB H \times \BB H$ which is dense in $\BB H \times \BB H$ with respect to the metric $\frk d_h \times \frk d_h$ and dense in $U\times U$ with respect to $\frk d_{h|_{U}} \times \frk d_{h|_{U}}$ for each $U\in \mcl U^- \cup \mcl U^+$. Hence, the following is true a.s. Suppose given $\ep > 0$ and $w_1,w_2\in\BB H$. Choose $U_1 , U_2 \in \mcl U^- \cup \mcl U^+$ with $w_1\in U_1$ and $w_2 \in U_2$. 
Then we can find $z_1 \in U_1$ and $z_2\in U_2$ such that $\frk d_h(z_1,z_2) = \wt{\frk d}_h(z_1,z_2)$, $\frk d_{h|_{U_1}}(z_1,w_1) \leq \ep$, and $ \frk d_{h|_{U_2}}(z_2,w_2) \leq \ep$. Note that the latter two estimates and the triangle inequality imply that $\frk d_h(z_1,z_2) = \wt{\frk d}_h(z_1,z_2) \leq \frk d_h(w_1,w_2) + 2\ep$. 
By another application of the triangle inequality, we have $\wt{\frk d}_h(w_1,w_2) \leq \wt{\frk d}_h(z_1,z_2) + 2\ep \leq \frk d(w_1,w_2) + 4\ep$, so since $\ep$ is arbitrary $\wt{\frk d}_h(w_1,w_2) \leq  \frk d(w_1,w_2)$ and hence $\wt{\frk d}_h(w_1,w_2) = \frk d_h(w_1,w_2)$.   
\end{proof}

\subsection{Metric gluing in the peanosphere}
\label{sec-general-gluing}

In this subsection we will deduce Theorem~\ref{thm-peanosphere-gluing} and Theorem~\ref{thm-peanosphere-gluing-finite} from Theorem~\ref{thm-wedge-gluing} and Theorem~\ref{thm-cone-gluing}. 

\begin{proof}[Proof of Theorem~\ref{thm-peanosphere-gluing}]
Let $\eta^L$ (resp.\ $\eta^R$) be the left (resp.\ right) boundary of $\eta'((-\infty , 0])$. Note that $\mcl U^- \cup \mcl U^+$ is precisely the set of connected components of $\BB C\setminus (\eta^L\cup \eta^R)$. 

By~\cite[Footnote 4]{wedges} and~\cite[Theorem~1.1]{ig4}, $\eta^L$ has the law of a chordal SLE$_{8/3}(-2/3)$ from $0$ to $\infty$ in $\BB H$. 
By~\cite[Theorem~1.5]{wedges}, the surface $(\BB C \setminus \eta^L , h|_{\BB C\setminus \eta^L} , 0 ,\infty)$ has the law of a weight-$4/3$ quantum wedge.
By Theorem~\ref{thm-cone-gluing}, it is a.s.\ the case that $(\BB C , \frk d_h)$ is the metric space quotient of $(\BB C\setminus \eta^L , \frk d_{h|_{\BB C\setminus \eta^L} })$ under the natural identification. 

By~\cite[Theorem~1.11]{ig4}, the conditional law of $\eta^R$ given $\eta^L$ is that of a chordal SLE$_{8/3}(-4/3 ; -4/3)$ from $0$ to $\infty$ in $\BB C\setminus \eta^L$. By Theorem~\ref{thm-wedge-gluing}, a.s.\ $(\BB C\setminus \eta^L , \frk d_{h|_{\BB C\setminus \eta^L} })$ is the metric quotient of the disjoint union of the metric spaces $(\ol U , \frk d_{h|_U})$ for $U\in\mcl U^-\cup \mcl U^+$ under the natural identification, except that we don't identify points on $\eta^L$. The theorem statement follows by combining this with the previous paragraph. 
\end{proof}

Next we will deduce Theorem~\ref{thm-peanosphere-gluing-finite} from Theorem~\ref{thm-peanosphere-gluing} and an absolute continuity argument.

\begin{proof}[Proof of Theorem~\ref{thm-peanosphere-gluing-finite}]
Since $\eta'$ is parameterized by $\mu_h$ and a.s.\ hits $\mu_h$-a.e.\ point of $\BB C$ exactly once, it follows that $\eta'(\BB t)$ is independent from $\eta'$ and its conditional law given $h$ is $\mu_h$. Hence we can assume that $(\BB C , h , \infty)$ is embedded into $\BB C$ in such a way that $\eta'(\BB t) = 0$, in which case $(\BB C , h , 0, \infty)$ is a doubly marked quantum sphere. We also assume that our embedding is such that $\mu_h(\BB D) = 1/2$. 

For $R> 1$, let $A_R$ be the closed annulus $\ol{B_R(0)} \setminus B_{1/R}(0)$. Let $\mcl U_R$ be the set of $U\in\mcl U^-\cup\mcl U^+$ such that $U\subset A_R$. We observe that our choice of embedding implies that $U$ is determined by $\eta'$, viewed as a curve modulo monotone re-parameterization. Let $W_R := \bigcup_{U\in \mcl U_R} U $. Almost surely, $0$ is not contained in $\bdy U$ for any $U\in\mcl U^-\cup \mcl U^+$ and every element of $\mcl U^-\cup\mcl U^+$ is bounded. It follows that 
\eqb \label{eqn-annulus-bubble-contain}
\text{for each $R > 1$, there exists $R'>1$ such that $A_R\subset W_{R'}$ .}
\eqe
 
Let $(\BB C  ,\wh h , 0, \infty)$ be a $\sqrt{8/3}$-quantum cone independent from $\eta'$ with the embedding chosen so that $\mu_h(\BB D) = 1/2$ and let $\wh\eta'$ be the curve obtained by parameterizing $\wh\eta'$ by $\mu_{\wh h}$-mass in such a way that $\wh\eta'(0) = 0$. Then for each $R>1$, the laws of $h|_{A_R}$ and $\wh h|_{A_R}$ are mutually absolutely continuous, so also the laws of $(h|_{A_R} , \mcl U_R)$ and $(\wh h|_{A_R} , \mcl U_R)$ are mutually absolutely continuous. 

By Theorem~\ref{thm-peanosphere-gluing} and the definition of the quotient metric, it is a.s.\ the case that for each $\ep > 0$ and each $z,w\in \BB C$, there exists $N  \in \BB N$, $U_{1},\dots , U_N \in \mcl U^-\cup \mcl U^+$, and $\frk d_{\wh h|_{U_{i}}}$-geodesics $\wh \gamma_i$ for $i\in [1,N]_{\BB Z}$ whose concatenation is a path from $z$ to $w$ such that
\eqbn
\sum_{i=1}^N  \op{len} \left(\wh\gamma_i ;  \frk d_{\wh h|_{U_{i}}} \right) \leq \frk d_{\wh h}(z,w) + \ep .
\eqen 
Note that if $z\in W_R$ and $\frk d_{\wh h}(z,w) \leq \frac14 \frk d_{\wh h}(z, \bdy W_R)$, then for small enough $\ep$ each of the sets $U_i$ for $i\in [1,N]_{\BB Z}$ must belong to $\mcl U_R$. 

For $R>1$ and $z\in W_R$, let $\rho_R(z) := \frk d_h(z , \bdy W_R)$. Since metric balls are locally determined by the field, we find that for each $z\in A_R$, both $\rho_R(z)$ and the restriction of $\frk d_h$ to $B_{\rho_R(z)} (z ;\frk d_h)$ are a.s.\ determined by $(h|_{A_R} , \mcl U_R)$. Consequently, it follows from the preceding paragraph and absolute continuity that a.s.\ for each $z,w \in W_R$ with $\frk d_h(z,w) \leq \frac14 \rho_R(z)$ and each $\ep > 0$, there exists $N  \in \BB N$, $U_{1},\dots , U_N \in \mcl U_R$, and $\frk d_{  h|_{U_{i}}}$-geodesics $\gamma_i$ for $i\in [1,N]_{\BB Z}$ whose concatenation is a path from $z$ to $w$ such that
\eqb  \label{eqn-sphere-path-approx}
\sum_{i=1}^N  \op{len} \left( \gamma_i ;  \frk d_{ h|_{U_{i}}} \right) \leq \frk d_{  h}(z,w) + \ep .
\eqe 

Now let $z_1,z_2$ be sampled uniformly from $\mu_h|_{A_R}$, normalized to be a probability measure and let $\gamma$ be the $\frk d_h$-geodesic from $z_1$ to $z_2$. By Lemma~\ref{prop-geodesic-plane}, $\gamma$ is a.s.\ unique and a.s.\ does not hit $0$ or $\infty$. Hence~\eqref{eqn-annulus-bubble-contain} implies that for each $p \in (0,1)$ there exists $R' >R$ such that 
$\BB P  \left[ \gamma \subset W_{R'} \right] \geq 1-p $.

On the event $\{ \gamma \subset W_{R'}\}$, there a.s.\ exists $M\in\BB N$ and times $0  = t_0 <t_1 < \dots < t_M = \frk d_{ h}(z_1,z_2)$ such that $\frk d_{ h}(   \gamma(t_i), \bdy W_{R'}) \geq 4(t_i - t_{i-1})$ for each $i\in[1,M]_{\BB Z}$. By applying~\eqref{eqn-sphere-path-approx} for each $i\in [1,M]_{\BB Z}$ with $(z,w) = (\gamma(t_{i-1}) , \gamma(t_i))$ and $\ep/M$ in place of $\ep$, we a.s.\ obtain a finitely many paths, each of which is a $\frk d_{h|_U}$-geodesic for some $U\in \mcl U_{R'}$, whose concatenation is a path from $z_1$ to $z_2$ and whose total length is at most $\frk d_h(z_1,z_2) + \ep$. Since $R>1$, $p\in (0,1)$, and $\ep>0$ are arbitrary and by~\eqref{eqn-annulus-bubble-contain}, we infer that if $\wt{\frk d}_h$ is the quotient metric on $\BB C$ obtained by identifying the metric spaces $(U , \frk d_{h|_U})$ for $U\in\mcl U^-\cup\mcl U^+$ as in the statement of the lemma, then a.s.\ $\frk d_h(z_1,z_2) = \wt{\frk d}_h(z_1,z_2)$. By the same argument used to conclude the proof of Theorem~\ref{thm-wedge-gluing}, we obtain the desired result. 
\end{proof}

\appendix

\section{Estimate for quantum diameters}
\label{sec-quantum-diam}

In this appendix we prove an estimate to the effect that the quantum diameters of certain subsets of $\BB H$ with respect to the restriction of the field corresponding to a $\sqrt{8/3}$-quantum wedge have a polynomial tail (similar estimates for other GFF-type distributions can be obtained using local absolute continuity). This estimate is only used in the proof of Lemma~\ref{prop-metric-bdy-wedge}.

\begin{lem} \label{lem-length-ball-diam}
Let $(\BB H , h , 0, \infty)$ be a $\sqrt{8/3}$-quantum wedge normalized so that $\nu_h([0,1]) = 1$ and let $U,V\subset \BB H$ be connected open sets such that $U$ is bounded, $\ol U \cap\BB H \subset V$, $1 \in \bdy V$, and $0\notin \ol U$. There is a universal constant $\wt\beta > 0$ such that for each $C>0$, 
\eqb \label{eqn-length-ball-diam}
\BB P\left[ \op{diam}\left( U  ; \frk d_{h|_{V}} \right)  \leq C \right] = 1 - O_C(C^{-\wt\beta}).
\eqe
(The dependence on the choice of $U,V$ is in the implicit constant in the $ O_C(C^{-\wt\beta})$ term.)
\end{lem}
\begin{proof}
The idea of the proof is to find a random subset $\mcl B$ of $\BB H$ such that the quantum surface parameterized by $\mcl B$ has the law of a quantum disk and $U\subset\mcl B\subset V$ with uniformly positive conditional probability given $h$. We can then use basic diameter estimates for the quantum disk which come from its definition in terms of the Brownian snake (Section~\ref{sec-brownian-disk}). 
To find such a subset, we will use results for SLE on the $\sqrt{8/3}$-quantum wedge from~\cite{wedges}. 
See Figure~\ref{fig-diam-estimate} for an illustration. 

To this end, let $\eta'$ be a chordal SLE$_6$ from $0$ to $\infty$ sampled independently from $h$ and let $\mcl B$ be the connected component of $\BB H\setminus \eta'$ with 1 on its boundary. 
By~\cite[Theorem 1.18]{wedges}, if we condition on $\nu_h(\bdy\mcl B)$ then the conditional law of the quantum surface $(\mcl B , h|_{\mcl B})$ is that of a quantum disk with specified boundary length and random area. Equivalently, by~\cite[Corollary 1.4]{lqg-tbm2}, the metric space $(\mcl B , \frk d_{h|_{\mcl B}})$ is a Brownian disk with boundary length $\nu_h(\bdy\mcl B)$ and area $(\nu_h(\bdy \mcl B))^2 X$ where $X$ is sampled from the distribution $(2\pi)^{-1/2} a^{-5/2} e^{-1/(2a)} \,da$ independently of $\nu_h(\bdy \mcl B)$. Using the definition of the Brownian disk with specified area and boundary length given in Section~\ref{sec-brownian-disk} together with the tail estimate for the maximum of the Brownian snake~\cite[Proposition~14]{serlet-snake} and the Gaussian tail bound for the maximum of a Brownian bridge, it is easily seen that there is a universal constant $\wt\beta_0 > 0$ such that for each $C>0$,  
\eqb \label{eqn-diam-given-area}
\BB P\left[ \op{diam}\left( \mcl B ; \frk d_{h|_{\mcl B}} \right) \leq C \nu_h(\bdy\mcl B)^{1/2} \,|\, \nu_h(\bdy\mcl B) \right] = 1 - O_C(C^{-\wt\beta_0 }) .
\eqe 
Here, to apply~\cite[Proposition~14]{serlet-snake}, we decompose the process $X$ into excursions above its record minimum, as in the proof of Lemma~\ref{prop-disk-bdy-holder}.

To get an estimate for the unconditional internal diameter of $\mcl B$, we need to estimate $\nu_h(\bdy\mcl B)$. We do this using the L\'evy process description of the left/right boundary length process for $\eta'$ from~\cite{wedges}. 
If we let $R_t$ for $t\geq 0$ be the $\sqrt{8/3}$-LQG length of the right outer boundary of $\eta'([0,t])$ minus the $\sqrt{8/3}$-LQG length of the interval to the right of 0 which is disconnected from $\infty$ by $\eta'([0,t])$, then the time at which $\eta'$ disconnects $\mcl B$ from $\infty$ is the same as the first time $t$ for which $R_t \leq -1$. If we let $T$ be this time, then $\nu_h(\bdy\mcl B) = R_{T^-} - R_T$.
By~\cite[Corollary 1.19]{wedges}, if $\eta'$ is parameterized according to so-called \emph{quantum natural time}, then $R_t$ evolves as a $3/2$-stable process with only downward jumps. 
Using this, the law of $R_{T^-} -R_T$ can be computed explicitly using a generalization of the arcsine law (see, e.g.,~\cite[Example~7]{dk-overshoot}).
In particular, $\BB P[R_{T^-} -R_T > A ]$ decays like a positive power of $A^{-1}$. Combining this with~\eqref{eqn-diam-given-area}, we find that there is a $\wt\beta  > 0$ such that
\eqb \label{eqn-diam-marginal}
\BB P\left[ \op{diam}\left( \mcl B ; \frk d_{h|_{\mcl B}} \right) \leq C    \right] = 1 - O_C(C^{-\wt\beta}) .
\eqe 

By replacing $U$ with its intersections with each of the left and right quarter planes in $\BB H$, we can assume without loss of generality that $\ol U$ does not disconnect 0 from $\infty$ in $\BB H$. 
Since $\eta'$ is independent from $h$ and $\eta'$ can be made to stay arbitrarily close to a deterministic curve with positive probability~\cite[Lemma~2.5]{miller-wu-dim}, 
under this assumption on $U$ we can find a constant $p > 0$ (depending only on $U$ and $V$) such that a.s.\ 
\eqb \label{eqn-bubble-pos-prob}
\BB P\left[ U  \subset \mcl B \subset V \,|\, h \right] \geq p .
\eqe 
If $ U  \subset \mcl B \subset V $, then (using Lemma~\ref{prop-internal-equal}) we have $\op{diam}\left( U  ; \frk d_{h|_{ V}} \right) \leq \op{diam}\left( \mcl B ; \frk d_{h|_{\mcl B}} \right)$. 
Hence, combining~\eqref{eqn-diam-marginal} and~\eqref{eqn-bubble-pos-prob} yields~\eqref{eqn-length-ball-diam}.
\end{proof}

\begin{figure}[t!]
 \begin{center}
\includegraphics[scale=.85]{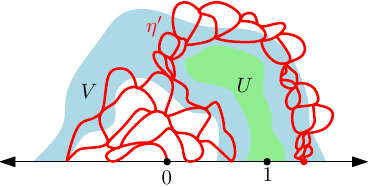}
\vspace{-0.01\textheight}
\caption{Illustration of the proof of Lemma~\ref{lem-length-ball-diam}. The connected component $\mcl B$ of $\BB H\setminus \eta'$ with $1$ on its boundary is a quantum disk, equivalently a Brownian disk, which allows us to estimate its diameter. On the other hand, and this component contains~$U$ and is contained in $V$ with positive conditional probability given~$h$. }\label{fig-diam-estimate}
\end{center}
\vspace{-1em}
\end{figure} 

The following consequence of Lemma~\ref{lem-length-ball-diam} is used in the proof of Lemma~\ref{prop-metric-bdy-wedge}.

\begin{lem} \label{lem-circle-avg-ball-diam}
Let $(\BB H , h , 0, \infty)$ be the circle average embedding of a $\sqrt{8/3}$-quantum wedge (so that $h$ is the random distribution from~\cite[Definition~4.5]{wedges}). Also fix $0 < a' < a < 1   < b < b'$ and define the semi-annuli $U := \BB H\cap (\ol{B_b(0) \setminus B_a(0)})$ and $V := \BB H\cap (\ol{B_{b'}(0) \setminus B_{a'}(0)})$.  There is a universal constant $\beta  > 0$ such that for each $C>0$, 
\eqb \label{eqn-circle-avg-ball-diam}
\BB P\left[ \op{diam}\left( U  ; \frk d_{h|_{V}} \right)  \leq C \right] = 1 - O_C(C^{-\beta}) .
\eqe
(The dependence on the choice of $a,b,a',b'$ is in the implicit constant in the $O_C(C^{-\beta})$ term.)
\end{lem}
\begin{proof}
We will extract the lemma from Lemma~\ref{lem-length-ball-diam} and a union bound over possible approximate values of $\nu_h([0,1])$ (which is random with our present choice of normalization for $h$). 
Fix $  \wt a', \wt a , \wt b , \wt b' > 0$ with $a' < \wt a' < \wt b' < b'$ and $a < \wt a < \wt b < a$. 
For $s > 0$, let $r_s > 0$ be chosen so that $\nu_h([0,r_s]) = s$ and let $U_s := \BB H\cap (\ol{B_{\wt b r_s}(0) \setminus B_{\wt a r_s}(0)})$ and $V := \BB H\cap (\ol{B_{\wt b' r_s}(0) \setminus B_{\wt a' r_s}(0)})$.

Let $\wt\beta > 0$ be as in Lemma~\ref{lem-length-ball-diam}. By Lemma~\ref{lem-length-ball-diam} together with the invariance of the law of the quantum wedge under scaling boundary lengths by $s > 0$, areas by $s^2$, and distances by $s^{1/2}$ (\cite[Proposition 4.7(i)]{wedges} and~\cite[Lemma 2.2]{lqg-tbm2}), we see that for each $s > 0$, 
\eqb \label{eqn-diam-s}
\BB P\left[ \op{diam}\left( U_s  ; \frk d_{h|_{V_s}} \right)  \leq C s^{1/2} \right] = 1 - O_C(C^{-\wt\beta}) .
\eqe 

We now need to transfer this bound from the pair $(U_s,V_s)$ to the pair $(U,V)$ in the statement of the lemma.
By our choice of $ \wt a', \wt a , \wt b , \wt b' $, there is an $\ep > 0$ (depending only on the $a$'s and $b$'s) such that if $r_s \in [1-\ep ,1+\ep]$, then $U_s \subset U$ and $V\subset V_s$. Hence
\eqb \label{eqn-diam-s-compare}
 r_s \in [1-\ep ,1+\ep] \quad \Rightarrow \quad \op{diam}\left( U  ; \frk d_{h|_{V}} \right) \leq  \op{diam}\left( U_s  ; \frk d_{h|_{V_s}} \right) .
\eqe
Since $\nu_h([1-\ep,1])$ and $\nu_h([1,1+\ep])$ have finite moments of all negative orders~\cite[Theorem 2.12]{rhodes-vargas-review}, for each $\delta >0$ it holds except on an event of probability decaying faster than any power of $C^{-1}$ that $\nu_h([1-\ep,1])$ and $\nu_h([1,1+\ep])$ are each at least $C^{-\delta}$. If this is the case, then the relations in~\eqref{eqn-diam-s-compare} both hold provided $|s - \nu_h([0,1])| \leq C^{-\delta}$. On the other hand, by standard moment estimates for the LQG measure (see~\cite[Theorems 2.11 and 2.12]{rhodes-vargas-review} plus the boundary measure analog of~\cite[Lemma A.3]{ghs-dist-exponent} to deal with the log singularity at 0) we find that except on an event of probability decaying faster than some positive power of $C^{-1}$, we have $\nu_h([0,1]) \in [C^{-\wt\beta/4} , C^{\wt\beta/4}]$. We now conclude by choosing $\delta < \wt\beta/4$, applying~\eqref{eqn-diam-s} to $O_C(C^{\wt\beta/4 + \delta})$ values of $s \in [C^{-\wt\beta/4} , C^{\wt\beta/4}]$, and taking a union bound.
\end{proof}

\section{Upper bound for the probability that $\mathrm{SLE}_\kappa(\rho)$ hits a point}
 
We will prove the following rough estimate, which is needed in the special case when $\kappa=8/3$ and $\ul\rho= (\frk w^--2, \frk w^+-2)$ in the proof of Lemma~\ref{prop-sle-intersect-mean}. We prove the lemma for general $\kappa$ and $\ul\rho$ since the proof is the same. 

\begin{lem} \label{lem-sle-hit}
Let $\kappa \in (0,8)$ and let $\ul\rho = (\rho_1,\dots,\rho_k)$ be a vector of weights. Also let $x,y\in \bdy \BB D$ and let $\eta$ be a chordal SLE$_\kappa(\ul\rho)$ from $x$ to $y$ in $\BB D$ with force points locations $z_1,\dots,z_k \in \bdy\BB D$. Assume that $\eta$ a.s.\ does not hit the continuation threshold, i.e., the sum of the force points which have been hit or disconnected from $y$ by $\eta$ up to any fixed time is $>-2$ (this means that $\eta$ is a.s.\ defined for all time~\cite[Theorem 2.2]{ig1}).  There exists $\alpha_0 = \alpha_0(\kappa,\ul\rho) > 0$ such that 
\eqb \label{eqn-sle-hit}
\BB P\left[ \eta \cap B_\ep(w) \not=\emptyset \right] \preceq \ep^{\alpha_0} ,\quad \forall \ep \in (0,1) ,\quad \forall w\in \BB D .
\eqe
Here the implicit constant depends on $\kappa , \ul\rho,$ and $w$ (but not on $x,y,z_1,\dots,z_n$) and is uniform for $w$ in compact subsets of $\BB D$. 
\end{lem} 

Lemma~\ref{lem-sle-hit} in the case when $\ul\rho = 0$, with exponent $\alpha_0 = 1-\kappa/8$, follows from~\cite[Proposition 4]{beffara-dim}. We expect that $\alpha_0 = 1-\kappa/8$ for general values of $\ul\rho$ as well, and that this can be proven, e.g., by adapting the arguments of~\cite[Section 3]{light-cone-dim}. However, we do not need this stronger estimate here so for the sake of brevity we do not derive it.

\begin{proof}[Proof of Lemma~\ref{lem-sle-hit}]
Let $r_0 := 1/100$. We first argue that there is a $p = p(\kappa,\ul\rho) > 0$, \emph{not} depending on $x,y,z_1,\dots,z_n$, such that 
\eqb  \label{eqn-one-pt-hit}
\BB P\left[ \eta \cap B_{r_0}(0) = \emptyset \right] \geq p  .
\eqe 
Indeed, by the Schramm-Wilson coordinate change formula~\cite[Theorem 3]{sw-coord}, if we let $\wt \rho$ be equal to $\kappa-6$ minus the sum of the coordinates of $\ul\rho$, then $\eta$ agrees in law with a radial SLE$_\kappa(\ul\rho  , \wt\rho)$ from $x$ to $0$ in $\BB D$, with force points at $z_1,\dots,z_n$ plus an extra force point of weight $\wt\rho$ at $y$, stopped at the (a.s.\ finite) time when it hits $f(y)$. By~\cite[Lemma~2.5]{miller-wu-dim}, such a process has positive probability to avoid $B_{r_0}(0)$ for any fixed choice of $x,y,z_1,\dots,z_n$. It is shown in~\cite[Section 2.2]{ig1} that the law of the driving process of SLE$_\kappa(\ul\rho)$ process depends continuously on the force point locations. By combining this with the radial analog of~\cite[Proposition 4.43]{lawler-book}, we find that the probability of avoiding $B_{r_0}(0)$ depends continuously on the force point locations. Since the space of possible choices of $x,y,z_1,\dots,z_n$ is compact, the preceding probability can be taken to be uniform over all possible choices of $x,y,z_1,\dots,z_n$. Thus~\eqref{eqn-one-pt-hit} holds.

We now iteratively apply~\eqref{eqn-one-pt-hit} to prove the statement of the lemma in the case $w = 0$. Let $\tau$ be the first time $\eta$ hits $B_{r_0}(0)$, so that by~\eqref{eqn-one-pt-hit} we have $\BB P[\tau<\infty] \leq 1 - p$. On the event $\{\tau <\infty\}$, let $f : \BB D\setminus \eta([0,\tau]) \rta \BB D$ be the conformal map which fixes $0$ and takes $\eta(\tau)$ to $x$. On the event $\{\tau<\infty\}$, the conditional law of $f(\eta|_{[\tau,\infty)})$ given $f(\eta|_{[0,\tau]})$ is that of an SLE$_\kappa(\ul\rho)$ in $\BB D$ started from $x$ with some choice of target point and force point locations. By~\eqref{eqn-one-pt-hit}, the conditional probability that $f(\eta|_{[0,\tau]})$ hits $B_{r_0}(0)$ is at most $1-p$. On the other hand, by standard distortion estimates we have $f (B_{r_0^2/4}(0)) \subset  B_{r_0}(0)$. Therefore, the probability that $\eta$ hits $B_{r_0^2/4}(0)$ is at most $(1-p)^2$. Iterating this $n$ times, we find that 
\eqbn
\BB P\left[ \eta\cap B_{r_0^{n }/4^{n-1}}(0) \not=\emptyset \right] \leq (1-p)^n  , \quad \forall n\in\BB N
\eqen
which implies~\eqref{eqn-sle-hit} for an appropriate choice of $\alpha_0$ in the case when $w=0$. 

To treat the case when $w\not=0$, let $\phi_w : \BB D\rta\BB D$ be the conformal map which fixes $x$ and takes $w$ to 0 and choose $c > 0$ so that $\phi_w(B_\ep(w)) \subset B_{c\ep}(0)$ for each $\ep \in(0,c)$. We can take $c$ to be uniform over all choices of $w$ in any fixed compact subset of $\BB D$. The curve $\phi_w(\eta)$ is an SLE$_\kappa(\ul\rho)$ with a possibly different choice of target point and force point locations. The bound~\eqref{eqn-sle-hit} for $w\not=0$ therefore follows from~\eqref{eqn-sle-hit} for $w=0$. 
\end{proof}

\bibliography{cibiblong,cibib}
\bibliographystyle{hmralphaabbrv}

\end{document}